\documentclass[preprint,10pt,numbers,sort&compress]{elsarticle}

\usepackage{soul}
\usepackage{psfrag}
\usepackage{a4wide}
\usepackage{rotating}
\usepackage{color}
\usepackage{amssymb}
\usepackage{amsmath}
\usepackage{amsthm}
\usepackage{verbatim}
\usepackage{cancel}
\usepackage{nicematrix}
\usepackage{multirow}

\newcommand{\f}[1]{\mathbf{#1}}
\newcommand{\ab}[1]{\boldsymbol{#1}}
\def\bfm#1{\boldsymbol{#1}}
\newcommand{\bb}[1]{\bfm{#1}}

\newcommand{\R}{\mathbb R}

\newcommand{\W}{\mathcal{W}}
\newcommand{\glob}{y}
\newcommand{\loc}{\zeta}
\newcommand{\LL}{i_0}
\newcommand{\RR}{i_1}
\newcommand{\g}{f}
\newcommand{\gC}{f}
\newcommand{\sm}{s}
\newcommand{\sS}{s}
\newcommand{\ot}{\theta}
\newcommand{\Side}{\tau}
\newcommand{\dd}{\partial}

\DeclareMathOperator{\Span}{span}

\newtheorem{prop}{Proposition}

\theoremstyle{definition}
\newtheorem{ex}{Example}

\newproof{pf}{proof}
\advance\textheight by 0.4cm
\advance\topmargin by -0.2cm

\definecolor{gold}{rgb}{1,0.7,0}
\definecolor{dred}{rgb}{0.92,0,0}
\definecolor{dgreen}{rgb}{0,0.6,0}

\begin{document}

\begin{frontmatter}

\title{Isogeometric collocation with smooth mixed degree splines over planar multi-patch domains}

\cortext[cor]{Corresponding author}

\author[vil]{Mario Kapl}
\ead{m.kapl@fh-kaernten.at}

\author[slo2]{Alja\v z Kosma\v c\corref{cor}}
\ead{aljaz.kosmac@iam.upr.si}

\author[slo1,slo2]{Vito Vitrih}
\ead{vito.vitrih@upr.si}

\address[vil]{Department of Engineering $\&$ IT, Carinthia University of Applied Sciences, Villach, Austria}


\address[slo1]{Faculty of Mathematics, 
Natural Sciences and Information Technologies,\\
University of Primorska, Glagolja\v{s}ka 8, Koper, Slovenia}
\address[slo2]{Andrej Maru\v{s}i\v{c} Institute,
University of Primorska, Muzejski trg 2, Koper, Slovenia}

\begin{abstract}

We present a novel isogeometric collocation method for solving the Poisson's and the biharmonic equation over planar bilinearly parameterized multi-patch geometries. The proposed approach relies on the use of a modified construction of the $C^s$-smooth mixed degree isogeometric spline space~\cite{KaKoVi24b} for $s=2$ and $s=4$ in case of the Poisson's and the biharmonic equation, respectively. The adapted spline space possesses the minimal possible degree~$p=s+1$ everywhere on the multi-patch domain except in a small neighborhood of the inner edges and of the vertices of patch valency greater than one where a degree $p=2s+1$ is required. 
This allows to solve the PDEs with a much lower number of degrees of freedom compared to employing the $C^s$-smooth spline space~\cite{KaVi20b} with the same high degree~$p=2s+1$ everywhere. To perform isogeometric collocation with the smooth mixed degree spline functions, we introduce and study two different sets of collocation points, namely first a generalization of the standard Greville points to the set of mixed degree Greville points and second the so-called mixed degree superconvergent points. 
The collocation method is further extended to the class of bilinear-like $G^s$ multi-patch parameterizations~\cite{KaVi17c}, which enables the modeling of multi-patch domains with curved boundaries, and is finally tested on the basis of several numerical examples. 
\end{abstract}

\end{frontmatter}

\section{Introduction}

Isogeometric Analysis~\cite{ANU:9260759, CottrellBook, HuCoBa04} is a numerical approach for solving a partial differential equation (PDE) by performing the numerical simulation with the same B-spline or NURBS functions as used for the CAD model to represent the corresponding computational domain. The CAD model, in particular in case of a complex domain, is often a multi-patch spline geometry~\cite{HoLa93, Fa97}, which is an unstructured mesh consisting of quadrilateral spline patches and can have extraordinary vertices, i.e.~vertices with valencies different from four. To solve a high-order PDE over such a multi-patch spline geometry, a smooth spline space is needed to describe the numerical solution of the PDE.  

When solving a PDE via the Galerkin approach, one replaces the strong form (i.e. the original form) of the PDE by its weak form. This allows the use of spline spaces with lower smoothness compared to the solving of the strong form but the costly matrix assembly as well as possible consistency and robustness issues of the involved numerical integration can be inspiring to use some alternative methods. One alternative is to solve the given strong form of the PDE via the collocation method, which requires no numerical integration and consequently the matrix assembly is considerably faster. However, the collocation method needs spline spaces of higher smoothness (and therefore of higher degrees) compared to the Galerkin method, namely $C^2$-smooth functions for second order PDEs such as the Poisson's equation, see e.g.~\cite{KaVi20, SuperConvergent2015,IsoCollocMethods2010, FahrenderLorenzisGomez2018, GomezLorenzisVariationalCollocation, MonSanTam2017, CostComparison2013}, and $C^4$-smooth functions for fourth order PDEs such as the biharmonic equation and the Kirchhoff plate problem, see e.g. \cite{KaKoVi24, GomezLorenzisVariationalCollocation, Maurin2018, RealiGomez2015, TsplinesIgC2016}.

The study and construction of exactly $C^s$-smooth ($\sm \geq 1$) isogeometric spline spaces over multi-patch domains has been originated for the case $\sm=1$. The existing methods can be classified according to the type of multi-patch spline geometry used, namely $C^1$-smooth multi-patch parameterizations with singularities in the vicincity of the extraordinary vertices, see e.g.~\cite{NgPe16, ToSpHu17, WeLiQiHuZhCa22}, $C^1$-smooth parameterizations with $G^1$-smooth caps in the neighborhood of the extraordinary vertices, see e.g.~\cite{KaPe17, KaPe18, NgKaPe15, WeFaLiWeCa23}, as well as in general just $G^1$-smooth, regular multi-patch parameterizations, see e.g.~\cite{BlMoXu20, ChAnRa18, ChAnRa19, mourrain2015geometrically}, which are often based on the use of particular geometries such as bilinear (e.g.~\cite{BeMa14, KaSaTa21}) or AS-$G^1$ multi-patch parameterizations (e.g.~\cite{CoSaTa16, FaKaKoVi24, FaJuKaTa22, KaSaTa19a}). The extension to higher smoothness $\sm \geq 2$, which is needed for performing isogeometric collocation, has been mainly studied for the third approach and has been firstly considered for the case $\sm=2$ (e.g. \cite{KaVi17a,KaVi17b, KaVi17c, KaVi19a}), which has then been further generalized to the case of an arbitrary smoothness~$\sm \geq 1$ in \cite{KaVi20b}.

The construction of the $C^s$-smooth isogeometric multi-patch spline space~\cite{KaVi20b}
works for planar bilinearly parameterized multi-patch domains as well as for the richer class of bilinear-like $G^{\sS}$ multi-patch parameterizations, cf.~\cite{KaVi17c, KaVi20b}, which enables the use of multi-patch domains with curved boundaries, too. A benefit of the $C^s$-smooth multi-patch spline space~\cite{KaVi20b} are its optimal approximation properties as numerically shown by performing $L^2$ approximation for $1 \leq \sS \leq 4$ in~\cite{KaVi20b}, and by solving the Poisson's and biharmonic equation via isogeometric collocation for $s=2$ and $s=4$ in~\cite{KaVi20} and \cite{KaKoVi24}, respectively. However, a drawback of the $C^{\sS}$-smooth spline space is that a quite high spline degree $p$ is needed, namely $p \geq 2\sS+1$, which hence leads to a high number of degrees of freedom involved. 

In order to overcome the latter issue, a $C^s$-smooth mixed degree and regularity isogeometric spline space over planar bilinear and bilinear-like $G^s$ multi-patch geometries has been developed in \cite{KaKoVi24b} and has been used to solve high order PDEs like the biharmonic and triharmonic equation via the Galerkin approach. The proposed $C^s$-smooth mixed degree and regularity spline space requires the high degree~$p=2s+1$ just in a small neighborhood around the edges and vertices of the multi-patch domain and can have the smallest possible degree~$p=s+1$ in the interior of the single patches, which allows to solve the PDEs with a much lower number of degrees of freedom compared to the spline space~\cite{KaVi20b}. 

In this paper, we aim to  perform also the multi-patch isogeometric collocation with a reduced number of degrees of freedom compared to the spline space~\cite{KaVi20b}, which has been used in \cite{KaVi20} and \cite{KaKoVi24} for solving the Poisson's and biharmonic equation, respectively. For this purpose, we will develop a collocation method employing a $C^s$-smooth mixed degree spline space which will be a modified version of the $C^s$-smooth mixed degree and regularity isogeometric spline space \cite{KaKoVi24b} and which will have the high spline degree $p=2s+1$ just in a small neighborhood around the inner edges and the vertices of patch valency greater than one. Therefore, the adapted $C^s$-smooth mixed degree spline space will even further reduce the places with the required high degree of~$p=2s+1$ compared to the $C^s$-smooth mixed degree and regularity spline space~\cite{KaKoVi24b}, where the high degree is needed in the vicinity of all inner and boundary edges and of all inner and boundary vertices. Thereby, the construction of the proposed mixed degree spline space will be based on the use of an appropriate mixed degree underlying spline space over the unit square $[0,1]^2$ to define the isogeometric functions on the single patches. 

For performing isogeometric collocation with the generated $C^s$-smooth mixed degree isogeometric spline space two different sets of collocation points, namely a generalization of the standard Greville points to the mixed degree Greville points as well as the so-called mixed degree superconvergent points, will be considered and studied. Several numerical examples will demonstrate the potential of the developed collocation method for solving the Poisson's and the biharmonic equation over multi-patch domains with $C^2$ and $C^4$-smooth functions of mostly low degree and hence as desired with a much lower number of degrees of freedom compared to employing the $C^s$-smooth spline space~\cite{KaVi20b} with the same high degree~$p=2s+1$ everywhere.  
In case of multi-patch domains as studied in this work, the use of the mixed degree Greville and mixed degree superconvergent points will lead to a slightly overdetermined linear system which will be solved by means of a least squares approach. Based on a two-patch example, we will propose by choosing a suitable subset of the mixed degree superconvergent points a first technique to impose (as in the one-patch case) a square linear system. 

The remainder of the paper is organized as follows. In Section~\ref{sec:collocation}, we will recall the basic concept of isogeometric collocation for solving the Poisson's and the biharmonic equations by means of $C^2$ and $C^4$-smooth finite discretization spline spaces, respectively. Section~\ref{sec:problem_statement} will present the used $C^s$-smooth, $s=2,4$, isogeometric discretization spline space which will be a $C^s$-smooth spline space of mixed degree possessing the minimal possible degree $p=s+1$ everywhere except in a small neighborhood of the inner edges and of the vertices with a valency greater one where the degree needs to be $p=2s+1$. Two possible choices of collocation points, namely the mixed degree Greville and the mixed degree superconvergent points will be introduced in Section~\ref{subsec:collocationPoints}. Section~\ref{section_Numerical_examples} will then present several numerical examples, which will study the convergence behavior under $h$-refinement with respect to the $L^2$-norm and $H^m$-seminorms for $1 \leq m \leq \sm$, $\sm=2,4$, and which will demonstrate the power of our collocation method to solve the Poisson's and the biharmonic equations. Finally, we will summarize the results and will present some possible further research topics in Section~\ref{sec:Conclusion}. 


\section{The multi-patch isogeometric collocation method} \label{sec:collocation}

In this section, we will present the isogeometric collocation approach, which will be used to solve two particular model problems, namely the Poisson's and the biharmonic equation, over planar multi-patch domains by means of $C^2$ and $C^4$-smooth isogeometric discretization spline spaces. Before, we will describe the configuration of the considered planar multi-patch domain.

\paragraph{The multi-patch configuration} 
We assume that $\Omega \subset \R^2$ is an open planar domain whose closure~$\overline{\Omega}$ can be described as a disjoint union of open quadrilateral patches~$\Omega^{(i)}$, $i \in \mathcal{I}_{\Omega}$, open edges~$\Gamma^{(i)}$, $i \in \mathcal{I}_{\Gamma}$, and vertices~$\bfm{\Xi}^{(i)}$, $i \in \mathcal{I}_{\Xi}$, cf. Fig.~\ref{fig:multipatchCase}.
Moreover, we assume that the closures of any two patches have either an empty intersection, possess exactly one common vertex or share the whole common edge, and that each patch is parameterized by a bilinear, bijective and regular geometry mapping~$\ab{F}^{(i)}$ as $\overline{\Omega^{(i)}} = \ab{F}^{(i)}([0,1]^{2})$, with
\begin{align*}
 \ab{F}^{(i)}: [0,1]^{2}  \rightarrow \R^{2}, \quad 
 \bb{\xi} =(\xi_1,\xi_2) \mapsto
 \ab{F}^{(i)}(\bb{\xi}) = \ab{F}^{(i)}(\xi_1,\xi_2) 
 , \quad i \in \mathcal{I}_{\Omega}.
\end{align*}
Later, it will be important to distinguish between inner and boundary edges and between inner and boundary vertices. Therefore, the index sets $\mathcal{I}_{\Gamma}$ and $\mathcal{I}_{\Xi}$ are also divided into $\mathcal{I}_{\Gamma} = \mathcal{I}_{\Gamma}^I \dot{\cup} \mathcal{I}_{\Gamma}^B$ and $\mathcal{I}_{\Xi} = \mathcal{I}_{\Xi}^I \dot{\cup} \mathcal{I}_{\Xi}^B$, respectively, where 
$I$ denotes the inner and $B$ the boundary case. 

\begin{figure}[bth]
    \centering
    \includegraphics[scale=0.21]{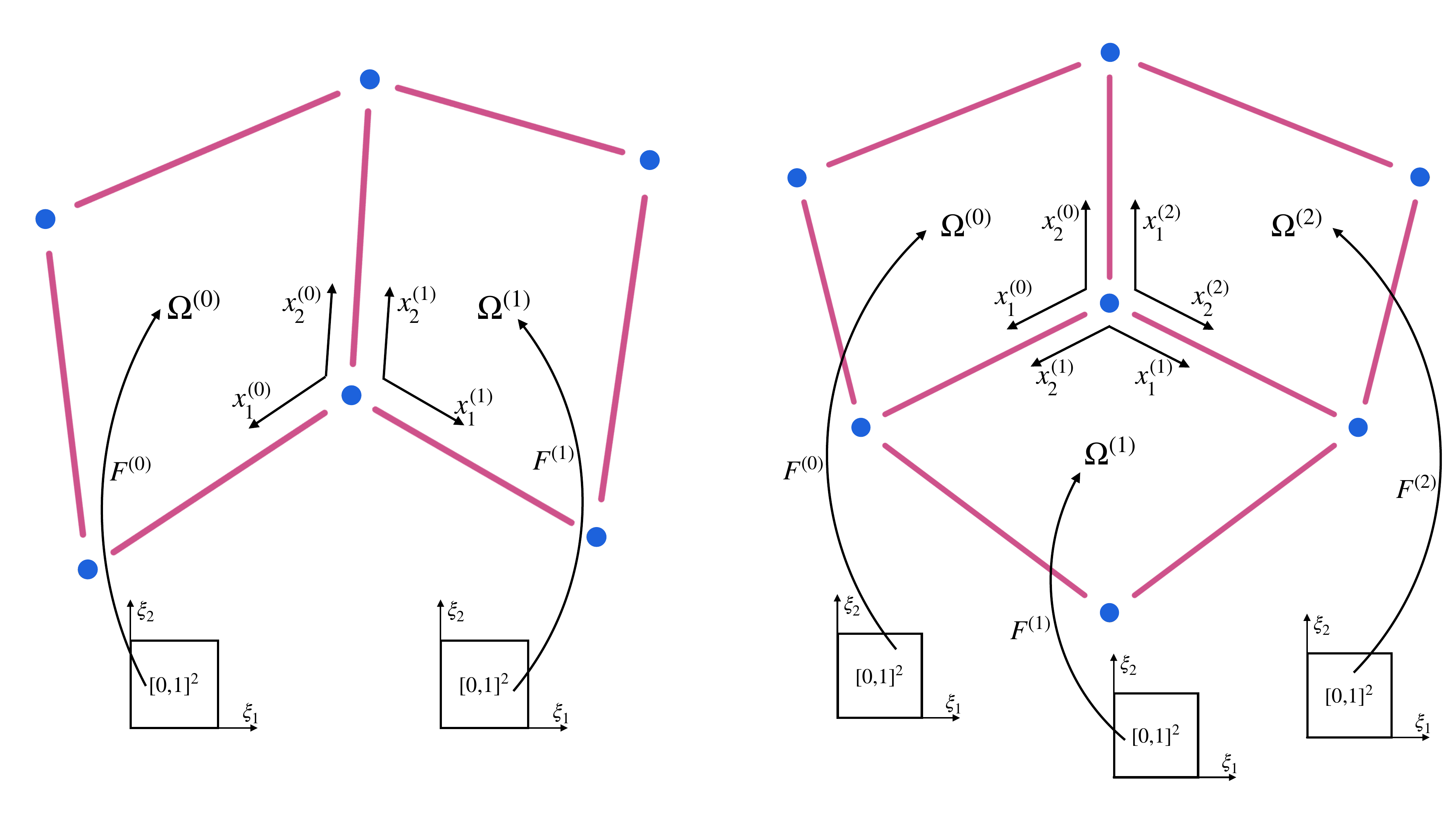}
    \caption{Examples of a two-patch domain $\overline{\Omega}$ (left) and of a three-patch domain $\overline{\Omega}$ (right) with the patches $\Omega^{(i)}$ and their associated geometry mappings $\ab{F}^{(i)}$, with the edges $\Gamma^{(i)}$ (violet) and with the vertices $\bfm{\Xi}^{(i)}$ (blue), where $x_1^{(i)} = \ab{F}^{(i)}(\xi_1,0)$ and $x_2^{(i)} = \ab{F}^{(i)}(0,\xi_2)$.
    }
    \label{fig:multipatchCase}
\end{figure}
\paragraph{Isogeometric collocation for the Poisson's equation} 

Let $g:\Omega \to \R$, $g_1: \partial \Omega \to \R$ be sufficiently smooth functions. We are interested in finding  $u:\overline{\Omega} \to \R$, $u \in C^2(\overline{\Omega})$, which strongly solves the 
Poisson's equation
\begin{align} \label{eq:Poisson}
\triangle u (\bfm{x})  =   g(\bfm{x}), \;\;  \bfm{x} \in \Omega, \quad {\rm and} \quad
u (\bfm{x})   = g_1(\bfm{x}) ,   \;\; \bfm{x} \in \partial \Omega. 
\end{align}
By following the multi-patch isogeometric collocation method~\cite{KaVi20}, we will compute a $C^2$-smooth approximation $u_h \in \mathcal{W}_h^2$ of the solution~$u$, where 
the space $\mathcal{W}_h^2$ denotes the finite dimensional discretization spline space containing $C^2$-smooth functions and $h$ denotes the selected mesh size.
For this we need a set of global collocation points~$\bfm{\glob}_j$, $j \in \mathcal{J}$, which are further separated into inner collocation points $\bfm{\glob}_j^I$, $j \in \mathcal{J}_I$, and boundary collocation points $\bfm{\glob}_j^B$,  $j \in \mathcal{J}_B$.
Inserting these points into \eqref{eq:Poisson}, we obtain 
\begin{align} \label{eq:PoissonPoints}
\triangle u_h (\bfm{\glob}_j^I)  =   g(\bfm{\glob}_j^I), \;\; j \in \mathcal{J}_I, \quad {\rm and} \quad 
u_h (\bfm{\glob}_j^B)  = g_1(\bfm{\glob}_j^B) ,   \;\; j \in \mathcal{J}_B.  
\end{align}
To use the isogeometric approach for solving problem~$\eqref{eq:PoissonPoints}$, we first have to express the global collocation points~$\bfm{\glob}_{j}$, $j \in \mathcal{J}$, with respect to local coordinates as
$$
\bfm{\loc}^{I}_{j} = \left( \bfm{F}^{(\iota_j)}\right)^{-1} \left( \bfm{\glob}_j^I \right) , \;  j \in \mathcal{J}_I, \quad \mbox{ and } \quad
\bfm{\loc}^{B}_{j} = \left( \bfm{F}^{(\iota_j)}\right)^{-1} \left( \bfm{\glob}_j^B \right) , \;  j \in \mathcal{J}_B,
$$
where 
\begin{equation} \label{eq:patch_index_selection}  
\iota_j := \min \{ i \in \mathcal{I}_\Omega, \; \bfm{\glob}_j \in \overline{\Omega^{(i)}} \}.
\end{equation}
Having the local collocation points $\bfm{\loc}^{I}_{j}$ and $\bfm{\loc}^{B}_{j}$, equations
\eqref{eq:PoissonPoints} can be transformed into 
\begin{align} \label{eq:collocationSystemLocalPoisson}
\frac{1}{\left| \det J \ab{F}^{(\iota_j)} \left( \bfm{\loc}^{I}_{j}  \right) \right|}  \left( \nabla \circ \left( N^{(\iota_j)}(\bb{\xi})  
\nabla \left(u_h ( \bfm{F}^{(\iota_j)}(\bb{\xi}))\right) \right)\right){\bigg|}_{\bfm{\xi} =\bfm{\loc}^{I}_{j} } & = 
g\left(\ab{F}^{(\iota_j)}\left(\bfm{\loc}^{I}_{j}\right)\right), \quad j \in \mathcal{J}_{I},
\nonumber \\[-0.25cm]
&  \\[-0.25cm]
u_h \left( \bfm{F}^{(\iota_j)}  \left(\bfm{\loc}^{B}_{j} \right)\right) &  = g_1\left(\ab{F}^{(\iota_j)}\left(\bfm{\loc}^{B}_{j}\right)\right),
 \mbox{ }j \in \mathcal{J}_{B}.\nonumber
\end{align}
This implies a linear system for the unknown coefficients $c_i$ of the approximation $u_h = \sum_{i \in \mathcal{I}} c_i \phi_i \in \mathcal{W}_h^2$, where 
$\mathcal{I} = \{0,1, \ldots, \dim \mathcal{W}_h^2-1 \}$, and $\{\phi_i\}_{i\in \mathcal{I}}$ is the basis of $\mathcal{W}_h^2$. 

\paragraph{Isogeometric collocation for the biharmonic equation} 

Let us recall the generalization of the approach for the Poisson's equation from the previous paragraph to the collocation method for the biharmonic equation, cf.~\cite{KaKoVi24}.
The goal is now to find  $u:\overline{\Omega} \to \R$, $u \in C^4(\overline{\Omega})$, which solves the biharmonic equation
\begin{align} \label{eq:biharmonic}
\triangle^2 u (\bfm{x}) =   g(\bfm{x}), \; \; \bfm{x} \in \overline{\Omega},  
\quad {\rm and} \quad
u (\bfm{x})  = g_1(\bfm{x}) ,   \;\; 
\partial_{\bfm{n}} u (\bfm{x}) = g_2(\bfm{x}) ,   \;\; \bfm{x} \in \partial \Omega, 
\end{align}
in strong form, 
where $g:\overline{\Omega} \to \R$ and $g_1, g_2: \partial \Omega \to \R$ are sufficiently smooth functions. 
More precisely, we want to get the $C^4$-smooth approximation $u_h \in \mathcal{W}_h^4$ of the exact solution~$u$, where $\mathcal{W}_h^4$ denotes now the finite dimensional discretization spline space containing $C^4$-smooth functions. 
Inserting global inner collocation points $\bfm{\glob}_j^I$, $j \in \mathcal{J}_I$, and global boundary collocation points $\bfm{\glob}_j^B$,  $j \in \mathcal{J}_B$, into \eqref{eq:biharmonic}, we obtain 
\begin{align} \label{eq:biharmonicPoints}
\triangle^2 u_h (\bfm{\glob}_j^I)   =   g(\bfm{\glob}_j^I), \; \; j \in \mathcal{J}_I, 
\quad {\rm and} \quad
u_h (\bfm{\glob}_j^B)  = g_1(\bfm{\glob}_j^B) , \;\;  
\partial_{\bfm{n}} u_h (\bfm{\glob}_j^B)   = g_2(\bfm{\glob}_j^B) ,   \;\;  j \in \mathcal{J}_B. 
\end{align}
Then, employing the associated inner and boundary local collocation points $\bfm{\loc}^{I}_{j}$, $j \in \mathcal{J}_{I}$, and $\bfm{\loc}^{B}_{j}$, $j \in \mathcal{J}_{B}$, respectively, equations \eqref{eq:biharmonicPoints} can be transformed (\cite{BaDe15,KaKoVi24}) into
\begin{align}  \label{eq:collocationSystemLocal}
\frac{1}{\left| \det J \ab{F}^{(\iota_j)} \left( \bfm{\loc}^{I}_{j}  \right) \right|}  \nabla \circ  \hspace{-0.05cm} \Bigg( N^{(\iota_j)}(\bb{\xi}) 
\nabla \Bigg(  \frac{1}{\left| \det J \ab{F}^{(\iota_j)}  \left( \bfm{\xi} \right) \right|}  & \nabla \circ \left( N^{(\iota_j)}(\bb{\xi}) 
\nabla \left(u_h ( \bfm{F}^{(\iota_j)}(\bb{\xi}) ) \right) \right) \hspace{-0.13cm} \bigg) \hspace{-0.05cm} \bigg) {\Bigg|}_{\bfm{\xi} =\bfm{\loc}^{I}_{j} } \nonumber\\
&
=
g\left(\ab{F}^{(\iota_j)}\left(\bfm{\loc}^{I}_{j}\right)\right),\;\, \;\; j \in \mathcal{J}_{I}, \nonumber \\
&  \\[-0.4cm]
\normalsize
u_h \left( \bfm{F}^{(\iota_j)}  \left(\bfm{\loc}^{B}_{j} \right)\right)  & = g_1\left(\ab{F}^{(\iota_j)}\left(\bfm{\loc}^{B}_{j}\right)\right),
 \mbox{ }j \in \mathcal{J}_{B}, \nonumber \\
 \left\langle\bfm{n},\left(J\bfm{F}^{(\iota_j)}\left(\bfm{\loc}^{B}_{j}\right)\right)^{-1}\nabla u_h\left( \bfm{F}^{(\iota_j)}  \left(\bfm{\loc}^{B}_{j} \right)\right) \right\rangle  & = g_2\left(\ab{F}^{(\iota_j)}\left(\bfm{\loc}^{B}_{j}\right)\right),
 \mbox{ }j \in \mathcal{J}_{B}. \nonumber
\end{align}
This again leads to a linear system for the unknown coefficients $c_i$ of the approximation $u_h = \sum_{i \in \mathcal{I}} c_i \phi_i \in \mathcal{W}_h^4$, where now $\mathcal{I} = \{0,1, \ldots, \dim \mathcal{W}_h^4-1 \}$, and $\{\phi_i\}_{i\in \mathcal{I}}$ is the basis of $\mathcal{W}_h^4$. 


\paragraph{Separation of collocation points} \label{subsec:eparation}

The straightforward separation of the set of global collocation points~$\bfm{\glob}_j$, $j \in \mathcal{J}$, into inner collocation points $\bfm{\glob}_j^I$, $j \in \mathcal{J}_I$, and boundary collocation points $\bfm{\glob}_j^B$,  $j \in \mathcal{J}_B$, would be to take all global collocation points which lie on the boundary of the domain $\partial \Omega$ as boundary collocation points,
and to take the remaining points as inner collocation points. 
However, this separation works well in the case of Poisson's equation, while for the biharmonic equation this would lead even in the one-patch case to an overdetermined linear system, since we would have too many collocation points in total. Therefore, we will follow in the biharmonic case the separation technique used in~\cite{RealiGomez2015, GomezRealiSangali2014, KaKoVi24}. This will imply a square linear system for the one-patch case, and a slightly overdetermined linear system for the multi-patch case. 

Below, we will present in Section~\ref{sec:problem_statement} the construction of a specific $C^s$-smooth mixed degree isogeometric multi-patch discretization spline space, which will be a variant of the $C^s$-smooth mixed degree and regularity spline space~\cite{KaKoVi24b}, and which will be employed for $s=2$ and $s=4$ to perform isogeometric collocation for solving the Poisson's and the biharmonic equation over several multi-patch domains, cf. Section~\ref{section_Numerical_examples}. For this purpose, we will further describe in Section~\ref{subsec:collocationPoints} two particular sets of collocation points, namely mixed degree Greville and mixed degree superconvergent points.




\section{
The smooth mixed degree isogeometric discretization spline space} \label{sec:problem_statement}
 
We will present a particular $C^s$-smooth mixed degree isogeometric spline space over a bilinearly parameterized planar multi-patch domain, whose construction will be motivated by the design of the $C^s$-smooth mixed degree and regularity spline space~\cite{KaKoVi24b}. These mixed degree spline spaces represent a good alternative to the ``standard" $C^s$-smooth spline space~\cite{KaVi20b}, that requires a high degree of $p=2\sm+1$ on the entire multi-patch domain, by reducing this high degree in most parts of the domain to the minimal possible one of $p=s+1$. In detail, the proposed $C^s$-smooth mixed degree spline space will possess the high spline degree $p=2s+1$ just in a small neighborhood of the inner edges and of the vertices of patch valency greater than one, and will hence reduce further parts of high degree compared to the mixed degree spline space~\cite{KaKoVi24b}, that still requires the high degree of $p=2s+1$ in the vicinity of all inner and boundary edges and of all inner and boundary vertices of the multi-patch domain.       



\subsection{The mixed degree underlying spline space on $[0,1]^2$}  \label{subsec:construction}

The construction of the $C^s$-smooth mixed degree isogeometric spline space will be based on the use of an appropriate mixed degree underlying spline space on $[0,1]^2$ to define the functions on the single patches. Before introducing this mixed degree underlying spline space, we will present some needed notations. Firstly, we will denote by $\mathcal{S}_h^{(p,p),(r,r)}([0,1]^2) =\mathcal{S}_h^{\ab{p},\ab{r}}([0,1]^2)$ the tensor-product spline space $\mathcal{S}_h^{p,r}([0,1]) \otimes \mathcal{S}_h^{p,r}([0,1])$, where $\mathcal{S}_h^{p,r}([0,1])$ is the univariate spline space on the unit interval~$[0,1]$ of degree~$p$, regularity~$r$ and mesh size~$h=\frac{1}{k+1}$ possessing the uniform open knot vector $(t_0^{p,r},\ldots,t_{2p+k(p-r)+1}^{p,r})$ with $k$ inner knots $\frac{i}{k+1}$, $i=1,\ldots,k$, of multiplicity $p-r$. Secondly, the B-spline bases of the spline spaces $\mathcal{S}_h^{p,r}([0,1])$ and  $\mathcal{S}_h^{\ab{p},\ab{r}}([0,1]^2)$ will be denoted by $N_{j}^{p,r}$ and $N_{j_1,j_2}^{\ab{p},\ab{r}}=N_{j_1}^{p,r}N_{j_2}^{p,r}$, respectively, with $j,j_1,j_2=0,1,\ldots,n_p-1$, 
where $n_{p}= \dim \mathcal{S}_h^{p,r}([0,1]) = p+1+k(p-r)$. 

The goal is to use a modified version of the mixed degree underlying 
space $\mathcal{S}_h^{(\ab{p}_1, \ab{p}_2),\ab{\sm}}([0,1]^2)$ from \cite{KaKoVi24b} which will consist on the one hand of spline functions of degree $p_1=\sm+1$ with vanishing derivatives of order~$\ot \leq \sS$ at the part of the boundary of $[0,1]^2$ corresponding to inner edges of the multi-patch domain $\overline{\Omega}$, and on the other hand of spline functions of degree $p_2=2\sm+1$ whose support is contained only in the vicinity of the same part of the boundary of $[0,1]^2$. This is different from~\cite{KaKoVi24b}, where the mixed degree underlying spline space $\mathcal{S}_h^{(\ab{p}_1, \ab{p}_2),\ab{\sm}}([0,1]^2)$ possesses spline functions of degree~$p_2=2s+1$ in the vicinity of all boundaries of $[0,1]^2$ independent if a boundary edge of $[0,1]^2$ corresponds to an inner edge of the multi-patch domain~$\overline{\Omega}$.

Below, we will briefly describe the construction of the modified mixed degree underlying spline space $\mathcal{S}_h^{(\ab{p}_1, \ab{p}_2),\ab{\sm}}([0,1]^2)$, cf.~Fig.~\ref{fig:SpacesMixed}, and will refer to \ref{sec:AppendixSec2} for the detailed construction:
\begin{itemize}
\item
We start with the B-splines $N_{j_1,j_2}^{\ab{p}_1,\ab{\sm}}$ from the spline space $\mathcal{S}_h^{\ab{p}_1,\ab{\sm}}([0,1]^2)$, $p_1=s+1$, that have vanishing derivatives of order $\ot \leq \sm$ at the part of the boundary $\partial([0,1]^2)$ corresponding to the inner edges of the multi-patch domain $\overline{\Omega}$. We denote the subspace spanned by these B-splines by $\mathcal{S}_{1} ([0,1]^2)$. 
\item In the next step, we add the B-splines $N_{j_1,j_2}^{\ab{p}_2,\ab{\sm}}$ from the space $\mathcal{S}_h^{\ab{p}_2,\ab{\sm}}([0,1]^2)$, $p_2=2s+1$, having a non-vanishing derivative of order $\ot \leq \sm$ at the same part of $\partial( [0,1]^2)$. These functions span the subspace denoted by $\mathcal{S}_{2} ([0,1]^2)$. 
\item Finally, to ensure the completeness of the 
underlying spline space $\mathcal{S}_h^{(\ab{p}_1, \ab{p}_2),\ab{\sm}}([0,1]^2)$, we add some of the previously eliminated B-splines from the space $\mathcal{S}_h^{\ab{p}_1,\ab{\sm}}([0,1]^2)$, after truncating them (in one or both directions) with respect to the space $\mathcal{S}_h^{\ab{p}_2,\ab{\sm}}([0,1]^2)$ such that their derivatives of order $\ot \leq \sm$ vanish at the part of $\partial([0,1]^2)$ associated to
the inner edges of the multi-patch domain $\overline{\Omega}$. Thereby,
$\overline{N}_{i}^{\,p_1,\sm}$ denotes the univariate truncation of ${N}_{i}^{p_1,\sm} = \sum_{j=0}^{n_{p_2}-1} \mu_j^{(i)} N_j^{p_2,\sm}$, and is defined as 
$
\overline{N}_{i}^{\,p_1,\sm} = 
\sum_{j=\sm+1}^{n_{p_2}-\sm-2} \mu_j^{(i)} N_j^{p_2,\sm}$, 
$i=0,1,\ldots, n_{p_1}-1, \,  \mu_j^{(i)} \in \R,\, \mu_j^{(i)} \geq 0.
$
These bivariate truncated B-splines 
span the subspace denoted by $\mathcal{\overline{S}}_1 ([0,1]^2)$. 
\end{itemize}
We therefrom get the full mixed degree underlying spline space $\mathcal{S}_h^{(\ab{p}_1, \ab{p}_2),\ab{\sm}}([0,1]^2)$ as the direct sum
\begin{equation}  \label{eq:MixedSpace}
 \mathcal{S}_h^{(\ab{p}_1, \ab{p}_2),\ab{\sm}}([0,1]^2) = \mathcal{S}_{1} ([0,1]^2)  \oplus \mathcal{\overline{S}}_1 ([0,1]^2) \oplus \mathcal{S}_2 ([0,1]^2),
\end{equation}
where the detailed description of the subspaces $\mathcal{S}_{1} ([0,1]^2)$, $\mathcal{\overline{S}}_1 ([0,1]^2)$ and $\mathcal{S}_2 ([0,1]^2)$ is based on the number and position of the edges of $[0,1]^2$ corresponding to the inner edges of $\overline{\Omega}$, which comprises five possible cases, namely four, three, two adjacient, two opposite or just one inner edge, and is explicitly given in \ref{sec:AppendixSec2}. The positions of the extrema of all basis functions of the five possible variants of the space $\mathcal{S}_h^{(\ab{p}_1, \ab{p}_2),\ab{\sm}}([0,1]^2)$  
are presented in Fig.~\ref{fig:SpacesMixed}. Note that the mixed degree underlying space $\mathcal{S}_h^{(\ab{p}_1, \ab{p}_2),\ab{\sm}}([0,1]^2)$ in~\cite{KaKoVi24b} coincides with the first variant, namely for the case of four inner edges, and is always used in~\cite{KaKoVi24b} independent of the number of inner edges. 
\begin{figure}[h!]
        \centering
\begin{tabular}{cc}
 \includegraphics[scale=0.27]{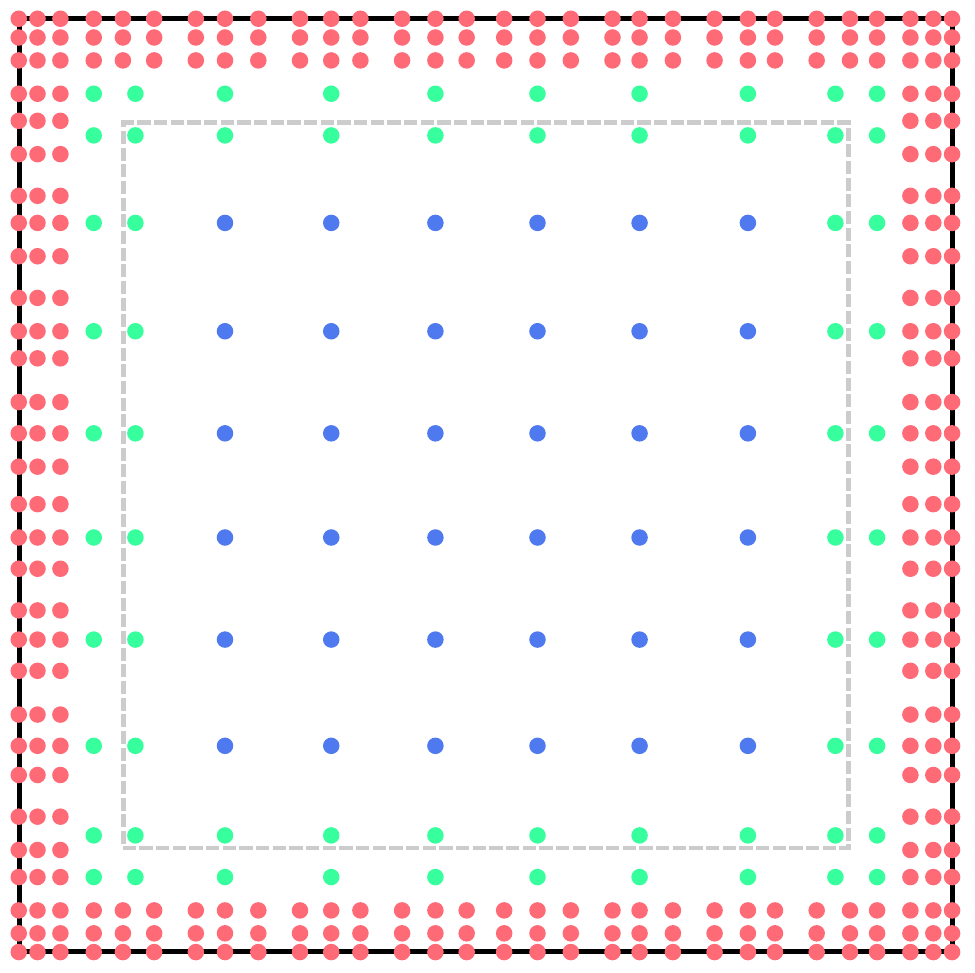} $\quad$ &
\includegraphics[scale=0.35]{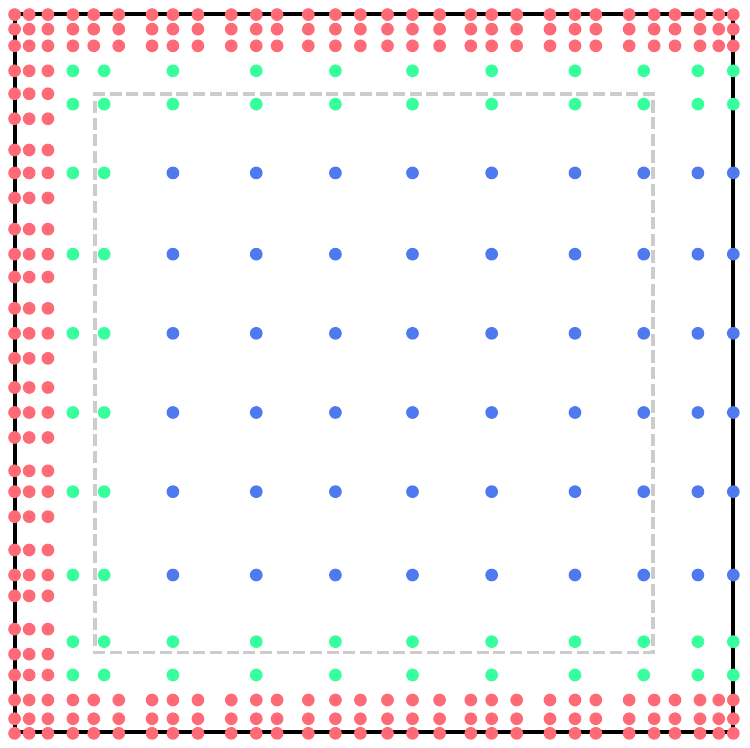}  \\
Four inner edges  & Three inner edges \\[0.4cm]
\end{tabular}
    \begin{tabular}{ccc}
\includegraphics[scale=0.35]{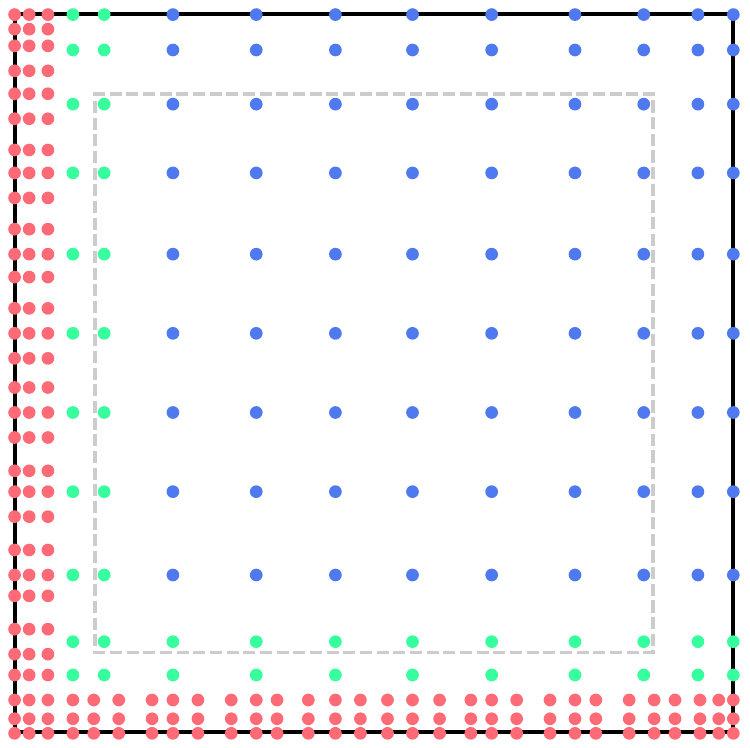} & \includegraphics[scale=0.35]{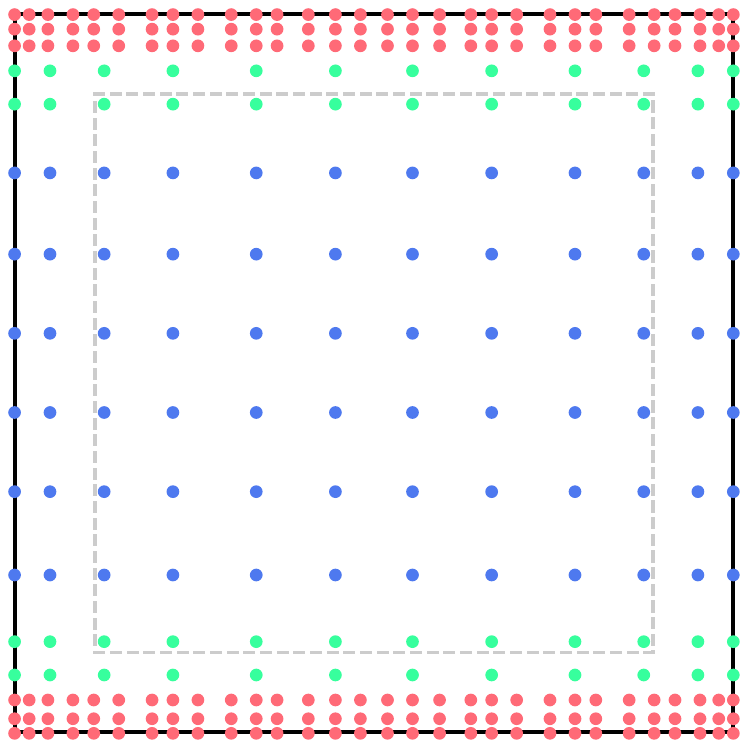}   &
\includegraphics[scale=0.35]{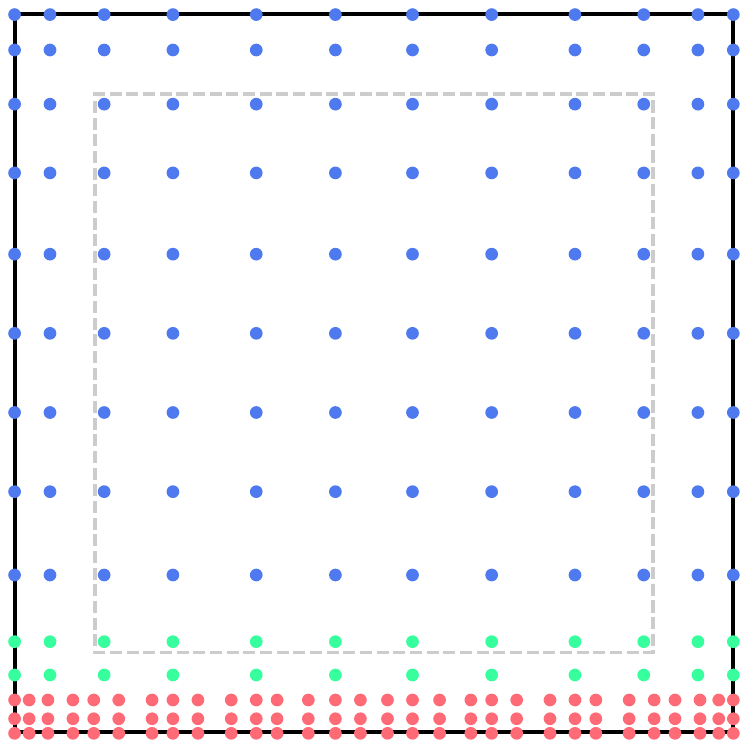} \\
Two adjacent inner edges & Two opposite inner edges &  One inner edge
\end{tabular}
    \caption{
    The positions of extrema of all basis functions 
    of the space $\mathcal{S}_h^{(\ab{\sm}+\ab{1},2\ab{\sm}+\ab{1}),\ab{\sm}}([0,1]^2)$ for $\sm=2$ 
    and $k=8$. The five possible different variants of the space $\mathcal{S}_h^{(\ab{p}_1, \ab{p}_2),\ab{\sm}}([0,1]^2)$ depend on the number and position of the edges of $[0,1]^2$ which correspond to the inner edges of the multi-patch domain $\overline{\Omega}$.
    The blue, green and red dots correspond to the functions belonging to the spaces $\mathcal{S}_1([0,1]^2)$, $\mathcal{\overline{S}}_1([0,1]^2)$ and $\mathcal{S}_2([0,1]^2)$, respectively. The dashed gray lines denote the boundary of the subdomain $[h,1-h]^2$. 
    }
    \label{fig:SpacesMixed}
\end{figure}
One can show (cf.~\cite{KaKoVi24b}) that 
all basis functions of the spaces $\mathcal{S}_1([0,1]^2)$,  $\mathcal{\overline{S}}_1([0,1]^2)$ and $\mathcal{S}_2([0,1]^2)$ are linearly independent and thus form a basis of the full space $\mathcal{S}_h^{(\ab{p}_1, \ab{p}_2),\ab{\sm}}([0,1]^2)$. 
They have a local support, are nonnegative and form a partition of unity. 
Additionally (cf.~\ref{sec:AppendixSec2}), we can derive a general dimension formula of the space $\mathcal{S}_h^{(\ab{p}_1, \ab{p}_2),\ab{\sm}}([0,1]^2)$:  
Let $E$, $1\leq E \leq 4$, and $V$, $0\leq V \leq 2$, be the number of edges and vertices of $[0,1]^2$ which correspond to the inner edges and boundary vertices of patch valency one of the multi-patch domain $\overline{\Omega}$, respectively. Then
\begin{equation*}  \label{eq:dimGeneral}  
\dim \mathcal{S}_h^{(\ab{p}_1, \ab{p}_2),\ab{\sm}}([0,1]^2) =  
(k+2)^2 + \big((2+2E)k+(12-E-2V)\big)s + \big((E \,k+(5-V)\big)s^2.
\end{equation*}

\subsection{The $C^{\sS}$-smooth mixed degree discretization spline space} \label{subsec:Cs_space}

We will describe the construction of the $C^s$-smooth mixed degree isogeometric spline space that will be used as discretization space for the isogeometric collocation method in Section~\ref{sec:collocation}. The generated spline space will be an adapted version of the spline space \cite{KaKoVi24b} with the aim to further reduce the places with functions of the high degree $p_2=2s+1$. Instead of having functions of the high degree~$p_2=2s+1$ in the vicinity of all edges and vertices as in~\cite{KaKoVi24b}, the adapted spline space will have functions of the high degree $p_2=2s+1$ just in the neighborhood of all inner edges and of all vertices of patch valency greater than one. The construction of the spline space will be based on the application of the mixed degree underlying spline space introduced in the previous subsection, and will allow to further reduce the required number of degrees of freedom for performing isogeometric collocation compared to the use of~\cite{KaKoVi24b} and especially to the use of~\cite{KaVi20b}, which consists of only functions of the high degree~$p_2=2s+1$, cf.~\cite{KaVi20, KaKoVi24}.  

For each inner edge $\Gamma^{(i)}$, $i \in \mathcal{I}_{\Gamma}^I$, with $\overline{\Gamma^{(i)}}=\overline{\Omega^{(i_0)}} \cap \overline{\Omega^{(i_1)}}$, $i_0,i_1 \in \mathcal{I}_{\Omega}$, assuming that $\ab{F}^{(i_0)}(0,\xi)=\ab{F}^{(i_1)}(0,\xi)$, $\xi \in [0,1]$, as shown in Fig.~\ref{fig:multipatchCase} (left), we define the spline functions $\g_\ell^{(i,\Side)}:[0,1]\to \R$, $\Side\in \{\LL,\RR\}$, as 
\begin{equation*}   \label{eq:gC2}
 \g_\ell^{(i,\Side)}(\xi) = \left(\alpha^{(i,\Side)}(\xi)\right)^{-\ell} \hspace{-0.1cm} \partial_1^\ell 
 \left(\phi \circ \ab{F}^{(\tau)}\right)(0,\xi) - \sum_{j=0}^{\ell-1} {\ell \choose j} 
 \left(\frac{\beta^{(i,\Side)}(\xi)}{\alpha^{(i,\Side)}(\xi)}\right)^{\ell-j}  \hspace{-0.3cm} \dd^{\ell-j} \gC_j^{(i,\tau)}(\xi),\; \,\ell=0,\ldots, \sS,
 \end{equation*}
where $\alpha^{(i,\Side)}$ and $\beta^{(i,\Side)}$
are linear polynomial functions (called gluing functions), 
determined as
\begin{equation*}  \label{eq:alphaLRbar}
 \alpha^{(i,\Side)}(\xi) = \lambda \det J \ab{F}^{(\Side)}(0,\xi) \quad {\rm and} 
 \quad \beta^{(i,\Side)}(\xi) = \frac{\partial_1 \ab{F}^{(\Side)}(0,\xi) \cdot \partial_2\ab{F}^{(\Side)}(0,\xi)}
 {||\partial_2 \ab{F}^{(\Side)}(0,\xi)||^{2}}, 
\end{equation*}
with $J \ab{F}^{(\Side)}$ 
being the Jacobian of $\ab{F}^{(\Side)}$ and 
$
 \lambda = \underset{\lambda >0}{\rm argmin} \left(\sum_{j=0}^1 || 
 \alpha^{(i,i_j)}+(-1)^j ||^2_{L^2}
 \right),
$
cf.~\cite{KaVi19a}. Then, we define a particular space 
of $C^s$-smooth isogeometric spline functions of mixed degree~$(\ab{p}_1,\ab{p}_2)$ and regularity $\ab{\sm}$ over the multi-patch domain~$\overline{\Omega}$, denoted by~$\W^{\sS}$, as
\begin{equation*} \label{eq:space_same_degree}
{\W}^{\sS} = \left\{ \begin{array}{lll} 
 \phi \in L^2(\overline{\Omega}): & \phi \circ \ab{F}^{(i)}  \in {\mathcal{S}_{h}^{(\ab{p}_1,\ab{p}_2),\ab{\sm}}([0,1]^{2})}, \, 
i \in \mathcal{I}_{\Omega}, 
\\[0.1cm]  
&  \g_\ell^{(i,\LL)}
= \g_\ell^{(i,\RR)}
, \; \ell=0,\ldots, \sS,\; i \in \mathcal{I}_{\Gamma}^I,\; i_0,i_1 \in \mathcal{I}_{\Omega} \\[0.1cm]
&  \gC^{(i)}_\ell \in \mathcal{S}_h^{p_2-\ell,2\sm -\ell}([0,1]) , \; \ell=0,\ldots, \sS,\; i \in \mathcal{I}_{\Gamma}^I 
\end{array} 
\right\}, 
\end{equation*}
where $\gC_{\ell}^{(i)}$ denotes the equally valued functions $\g_\ell^{(i,\LL)}= \g_\ell^{(i,\RR)}$, $\ell =0,\ldots,\sS$, which describes a specific derivative of order~$\ell$ of $\phi$ across the inner edge~$\Gamma^{(i)}$, cf.~\cite{KaVi20b}. 
The spline space $\W^{\sS}$ can 
be constructed as in \cite{KaKoVi24b} as the direct sum of smaller subspaces that are associated with the individual patches~$\Omega^{(i)}$, $i \in \mathcal{I}_\Omega$, edges~$\Gamma^{(i)}$, $i \in \mathcal{I}_\Gamma$, and vertices $\bfm{\Xi}^{(i)}$, $i \in \mathcal{I}_\Xi$, i.e.
\begin{equation*} \label{eq:direct_sum_subspace}
\W^{\sS} = \left(\bigoplus_{i \in \mathcal{I}_\Omega} \mathcal{W}^\sS_{\Omega^{(i)}}\right) 
    \oplus \left(\bigoplus_{i \in \mathcal{I}_\Gamma} \mathcal{W}^\sS_{\Gamma^{(i)}}\right) 
    \oplus \left(\bigoplus_{i \in \mathcal{I}_\Xi} \mathcal{W}^\sS_{\bfm{\Xi}^{(i)}}\right) ,
\end{equation*}
with the difference of using for~$\mathcal{S}_h^{(\ab{p}_1, \ab{p}_2),\ab{\sm}}([0,1]^2)$ the modified mixed degree underlying spline described in the previous subsection.
In order to simplify the notation, let us define the three operators
$ {\phi}_{\Omega^{(i)}}: {\mathcal{S}_{h}^{(\ab{p}_1,\ab{p}_2),\ab{\sm}}([0,1]^{2})} \to \W^{\sS} $, 
$ \phi_{\Gamma^{(i)}}: {\mathcal{S}_{h}^{(\ab{p}_1,\ab{p}_2),\ab{\sm}}([0,1]^{2})} \times {\mathcal{S}_{h}^{(\ab{p}_1,\ab{p}_2),\ab{\sm}}([0,1]^{2})} \to \W^{\sS} $ and 
$ \phi_{\bfm{\Xi}^{(i)}} : \prod_{\rho=0}^{\nu_i-1} \mathcal{S}_{h}^{(\ab{p}_1,\ab{p}_2),\ab{\sm}}
\to \W^{\sS} $ as 
\begin{equation*}  \label{eq:PhiOmega2}
{\phi}_{\Omega^{(i)}}(\mathcal{S})(\bfm{x})  = 
\begin{cases}
   (\mathcal{S}\circ (\ab{F}^{(i)})^{-1})(\bfm{x}) \;
\mbox{ if }\f \, \bfm{x} \in \overline{\Omega^{(i)}},
\\ 0 \quad \mbox{ if }\f \, \bfm{x} \in \overline{\Omega} \backslash \overline{\Omega^{(i)}},
\end{cases} 
\end{equation*} 
\begin{equation*}  \label{eq:defphiGamma}
    \phi_{\Gamma^{(i)}}(\mathcal{S}_0,{\mathcal{S}_1}) (\bfm{x}) = 
\begin{cases}
  ({\mathcal{S}_\rho} \circ (\ab{F}^{(i_\rho)})^{-1})(\bfm{x}) \;
\mbox{ if }\f \, \bfm{x} \in \overline{\Omega^{(i_\rho)}}, \; \rho=0,1,
\\[0.15cm] 
  0\quad {\rm otherwise},
\end{cases}
\end{equation*}
\begin{equation*}  \label{eq:defphiXi}
  \phi_{\bfm{\Xi}^{(i)}}(\mathcal{S}_0,\ldots,\mathcal{S}_{\nu_i-1}) (\ab{x}) = 
  \begin{cases}
   \left( \mathcal{S}_\rho \circ (\ab{F}^{(i_\rho)})^{-1}\right)(\bfm{x}) \;
\mbox{ if }\f \, \bfm{x} \in \overline{\Omega^{(i_\rho)}},\; \rho=0,1,\ldots,{\nu}_i -1,
\\
0 \quad \mbox{ otherwise,}
\end{cases}
\end{equation*}
where $\overline{\Gamma^{(i)}}=\overline{\Omega^{(i_0)}} \cap \overline{\Omega^{(i_1)}}$, $i \in \mathcal{I}_{\Gamma}^I$, $i_0,i_1 \in \mathcal{I}_{\Omega}$, and $\bfm{\Xi}^{(i)} = \cap_{\rho=0}^{\nu_i-1} \overline{\Omega^{(i_\rho)}}$, $i \in \mathcal{I}_{\Xi}^{I}$.
In the following, we will describe the design of the single subspaces, which will be constructed as the span of corresponding basis functions.
\subsubsection{The patch subspace $\W^{\sS}_{\Omega^{(i)}}$}
We have to distinguish between different cases with respect to the number of inner edges of the multi-patch domain~$\overline{\Omega}$ that are contained in the patch~$\overline{\Omega^{(i)}}$. 
The construction depends on whether we have four, three, two adjacent, two opposite or only one inner edge contained in $\overline{\Omega^{(i)}}$. The case with four inner edges has been already described in \cite{KaKoVi24b}, while all the other cases are modifications that are needed in order to reduce the degree of the functions to $p_1=s+1$, and hence the number of degrees of freedom also wherever possible near the boundary of the multi-patch domain $\overline{\Omega}$. 

\paragraph{Four inner edges}

The patch subspace $\W^{\sS}_{\Omega^{(i)}}$ is given as
\begin{align*}
& \W^{\sS}_{\Omega^{(i)}}  = \Span \left\{ {\phi}_{\Omega^{(i)}}(N_{j_1,j_2}^{\ab{p}_1,\ab{\sm}}) \; |\;  j_1,j_2= \sm+1,\ldots, n_{p_1}-\sm-2  \right\} \oplus \\ & \, \Span \left\{ 
{\phi}_{\Omega^{(i)}}
(\overline{N}_{j_1}^{{p}_1,\sm} \overline{N}_{j_2}^{{p}_1,{\sm}})\; |\;  j_1=1,\ldots,\sm, n_{p_1}-\sm-1, \ldots,n_{p_1}-2; \; j_2=1,\ldots,n_{p_1}-2 \right\} \oplus\\
& \, \Span \left\{ 
{\phi}_{\Omega^{(i)}}
(\overline{N}_{j_1}^{{p}_1,\sm} \overline{N}_{j_2}^{{p}_1,{\sm}})\, |\; j_1=\sm+1,\ldots,n_{p_1}-\sm-2; \; j_2=1,\ldots,\sm, n_{p_1}-\sm-1, \ldots,n_{p_1}-2  \right\}.
\end{align*}

\paragraph{Three inner edges} Assuming that the inner edges 
are on the left, on the bottom and on the top, then the subspace~$\W^{\sS}_{\Omega^{(i)}}$ is equal to 
\begin{align*}
& \W^{\sS}_{\Omega^{(i)}}  = \Span \left\{ {\phi}_{\Omega^{(i)}}(N_{j_1,j_2}^{\ab{p}_1,\ab{\sm}}) \; |\;  j_1,j_2= \sm+1,\ldots, n_{p_1}-\sm-2  \right\} \oplus \\ 
& \, \Span \left\{ 
{\phi}_{\Omega^{(i)}}
(\overline{N}_{j_1}^{{p}_1,\sm} \overline{N}_{j_2}^{{p}_1,{\sm}})\;|\;  j_1=1,\ldots,\sm; \; j_2=1,\ldots,n_{p_1}-2 \right\} \oplus\\
& \, \Span \left\{ 
{\phi}_{\Omega^{(i)}}
(\overline{N}_{j_1}^{{p}_1,\sm} \overline{N}_{j_2}^{{p}_1,{\sm}})\,|\;  j_1=\sm+1,\ldots,n_{p_1}-\sm-2; \; j_2=1, \ldots,\sm,n_{p_1}-\sm-1,\ldots,n_{p_1}-2 \right\}.
\end{align*}
\paragraph{Two adjacent inner edges} Let the left and the bottom edge be the two inner edges. Then 
\begin{align*}
\W^{\sS}_{\Omega^{(i)}} & = \Span \left\{ {\phi}_{\Omega^{(i)}} (N_{j_1,j_2}^{\ab{p}_1,\ab{\sm}}) \;|\;  j_1,j_2= \sm+1,\ldots, n_{p_1}-\sm-2  \right\} \oplus \\ & \; \Span \left\{ 
{\phi}_{\Omega^{(i)}}
(\overline{N}_{j_1}^{{p}_1,\sm} \overline{N}_{j_2}^{{p}_1,{\sm}})\;|\;  j_1=1,\ldots,\sm; \; j_2=1,\ldots,n_{p_1}-\sm - 2 \right\} \oplus\\
& \; \Span \left\{ 
{\phi}_{\Omega^{(i)}}
(\overline{N}_{j_1}^{{p}_1,\sm} \overline{N}_{j_2}^{{p}_1,{\sm}})\;|\;  j_1=\sm+1,\ldots,n_{p_1}-\sm-2; \; j_2=1, \ldots,\sm \right\}.
\end{align*}

\paragraph{Two opposite inner edges} Assuming that the two inner edges are the bottom and the top one, then the subspace~$\W^{\sS}_{\Omega^{(i)}}$ is given as
\begin{align*}
& \W^{\sS}_{\Omega^{(i)}}  = \Span \left\{ {\phi}_{\Omega^{(i)}}(N_{j_1,j_2}^{\ab{p}_1,\ab{\sm}}) \;|\;  j_1,j_2= \sm+1,\ldots, n_{p_1}-\sm-2  \right\} \oplus \\ 
& \,  \Span \left\{ 
{\phi}_{\Omega^{(i)}}
(\overline{N}_{j_1}^{{p}_1,\sm} \overline{N}_{j_2}^{{p}_1,{\sm}}) \,|\;  j_1=\sm+1,\ldots,n_{p_1}-\sm-2; \; j_2=1, \ldots,\sm,n_{p_1}-\sm-1,\ldots,n_{p_1}-2 \right\}.
\end{align*}

\paragraph{One inner edge} Let the bottom edge be the only inner edge. Then 
\begin{align*}
\W^{\sS}_{\Omega^{(i)}} & = \Span \left\{ {\phi}_{\Omega^{(i)}}(N_{j_1,j_2}^{\ab{p}_1,\ab{\sm}}) \;|\;  j_1,j_2= \sm+1,\ldots, n_{p_1}-\sm-2  \right\} \oplus \\ 
& \; \Span \left\{ 
{\phi}_{\Omega^{(i)}}
(\overline{N}_{j_1}^{{p}_1,\sm} \overline{N}_{j_2}^{{p}_1,{\sm}})\;|\;  j_1=\sm+1,\ldots,n_{p_1}-\sm-2; \; j_2=1, \ldots,\sm \right\}.
\end{align*}


\subsubsection{The edge subspace $\W^{\sS}_{\Gamma^{(i)}}$}

Let us start with the construction of the edge subspace~$\W^{\sS}_{\Gamma^{(i)}}$ for an \emph{inner edge}~$\Gamma^{(i)}$, $i \in \mathcal{I}_{\Gamma}^I$, with $\overline{\Gamma^{(i)}}=\overline{\Omega^{(i_0)}} \cap \overline{\Omega^{(i_1)}}$, $i_0,i_1 \in \mathcal{I}_{\Omega}$, assuming that 
$\ab{F}^{(i_0)}(0,\xi) = \ab{F}^{(i_1)}(0,\xi)$, $\xi \in [0,1]$, cf. 
Fig.~\ref{fig:multipatchCase} (left). This subspace has already been described in \cite{KaVi20b, KaKoVi24b} and is recalled here for the sake of completeness. It is generated as
\begin{equation*} \label{eq:spaceW0hGamma}
\W^\sS_{\Gamma^{(i)}} =  \Span \left\{\phi_{\Gamma^{(i)}}(\g_{\Gamma^{(i)}; j_1,j_2}^{(i_0)},\g_{\Gamma^{(i)}; j_1,j_2}^{(i_1)}) \; | \; j_2=2\sm+1-j_1,
\ldots, 
k(j_1+1);
\; j_1=0,\ldots,\sm \right\},
\end{equation*} 
with 
the spline functions
\begin{equation}   \label{eq:basisFunctionsGenericG}
 \g_{\Gamma^{(i)}; j_1,j_2}^{(\Side)}(\xi_1,\xi_2) = 
\gamma_{j_1}
   \sum_{\ell=j_1}^{\sm}  {\ell \choose j_1} \left( \beta^{(\Side)}(\xi_2)\right)^{\ell-j_1} 
 \hspace{-0.07cm} \left( \alpha^{(\Side)}(\xi_2)\right)^{j_1} \hspace{-0.05cm} \dd^{\ell-j_1} \hspace{-0.07cm}\left( N_{j_2}^{p_2-j_1,
 2\sm-j_1}(\xi_2)\right) \,  M_\ell^{p_2,\sm} (\xi_1),
\end{equation}
where   
\begin{equation*}  \label{eq:M012}
 \gamma_{j_1} = h^{-j_1} \prod_{\rho=0}^{j_1-1}(p_2-\rho) \quad {\rm and} \quad
  M_\ell^{p_2,\sm}(\xi) = \sum_{j=\ell}^\sm \frac{ {j \choose \ell} h^\ell}{\prod_{\rho=0}^{\ell-1} (p_2-\rho)} N_j^{p_2,\sm}(\xi) , \quad \ell = 0,1,\ldots,\sm.
\end{equation*}

For a \emph{boundary edge} $\Gamma^{(i)}$, $i \in \mathcal{I}_{\Gamma}^{B}$, with $\overline{\Gamma^{(i)}} \subseteq \overline{\Omega^{(i_0)}}$, $i_0 \in \mathcal{I}_{\Omega}$, 
we assume 
that its closure $\overline{\Gamma^{(i)}}$ 
is given by $\ab{F}^{(i_0)}(\{0 \} \times [0,1]) $. The construction of the boundary 
edge subspace $\W^{\sS}_{\Gamma^{(i)}}$ depends on whether $\overline{\Gamma^{(i)}}$ contains two, one or no boundary vertex~$\ab{\Xi}^{(i)}$, $i \in \mathcal{I}_{\Xi}^{B}$, of patch valency one.
\paragraph{Two boundary vertices of patch valency one} The edge subspace $\W^\sS_{\Gamma^{(i)}}$ is equal to
\begin{align*} 
\W^\sS_{\Gamma^{(i)}} & =  \Span \left\{ {\phi}_{\Omega^{(i_0)}} (N_{j_1,j_2}^{\ab{p}_1,\ab{\sm}}) \; |\; j_2= 2\sm+1-j_1,\ldots, n_{p_1}+j_1-(2\sm+2);\;
j_1=0,1,\ldots,\sm \right\}.
\end{align*}
\paragraph{One boundary vertex of patch valency one} Let $\ab{F}^{(i_0)}(\ab{0})$ be the boundary vertex of patch valency one. Then, we have 
\begin{align*} 
\W^\sS_{\Gamma^{(i)}} & =  \Span \left\{ {\phi}_{\Omega^{(i_0)}}(N_{j_1,j_2}^{\ab{p}_1,\ab{\sm}}) \; |\; j_2= 2\sm +1-j_1 ,\ldots, n_{p_1}-\sm-2;\;
j_1=0,1,\ldots,\sm \right\} \oplus \\
& \; \Span \left\{ 
{\phi}_{\Omega^{(i_0)}}({N}_{j_1}^{\,{p}_1,{\sm}} \, \overline{N}_{j_2}^{\,{p}_1,{\sm}}) \; |\; j_2= n_{p_1}-\sm-1 , \ldots , n_{p_1}+j_1-\sm-2  ;\;
j_1=1,\ldots,\sm \right\}.
\end{align*} 
\paragraph{No boundary vertex of patch valency one} The edge subspace $\W^\sS_{\Gamma^{(i)}}$ is given as
\begin{align*} 
& \W^\sS_{\Gamma^{(i)}}  =  \Span \left\{ {\phi}_{\Omega^{(i_0)}}(N_{j_1,j_2}^{\ab{p}_1,\ab{\sm}}) \; |\; j_2= \sm +1 ,\ldots, n_{p_1}-\sm -2;\;
j_1=0,1,\ldots,\sm \right\} \oplus \\
&  
 \Span \left\{
{\phi}_{\Omega^{(i_0)}}({N}_{j_1}^{\,{p}_1,{\sm}} \, \overline{N}_{j_2}^{\,{p}_1,{\sm}}) |\, j_2= \sm+1-j_1, \ldots , \sm , n_{p_1} - \sm-1, \ldots, n_{p_1}+j_1-\sm-2 ;
j_1=1,\ldots,\sm \right\}\hspace{-0.05cm}.
\end{align*}  

\subsubsection{The vertex subspace~$\W^{\sS}_{\ab{\Xi}^{(i)}}$}

Let us denote the patch valency of a vertex~$\ab{\Xi}^{(i)}$, $i \in \mathcal{I}_{\Xi}$, by $\nu_i$. The construction of the vertex subspace~$\W^{\sS}_{\ab{\Xi}^{(i)}}$ differs whether we have an inner or boundary vertex~$\ab{\Xi}^{(i)}$. 
Let us first briefly recall the case where $\ab{\Xi}^{(i)}$ is an \emph{inner vertex}, i.e.~$i \in \mathcal{I}_{\Xi}^{I}$, assuming that all patches $\Omega^{(i_\rho)}$, $\rho=0,1,\ldots,\nu_i-1$, around the vertex~$\bfm{\Xi}^{(i)}$, i.e. $\bfm{\Xi}^{(i)} = \cap_{\rho=0}^{\nu_i-1} \overline{\Omega^{(i_\rho)}}$, are parameterized as shown in Fig.~\ref{fig:multipatchCase} (right), and relabel the common edges $\overline{\Omega^{(i_\rho)}} \cap \overline{\Omega^{(i_{\rho+1})}}$, $\rho=0,1,\ldots,\nu_i-1$, by $\Gamma^{(i_{\rho+1})}$, where~$\rho$ is taken modulo~$\nu_i$. Then, the vertex subspace~$\W^{\sS}_{\bfm{\Xi}^{(i)}}$ is defined as the span of functions $\phi_{\bfm{\Xi}^{(i)}}(\mathcal{F}_0,\ldots,\mathcal{F}_{\nu_i-1})$,
where 
$$
\mathcal{F}_\rho = f_{i_\rho}^{\Gamma^{(i_\rho)}}
  +  f_{i_\rho}^{\Gamma^{(i_{\rho+1})}} \, - 
 f_{i_\rho}^{\Omega^{(i_\rho)}}, \quad \rho=0,1,\ldots,{\nu}_i -1,
$$
with the spline functions
\begin{align} \label{eq:g_vertex2} 
    f_{i_\rho}^{\Gamma^{(i_{\rho+\tau})}}(\xi_1,\xi_2)  & = \sum_{j_1=0}^\sm \sum_{j_2=0}^{2\sm -j_1} a^{\Gamma^{(i_{\rho+\tau})}}_{j_1,j_2} \, 
  f_{\Gamma^{(i_{\rho+\tau}); j_1,j_2}}^{(i_\rho)} (\xi_{2-\tau},\xi_{1+\tau}), \quad { a^{\Gamma^{(i_{\rho+\tau})}}_{j_1,j_2} \in \R},
  \quad \tau = 0,1, \nonumber \\[-0.3cm] 
  \\[-0.3cm]
 f_{i_\rho}^{\Omega^{(i_\rho)}}(\xi_1,\xi_2)  & =  \sum_{j_1=0}^{\sm} \sum_{j_2=0}^{\sm } a^{(i_\rho)}_{j_1,j_2} N_{j_1,j_2}^{\ab{p}_2,\ab{\sm}} 
 (\xi_1,\xi_2), \quad {a_{j_1,j_2}^{(i_\rho)} \in \R }, \nonumber
\end{align}
and with the spline functions~$f_{\Gamma^{(i_{\rho+\tau}); j_1,j_2}}^{(i_\rho)}$, $\tau=0,1$, given in~\eqref{eq:basisFunctionsGenericG}. 


Thereby, the isogeometric functions~
$\phi_{\bfm{\Xi}^{(i)}}(\mathcal{F}_0,\ldots,\mathcal{F}_{\nu_i-1})$ are determined via 
selections of the coefficients $a^{\Gamma^{(i_{\rho+\tau})}}_{j_1,j_2}$, $a_{j_1,j_2}^{(i_\rho)}$, which form a basis of the kernel of the homogeneous linear system
\begin{equation} \label{eq:vertex_homogeneous_system}
  \partial_{1}^{\ell_1}  
  \partial_{2}^{\ell_2} 
   \left( f_{i_\rho}^{\Gamma^{(i_{\rho+1})}} - f_{i_\rho}^{\Gamma^{(i_\rho)}} \right)  (\bfm{0}) = 0 \quad \mbox{ and } \quad
    \partial_{1^{}}^{\ell_1}
   \partial_{2^{}}^{\ell_2}
   \left( f_{i_\rho}^{\Gamma^{(i_{\rho+1})}} - f_{i_\rho}^{\Omega^{(i_\rho)}} \right)  (\bfm{0}) = 0, 
\end{equation} 
for $ 0 \leq \ell_1, \ell_2 \leq \sm$ and $\rho =0,1, \ldots, \nu_{i}-1$.
In our examples in Section~\ref{section_Numerical_examples}, the basis of the kernel of~\eqref{eq:vertex_homogeneous_system} will be constructed by means of the concept of minimal determining sets, cf. \cite{KaVi17a,LaSch07}.

Let now $\ab{\Xi}^{(i)}$ be a \emph{boundary vertex}, i.e. $i \in \mathcal{I}_{\Xi}^{B}$, assuming that the two boundary edges are labeled as $\Gamma^{(i_0)}$ and $\Gamma^{(i_{\nu_{i}})}$.
The construction of the boundary vertex subspace $\W^{\sS}_{\bfm{\Xi}^{(i)}}$ differs whether we have a boundary vertex~$\bfm{\Xi}^{(i)}$, $i \in \mathcal{I}_{\Xi}^B$, of patch valency~$\nu_i \geq 3$, $\nu_i=2$ or $\nu_i=1$:
\paragraph{A patch valency of $\nu_i \geq 3$}
The vertex subspace~$\W^{\sS}_{\ab{\Xi}^{(i)}}$ is constructed as for an inner vertex with the difference that the functions~$f_{\Gamma^{(i_0)};j_1,j_2}^{(i_0)}$ and $f_{\Gamma^{(i_{v_i})};j_1,j_2}^{(i_{v_{i}-1})}$ in~\eqref{eq:g_vertex2} for the patches~$\Omega^{(i_0)}$ and $\Omega^{(i_{v_i-1})}$ are just the standard B-splines $N_{j_1,j_2}^{\ab{p}_2,\ab{\sm}}$ for $j_1, j_2 =0,1,\ldots,\sm$, and 
splines ${N}_{j_1}^{\,{p}_1,{\sm}} \, \overline{N}_{j_2-\sm}^{\,{p}_1,{\sm}}$ for indices $j_1=0,1,\ldots,\sm-1$ and $j_2=\sm+1,\ldots,2\sm-j_1$.

\paragraph{A patch valency of $\nu_i =2$}
Assuming that $\bfm{\Xi}^{(i)} = \ab{F}^{(i_0)}(\ab{0}) = \ab{F}^{(i_1)}(\ab{0})$, $i_0,i_1 \in \mathcal{I}_{\Omega}$, that the common edge $\overline{\Omega^{(i_0)}} \cap \overline{\Omega^{(i_1)}}$ is denoted by $\overline{\Gamma^{(j_0)}}$, $j_0 \in \mathcal{I}_{\Gamma}^{I}$, and that the two patches~$\Omega^{(i_0)}$ and $\Omega^{(i_1)}$ are parameterized as shown in Fig.~\ref{fig:multipatchCase} (left), then the vertex subspace~$\mathcal{W}^{\sS}_{\bfm{\Xi}^{(i)}}$ is directly constructed as
\begin{align*} \label{eq:spaceWXB}
\mathcal{W}^{\sS}_{\bfm{\Xi}^{(i)}} =&  \Span \left\{ \widetilde{\phi}_{\bfm{\Xi}^{(i)}; j_1,j_2} \; |\;   
j_1=0,1,\ldots,3 \sm; \; j_2 = \begin{cases}
                         0,1,\ldots, 2\sm-j_1 & \mbox{if }j_1 \leq 2\sm \\
                         0,1,\ldots, 3 \sm-j_1 & \mbox{if }j_1 > 2\sm 
                                \end{cases}
 \; \right \} ,
\end{align*}
with the isogeometric functions
\begin{equation*} \label{eq:phi_vertex_2}
 \widetilde{\phi}_{\bfm{\Xi}^{(i)};j_1,j_2} 
 = 
    \begin{cases} 
     \phi_{\Gamma^{(j_0)}}
     (\g_{\Gamma^{(j_0)}; j_1,j_2}^{(i_0)},\g_{\Gamma^{(j_0)}; j_1,j_2}^{(i_1)})
     & \mbox{if }j_1=0,\ldots ,\sm\\ 
     {\phi}_{\Omega^{(i_0)}} (\overline{N}_{j_1-\sm}^{\,{p}_1,{\sm}} \, {N}_{j_2}^{\,{p}_1,{\sm}})
     & \mbox{if }j_1=\sm+1,\ldots,2\sm\\
     {\phi}_{\Omega^{(i_1)}}(\overline{N}_{j_1-2\sm}^{\,{p}_1,{\sm}} \, {N}_{j_2}^{\,{p}_1,{\sm}})
     & \mbox{if }j_1=2 \sm +1,\ldots,3 \sm.
    \end{cases}
\end{equation*}

\paragraph{A patch valency of $\nu_i =1$}
Assuming that $\bfm{\Xi}^{(i)}= \ab{F}^{(i_0)}(\ab{0})$, $i_0 \in \mathcal{I}_{\Omega}$, the vertex subspace~$\mathcal{W}^{\sS}_{\bfm{\Xi}^{(i)}}$ is simply 
given as 
$$
\mathcal{W}^{\sS}_{\bfm{\Xi}^{(i)}} =  \Span \left\{ {\phi}_{\Omega^{(i_0)}}(N_{j_1,j_2}^{\ab{p}_1,\ab{\sm}}) \; |\;   
j_1,j_2 =  0,1,\ldots, 2 \sm; \; j_1+j_2 \leq 2\sm \right\}.
$$



\section{The selection of collocation points} \label{subsec:collocationPoints}

In this section, we will present two possible sets of collocation points for the use of the $C^s$-smooth mixed degree isogeometric spline space~$\W^{\sS}$ from Section~\ref{sec:problem_statement} as discretization space for the isogeometric multi-patch collocation methods introduced in Section~\ref{sec:collocation} for solving the Poisson's and the biharmonic equation. Thereby, the global collocation points will be selected via local collocation points $\bfm{\loc}^{(i)}_{\bfm{j}} = (\loc^{(i)}_{j_1},\loc^{(i)}_{j_2}) \in [0,1]^2 $ for each patch~$\overline{\Omega^{(i)}}$, $i \in \mathcal{I}_{\Omega}$, as $\ab{F}^{(i)}(\bfm{\loc}^{(i)}_{\bfm{j}}) \in \overline{\Omega^{(i)}}$. For the sake of simplicity, we will use the same local collocation points $\bfm{\loc}_{\bfm{j}} = (\loc_{j_1},\loc_{j_2}) \in [0,1]^2$ for each patch~$\overline{\Omega^{(i)}}$, $i \in \mathcal{I}_{\Omega}$. If more local collocation points from different patches define the same global collocation point, then we keep just the one for the 
patch~$\overline{\Omega^{(i)}}$ with the 
lowest patch index~$i$, cf.~\eqref{eq:patch_index_selection}, to prevent the possible repetition of global collocation points. 
We will study now the two different choices of 
local collocation points~$\bfm{\loc}_{\bfm{j}}$ which will then define the global collocation points in the multi-patch isogeometric collocation techniques in Section~\ref{sec:collocation}.  For both sets of collocation points, we will restrict ourselves to the case $s=2$ and $s=4$.

\subsection{The mixed degree Greville points} The first choice of collocation points is dedicted to the generalization of the Greville points for the underlying spline spaces $\mathcal{S}^{\ab{p}_i,\ab{\sm}}_h([0,1]^2)$, $i=1,2$, to the mixed degree Greville points for the mixed degree underlying spline space $\mathcal{S}^{(\ab{p}_1,\ab{p}_2),\ab{\sm}}_h([0,1]^2)$. Let us first recall the univariate Greville points for the space $\mathcal{S}^{{p}_i,\sm}_h([0,1])$
which are defined as 
\[
 \loc_{j}^{p_i,\sm}=\frac{t_{j+1}^{p_i,\sm} + \ldots + t_{j+p_i}^{p_i,\sm}}{p_i}, \quad  j \in \{0,1,\ldots, n_{p_i}-1 \}, \; i=1,2.
\]
We now have to prescribe one Greville point to each of the basis functions from the spaces $\mathcal{S}_1([0,1]^2)$,  $\overline{\mathcal{S}}_1([0,1]^2)$ and $\mathcal{S}_2([0,1]^2)$. Clearly, we prescribe the collocation points $(\loc_{j_1}^{p_i,\sm},\loc_{j_2}^{p_i,\sm})$ to the functions $N_{(j_1,j_2)}^{\ab{p}_i,\ab{\sm}}$ from the spaces $\mathcal{S}_1([0,1]^2)$ and $\mathcal{S}_2([0,1]^2)$. It remains to find suitable collocation points $(\overline{\loc}_{j_1}^{\,p_1,\sm},\overline{\loc}_{j_2}^{\,p_1,\sm})$ for the truncated basis functions 
from the space $\overline{\mathcal{S}}_1([0,1]^2)$. 
More precisely, we have to determine collocation points $\overline{\loc}_{j}^{\,p_1,\sm}$ for the truncated univariate B-splines $\overline{N}_{j}^{p_1,\sm}$.  
For the case when $j \in \{\sm+1,\ldots,n_{p_1}-\sm-2\}$, we clearly take $\overline{\loc}_{j}^{\,p_1,\sm} = \loc_{j}^{p_1,\sm}$. For the other two cases, namely where the truncation has an effect, i.e., for $j \in \{1,\ldots,\sm\}$ and for $j \in \{n_{p_1}-\sm-1,\ldots,n_{p_1}-2\}$, we select them in a symmetric way. Therefore, let us explain the selection of collocation points only for the index set $
\{1,\ldots,\sm\}$. In order to avoid the overlapping of collocation points corresponding to different sets of basis functions, we have to satisfy the following relations
\begin{equation} \label{eq:GrevilleRelations}
\loc_{0}^{p_2,\sm}<\cdots<\loc_{\sm}^{p_2,\sm}< 
\overline{\loc}_{1}^{\,p_1,\sm} < \cdots <
\overline{\loc}_{\sm}^{\,p_1,\sm} < 
\loc_{\sm+1}^{p_1,\sm} < \cdots < \loc_{n_{p_1}-\sm-2}^{p_1,\sm}.
\end{equation}
The following proposition gives an appropriate selection. 
\begin{prop}
    Let $s \in \{2,4\} $, and let
    \begin{equation} \label{eq:truncatedGreville}
        \overline{\loc}_{i}^{\,p_1,\sm}=\loc_{\sm+i}^{p_2,\sm} \quad {\rm and} \quad 
        \overline{\loc}_{\frac{\sm}{2}+i}^{\,p_1,\sm} =\loc_{\frac{\sm}{2}+i}^{p_1,\sm}, \quad i=1,\ldots,\frac{\sm}{2}.
    \end{equation}
     Then relations \eqref{eq:GrevilleRelations} are fulfilled. 
\end{prop}
\begin{proof}
    We only have to prove that $\overline{\loc}_{\frac{\sm}{2}}^{\,p_1,\sm} < \overline{\loc}_{\frac{\sm}{2}+1}^{\,p_1,\sm}$, since all other relations are trivially fulfilled. By \eqref{eq:truncatedGreville}, this is equivalent to show that
    $\loc_{\sm+\frac{\sm}{2}}^{p_2,\sm} < \loc_{\frac{\sm}{2}+1}^{p_1,\sm}$, which follows by
     \begin{eqnarray*}
    && \hspace{-0.7cm} \loc_{\sm+\frac{\sm}{2}}^{p_2,\sm} = \frac{ 
    \left(p_2-\frac{3\sm}{2}\right)0 + (p_2-\sm)h + \left(-p_2 +\frac{5\sm}{2}\right) 2h}{p_2}  = \frac{ (\sm+1)h + \left(\frac{\sm}{2}-1\right)2h}{2\sm+1} = \frac{2\sm-1}{2\sm+1} h < h,\\
    && \hspace{-0.7cm} \loc_{\frac{\sm}{2}+1}^{p_1,\sm}  =  \frac{\left(p_1-\frac{\sm}{2}-1\right)0 + (p_1-\sm)\left(h+2h+\ldots +\left(\frac{\sm}{2}+1\right)h\right)}{p_1} = 
    \frac{h\left( 1+2+\ldots + (\frac{\sm}{2}+1) \right)}{\sm+1} \\
     && \;\; =  h \frac{\left(\frac{\sm}{2}+1 \right) \left( \frac{\sm}{2}+2\right))}{2(\sm+1)} = h + h \frac{\sm(\sm-2)}{8(\sm+1)} \geq h.
    \end{eqnarray*}
\end{proof}
To summarize, we obtain the mixed degree Greville points for the mixed degree underlying spline space $\mathcal{S}_{h}^{(\ab{p}_1,\ab{p}_2),\ab{\sm}}([0,1]^2)$ by collecting the Greville points corresponding to the basis functions from the spaces $\mathcal{S}_{1}([0,1]^2)$, $\overline{\mathcal{S}}_{1}([0,1]^2)$ and $\mathcal{S}_{2}([0,1]^2)$, cf. Fig.~\ref{fig:Greville59} for the case where all edges of $[0,1]^2$ correspond to inner edges of the multi-patch domain~$\overline{\Omega}$.
\begin{figure}
    \centering
    \includegraphics[scale=0.33]{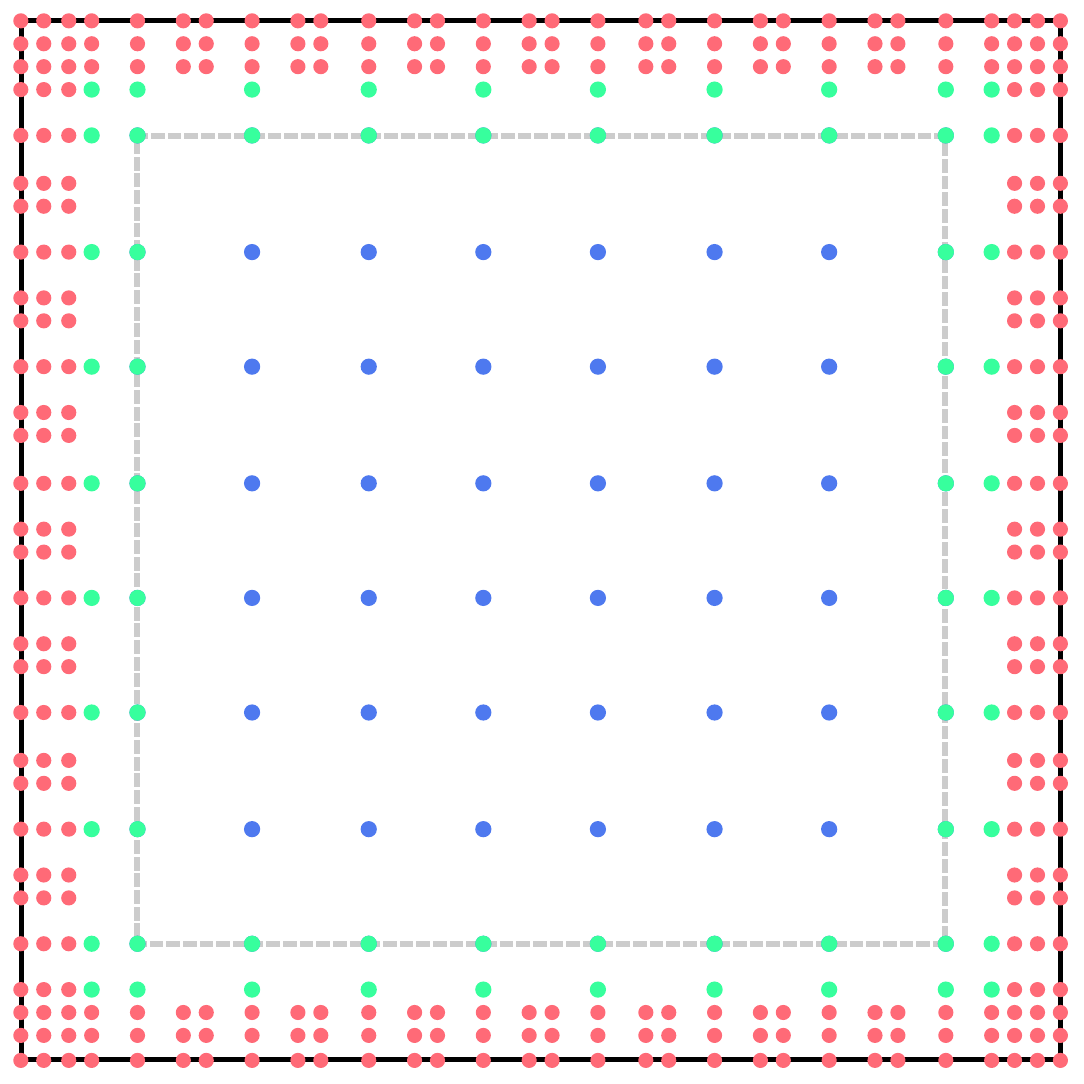}
    \hskip3em
    \includegraphics[scale=0.33]{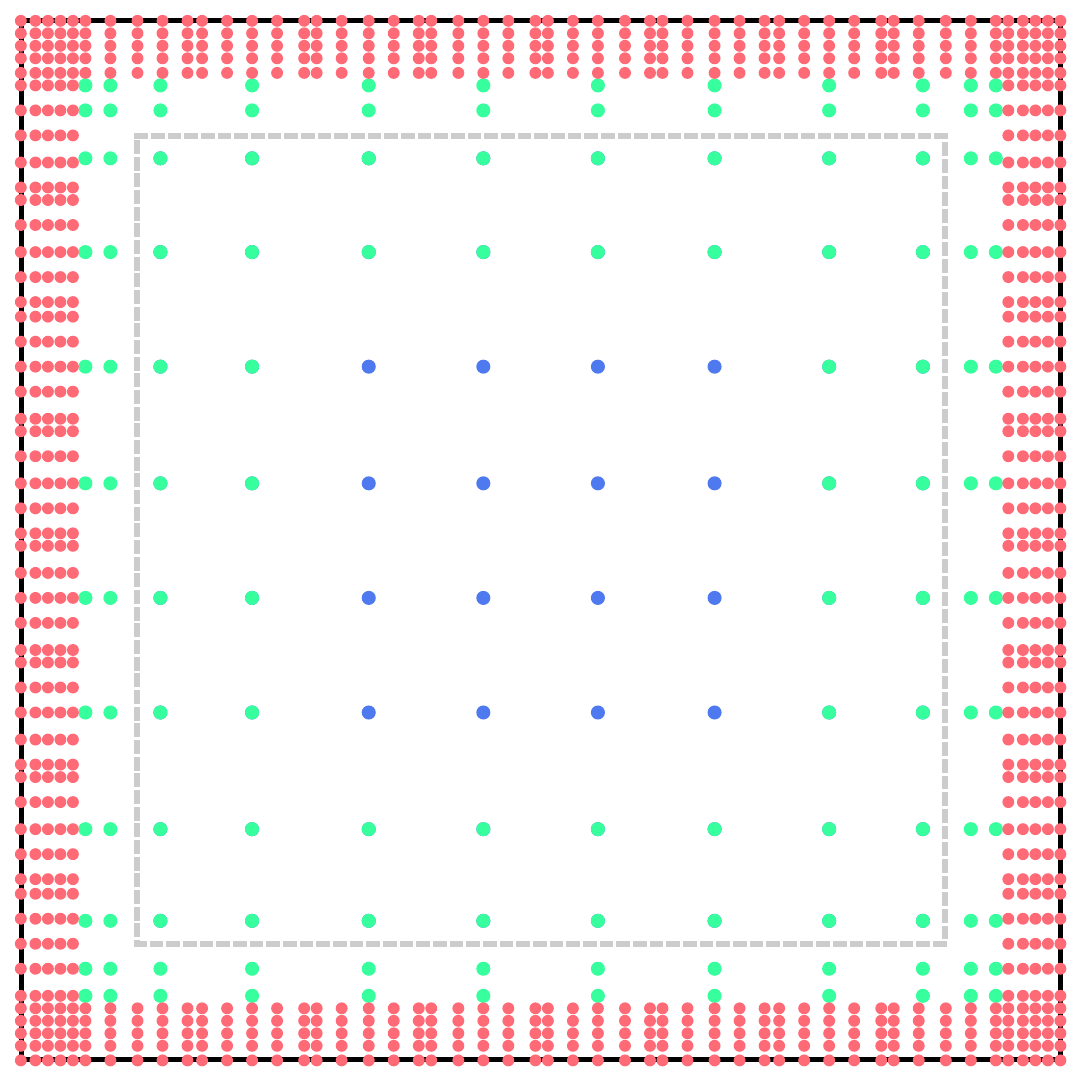}
    \caption{The mixed degree Greville points for the spaces $\mathcal{S}_{1/9}^{(\ab{3},\ab{5}),\ab{2}}$ (left) and $\mathcal{S}_{1/9}^{(\ab{5},\ab{9}),\ab{4}}$ (right). The blue, green and red collocation points correspond to the basis functions from the spaces $\mathcal{S}_1([0,1]^2)$, $\overline{\mathcal{S}}_1([0,1]^2)$ and $\mathcal{S}_2([0,1]^2)$, respectively. 
    The dashed gray lines denote the boundary of the subdomain $[h,1-h]^2$. 
    }
    \label{fig:Greville59}
\end{figure}

\subsection{The mixed degree superconvergent points}
\label{subsec:superconvergent}
The second choice of collocation points will be called mixed degree superconvergent points. Superconvergent points are the roots of the corresponding Galerkin residual $D^\sm(u-u_h)$. Estimates of these points can be computed by solving the \sm-th order ordinary differential equation  
\begin{equation*} \label{eq:biharmonic_equation}
u^{(\sm)}(x) =f(x), \quad x \in (0,1), \quad {\rm and} \quad \left(\frac{d^\ell}{dx^\ell}u\right)(0)=\left(\frac{d^\ell}{dx^\ell}u\right)(1)=0, \quad \ell=0,1,\ldots,\frac{\sm}{2}-1,
\end{equation*}
with some particular function~$f$ as demonstrated in~\cite{GomezLorenzisVariationalCollocation}. Superconvergent points have been mainly studied for the one-patch case and for splines with maximal regularity, i.e.~$\sm=p-1$, e.g.~in \cite{IsoCollocMethods2010, SuperConvergent2015,GomezLorenzisVariationalCollocation,MonSanTam2017} for second order PDEs, and e.g. in~\cite{GomezLorenzisVariationalCollocation,RealiGomez2015,Maurin2018} for fourth order problems. Since the use of splines with maximal regularity is not possible for the multi-patch case, superconvergent points for splines with a lower regularity than $\sm=p-1$ have been investigated, namely for the case of second order problems in~\cite{KaVi20} and for the case of fourth order problems in~\cite{KaKoVi24}. Below, we will focus on the description of the selection of the mixed degree superconvergent points for the mixed degree underlying spline spaces $\mathcal{S}^{(\ab{3},\ab{5}),\ab{2}}_h([0,1]^2)$ and $\mathcal{S}^{(\ab{5},\ab{9}),\ab{4}}_h([0,1]^2)$, which will allow to solve the Poisson's and the biharmonic equation, respectively, over multi-patch domains by means of isogeometric collocation.

Let us first consider the case $\mathcal{S}^{(\ab{3},\ab{5}),\ab{2}}_h([0,1]^2)$. The superconvergent points for the spaces $\mathcal{S}^{3,2}_h([0,1])$ and $\mathcal{S}^{5,2}_h([0,1])$ on each knot span with respect to the reference interval~$[-1,1]$ are given as
\begin{equation} \label{eq:superconvergent}
\left\{ \pm \frac{1}{\sqrt{3}} \right\} \quad {\rm and} \quad
\left\{ \pm \sqrt{\frac{1}{{15}} (6-\sqrt{21})} ,\, \pm \sqrt{\frac{1}{{15}} (6+\sqrt{21})} \right\},
\end{equation}
respectively, cf.~\cite{KaVi20}. In order to get the appropriate cardinality for the set of the mixed degree superconvergent points, we will take clustered subsets of the superconvergent points~\eqref{eq:superconvergent}. The goal is to replace the mixed degree Greville points for the space $\mathcal{S}_{h}^{(\ab{3},\ab{5}),\ab{2}}([0,1]^2)$, which already have the correct number, with the same number of superconvergent points. More precisely, we will replace the $k^2$ Greville points $(\loc_{j_1}^{3,2},\loc_{j_2}^{3,2})$ which are contained in the subdomain $[h,1-h]^2$, and the remaining Greville points 
which belong to the subdomain $[0,1]^2 \backslash [h,1-h]^2$ with the same number of superconvergent points each. In the following, we will describe this procedure for the case where all edges of $[0,1]^2$ correspond to inner edges of the multi-patch domain~$\overline{\Omega}$. The modifications to all other cases 
are straightforward and follow the concept visualized in Fig.~\ref{fig:SpacesMixed}.

First, we have to select $k$ among $2(k-1)$ superconvergent points on the interval $[h,1-h]$ for the space $\mathcal{S}_h^{3,2}([0,1])$. For that we skip the redundant $k-2$ points in a clustered way, cf.~Fig.~\ref{fig:SuperconvergentMixed}. To select a proper subset of superconvergent points on the subdomain $[0,1]^2 \backslash [h,1-h]^2$, we first have to construct the clustered set of superconvergent points for the space $\mathcal{S}_h^{5,2}([0,1])$. Since the boundary points of $[0,1]$ are not contained in the set of all superconvergent points, they have to be added to the set to be able to impose boundary conditions. Therefore, we have to omit $4(k+1)-(3k+6)+2=k$ points in order to get a clustered subset with the proper cardinality. This can be done as explained in \cite[Fig.~3]{KaVi20}. This leads again to $\sm+1+\frac{\sm}{2}=4$ clustered superconvergent points each on the intervals $[0,h]$ and $[1-h,1]$, as in the case of the mixed degree Greville points. Finally, we construct the tensor product of these points and select the ones which are related to the basis functions from the space $\mathcal{S}_2([0,1]^2)$, and the ones that correspond to the basis functions from the space $\overline{\mathcal{S}}_1([0,1]^2)$, see Fig.~\ref{fig:SuperconvergentMixed} (left). 

\begin{figure}[h!]
        \centering
    \includegraphics[scale=0.34]{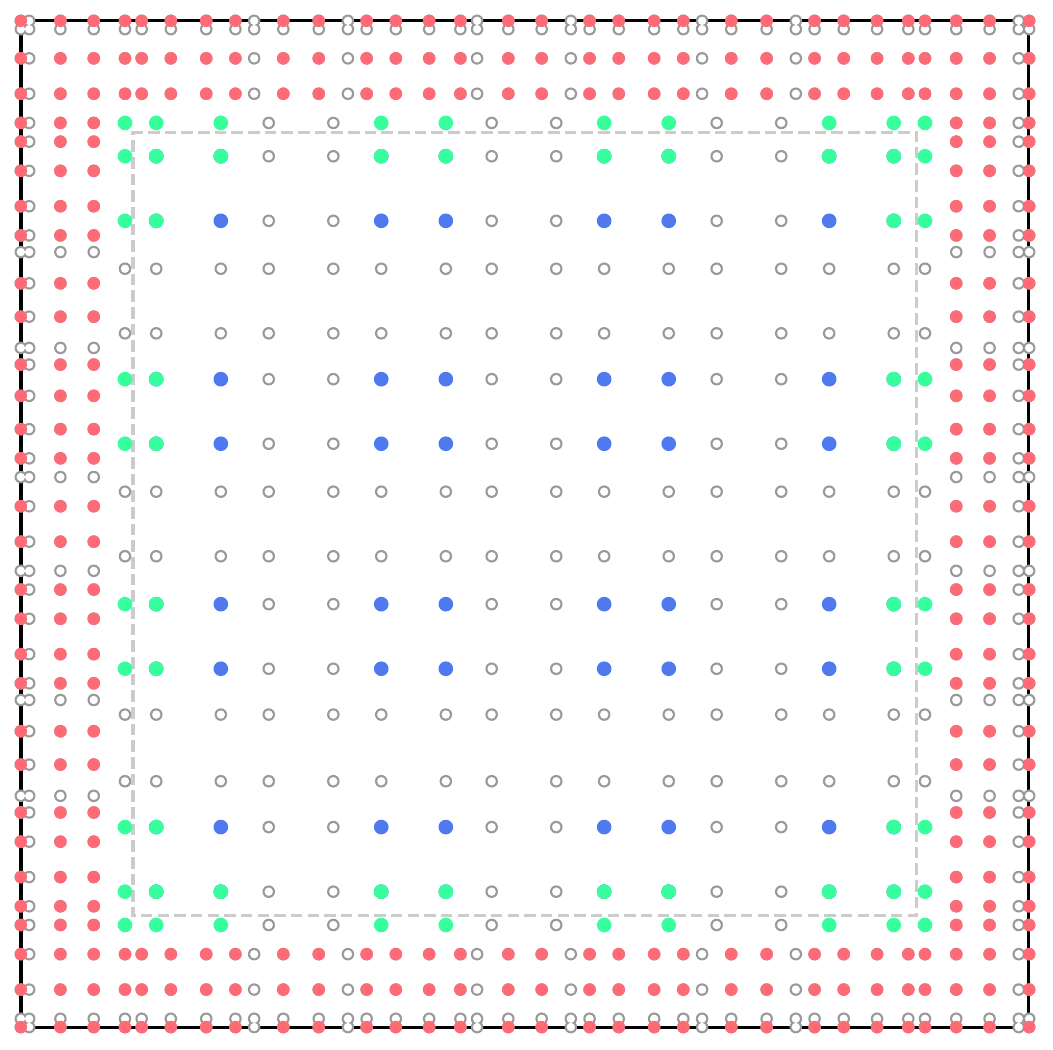}
    \hskip3em
    \includegraphics[scale=0.34]{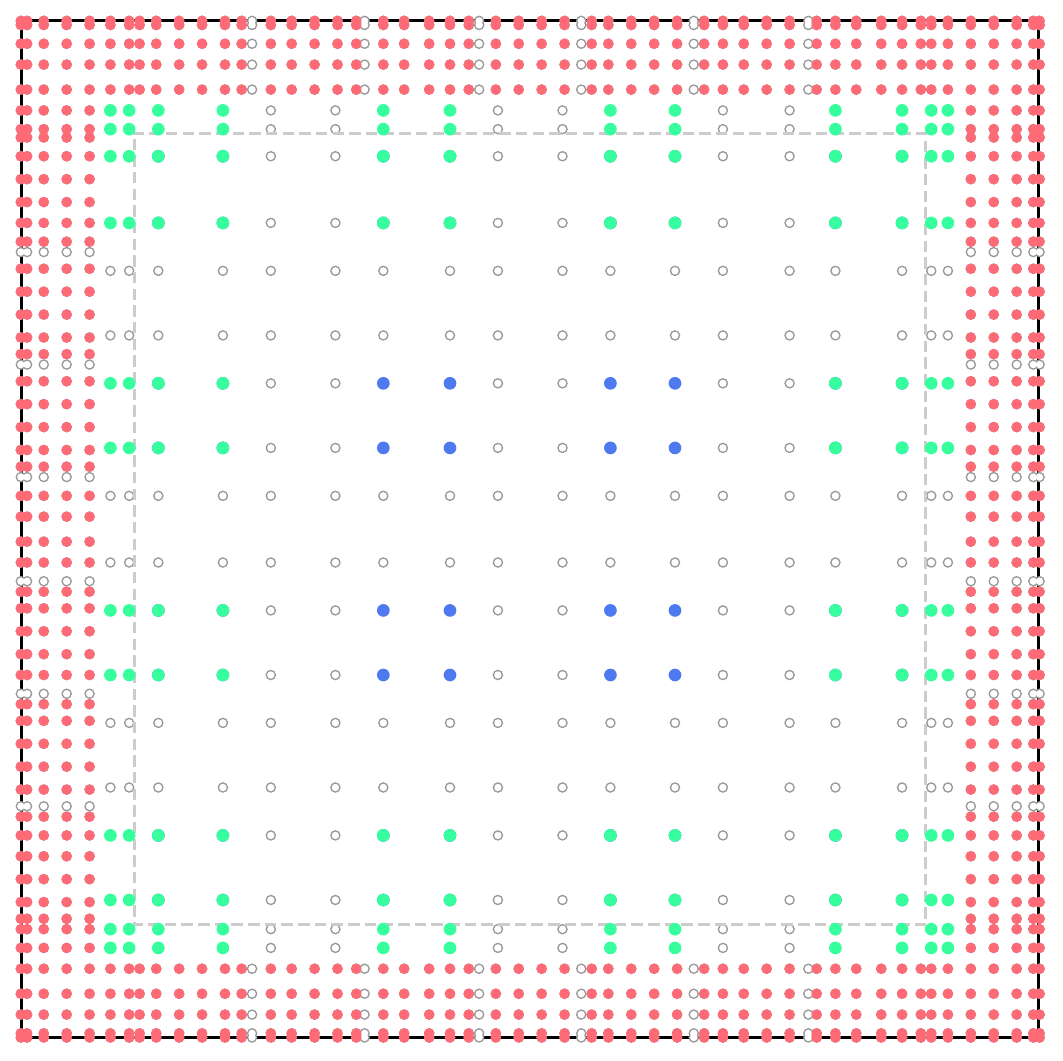}
    \caption{All mixed degree superconvergent points for the spaces $\mathcal{S}_{1/9}^{(\ab{3},\ab{5}),\ab{2}}$ (left) and $\mathcal{S}_{1/9}^{(\ab{5},\ab{9}),\ab{4}}$ (right). The blue, green and red points are the clustered mixed degree superconvergent points that correspond to the basis functions from the spaces $\mathcal{S}_1([0,1]^2)$, $\overline{\mathcal{S}}_1([0,1]^2)$ and $\mathcal{S}_2([0,1]^2)$, respectively. The dashed gray lines denote the boundary of the subdomain $[h,1-h]^2$.  
    }
    \label{fig:SuperconvergentMixed}
\end{figure}

Let us consider now the case $\mathcal{S}^{(\ab{5},\ab{9}),\ab{4}}_h([0,1]^2)$. For this case, the superconvergent points for the spaces $\mathcal{S}^{5,4}_h([0,1])$ and $\mathcal{S}^{9,4}_h([0,1])$ on each knot span with respect to the reference interval~$[-1,1]$ are given as 
$$
\left\{ \pm \frac{1}{\sqrt{3}} \right\} \quad {\rm and} \quad
\left\{ \pm 0.2072795685478027, \pm 0.5963052503103114, \pm 0.9098737952346008  \right\},
$$
respectively, cf.~\cite{KaKoVi24}, where the latter ones are the roots of the sextic polynomial $4823 x^6 - 5915 x^4 + 1665 x^2 -61$. Therefore, the superconvergent points on the subdomain $[h,1-h]^2$ are exactly the same as in the case $\sm=2$ above. To select a proper subset of superconvergent points for the space $\mathcal{S}_h^{\ab{9},\ab{4}}([0,1]^2)$ on the subdomain $[0,1]^2 \backslash [h,1-h]^2$, we again first construct the clustered set of superconvergent points for the space $\mathcal{S}_h^{9,4}([0,1])$ and add boundary points of $[0,1]$ as presented in \cite[Fig.~3]{KaKoVi24}.
In this way, we again have $\sm+1+\frac{\sm}{2}=7$ superconvergent points each on the intervals $[0,h]$ and $[1-h,1]$ as in the case of mixed degree Greville points. Finally, we construct the tensor product of these points and select the ones which correspond to the basis functions from the spaces $\mathcal{S}_2([0,1]^2)$ and $\overline{\mathcal{S}}_1([0,1]^2)$, see Fig.~\ref{fig:SuperconvergentMixed} (right). 
Similarly as above, this covers the case where all edges of $[0,1]^2$ correspond to inner edges of the multi-patch domain~$\overline{\Omega}$. For all other cases, we follow the concept from Fig.~\ref{fig:SpacesMixed}.


\section{Numerical examples}  \label{section_Numerical_examples}

On the basis of several examples, we will demonstrate the potential of our isogeometric collocation method for solving the Poisson's and the biharmonic equation over planar multi-patch domains, and will further study the convergence under $h$-refinement for both types of considered collocation points, namely for the mixed degree Greville points and for the mixed degree superconvergent points. 


In all examples 
(Examples~\ref{ex:collocationPoison}--\ref{ex:Lshape_squareSystem}) the right side functions $g, g_1$ for the Poisson's equation \eqref{eq:Poisson} and the right side functions $g, g_1, g_2$ for the biharmonic equation \eqref{eq:biharmonic} will be obtained from the exact solution  
\begin{equation}  \label{eq:exactSolution}
 u(x_1,x_2)= \cos\left(x_1\right) \sin\left(x_2\right). 
\end{equation}
We will use the $C^2$ and $C^4$-smooth isogeometric spline spaces~$\W^2$ and $\W^4$ with the mesh sizes $h=1/2^i$, $i=3,\ldots,6$, denoted by $\W_h^2$ and $\W_h^4$, as discretization spaces for computing a $C^2$-smooth and $C^4$-smooth approximation of the solution of the Poisson's equation~\eqref{eq:Poisson} and of the biharmonic equation~\eqref{eq:biharmonic}, respectively. The quality of the obtained approximants $u_h \in \mathcal{W}^\sm_h$ with respect to the exact solution \eqref{eq:exactSolution} will be investigated by calculating the relative errors with respect to the $L^2$-norm and $H^1$-seminorm as well as 
with respect to equivalents (cf.~\cite{BaDe15})
of the $H^m$-seminorms, $2 \leq m \leq \sm$, $\sm=2,4$, i.e.
\begin{equation} \label{eq:eqiuv2seminorms}
\frac{\| u-u_h\|_{L^2}}{\| u \|_{L^2}}, \quad
\frac{| u-u_h|_{H^1}}{| u |_{H^1}}, \quad
 \frac{\| \Delta u- \Delta u_h\|_{L^2}}{\| \Delta u \|_{L^2}}, \quad
\frac{ \| \nabla (\Delta u) - \nabla(\Delta u_h)\|_{L^2}}{\| \nabla(\Delta u )\|_{L^2}}, \quad 
\frac{\| \Delta^2 u- \Delta^2 u_h\|_{L^2}}{\| \Delta^2 u \|_{L^2}}.
\end{equation}

In Examples~\ref{ex:collocationPoison} and \ref{ex:collocationBiharmonic}, we will first consider three bilinearly parameterized (multi-patch) domains. The first one will be a bilinear one-patch domain (Domain A), cf. Fig.~\ref{fig:bilinearDomains} (left), with the vertices 
\begin{equation} \label{eq:DomainA}
\ab{\Xi}_{0} =\begin{pmatrix} 3/4 \\ 3/4 \end{pmatrix}, \; \ab{\Xi}_{1} = \begin{pmatrix} 15/4 \\ 9/4 \end{pmatrix}, \; \ab{\Xi}_{2} =\begin{pmatrix} 21/20 \\ 27/10 \end{pmatrix}, \;
\ab{\Xi}_{3} = \begin{pmatrix} 3 \\ 3 \end{pmatrix}.
\end{equation}
The second one will be a bilinear three-patch domain (Domain B) having the seven vertices
\begin{equation*} \label{eq:verticesThreePatch}
\begin{split}
\ab{\Xi}_0  & = \begin{pmatrix} 12 \\ 8 \end{pmatrix}, \; \ab{\Xi}_1 = \begin{pmatrix} 14 \\ 44/5  \end{pmatrix},  \;\ab{\Xi}_2 = \begin{pmatrix} 61/5 \\ 109/10  \end{pmatrix},  \;\ab{\Xi}_3 = \begin{pmatrix} 10 \\ 91/10   \end{pmatrix},  
\\
\ab{\Xi}_4  & = \begin{pmatrix} 87/10  \\ 36/5  \end{pmatrix},  \;\ab{\Xi}_5 = \begin{pmatrix} 111/10 \\ 6  \end{pmatrix},\;
\ab{\Xi}_6  = \begin{pmatrix} 71/5 \\  13/2 \end{pmatrix}, \;
\end{split}
\end{equation*}
which define the parameterizations~$\ab{F}^{(i)}$ of the individual patches $\Omega^{(i)}$ as
\begin{equation}  \label{eq:three_patch_domain_param}
\ab{F}^{(i)}(\ab{\xi}) = \ab{\Xi}_0 (1-\xi_1) (1-\xi_2) + \ab{\Xi}_{2i+1} \,(1-\xi_1) \xi_2  + \ab{\Xi}_{2i+2} \,\xi_1 \xi_2 + \ab{\Xi}_{2i+3} \,\xi_1 (1-\xi_2), 
\quad i=0,1,2, 
\end{equation} 
where $\ab{\Xi}_{7} := \ab{\Xi}_{1}$, cf. Fig.~\ref{fig:bilinearDomains} (middle). 
The third bilinear multi-patch domain (Domain C) will be the five-patch \emph{Paper Plane domain}~
visualized in Fig.~\ref{fig:bilinearDomains} (right), which is determined by the eleven points 
\begin{equation*} \label{eq:verticesThreePatch}
\begin{split}
\quad \ab{\Xi}_0  & = \begin{pmatrix} 14/5 \\ 3 \end{pmatrix}, \; \ab{\Xi}_1 = \begin{pmatrix} 31/10 \\ 4 \end{pmatrix},  \;\ab{\Xi}_2 = \begin{pmatrix} 14/5 \\ 26/5 \end{pmatrix},  \;\ab{\Xi}_3 = \begin{pmatrix} 5/2 \\ 4  \end{pmatrix},  
\;
\ab{\Xi}_4   = \begin{pmatrix} 1 \\ 7/5 \end{pmatrix},  \;\ab{\Xi}_5 = \begin{pmatrix} 1 \\ 1 \end{pmatrix},\\
\ab{\Xi}_6  & = \begin{pmatrix} 17/10 \\ 9/10 \end{pmatrix}, \;
\ab{\Xi}_7   = \begin{pmatrix} 14/5 \\ 11/10 \end{pmatrix}, \;
\ab{\Xi}_8 = \begin{pmatrix} 39/10 \\ 9/10 \end{pmatrix}, \;
\ab{\Xi}_9 = \begin{pmatrix} 23/5 \\ 1 \end{pmatrix}, \;
\ab{\Xi}_{10} = \begin{pmatrix} 23/5 \\ 7/5 \end{pmatrix}, \;
\end{split}
\end{equation*}
and by the parameterizations $\ab{F}^{(i)}$ of the individual patches $\overline{\Omega^{(i)}}$ given as
\begin{equation}  \label{eq:five_patch_domain_param}
\ab{F}^{(i)}(\ab{\xi}) = \ab{\Xi}_0 (1-\xi_1) (1-\xi_2) + \ab{\Xi}_{2i+1} \,(1-\xi_1) \xi_2  + \ab{\Xi}_{2i+2} \,\xi_1 \xi_2 + \ab{\Xi}_{2i+3} \,\xi_1 (1-\xi_2), 
\quad i=0,\ldots,4, 
\end{equation} 
with $\ab{\Xi}_{11} := \ab{\Xi}_{1} $.

\begin{figure}[htb!]
    \centering
     \includegraphics[scale=0.5]{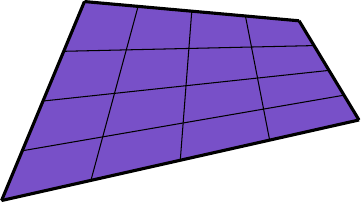}
     \hskip2.5em
    \includegraphics[scale=0.27]{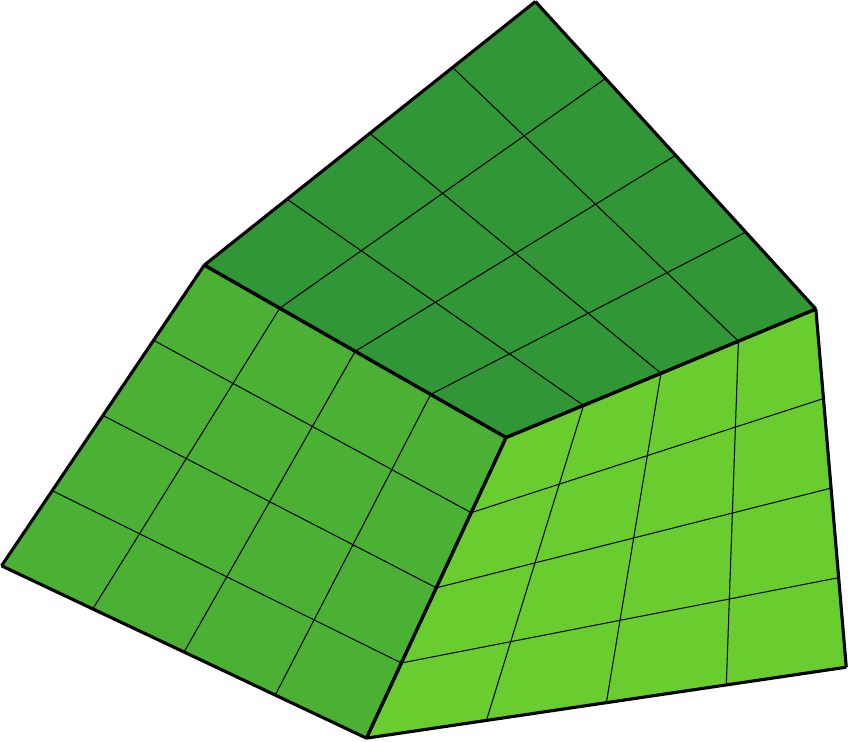}
     \hskip3.em
    \includegraphics[scale=0.37]{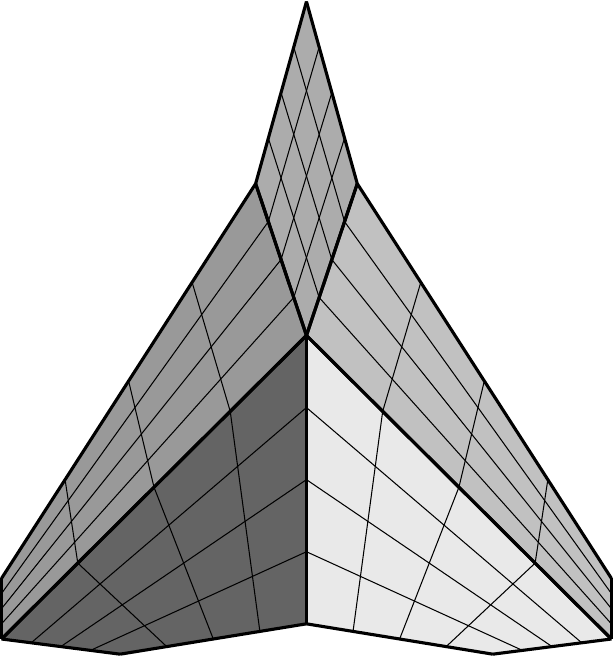}
    \\[0.3cm]
    \includegraphics[scale=0.13]{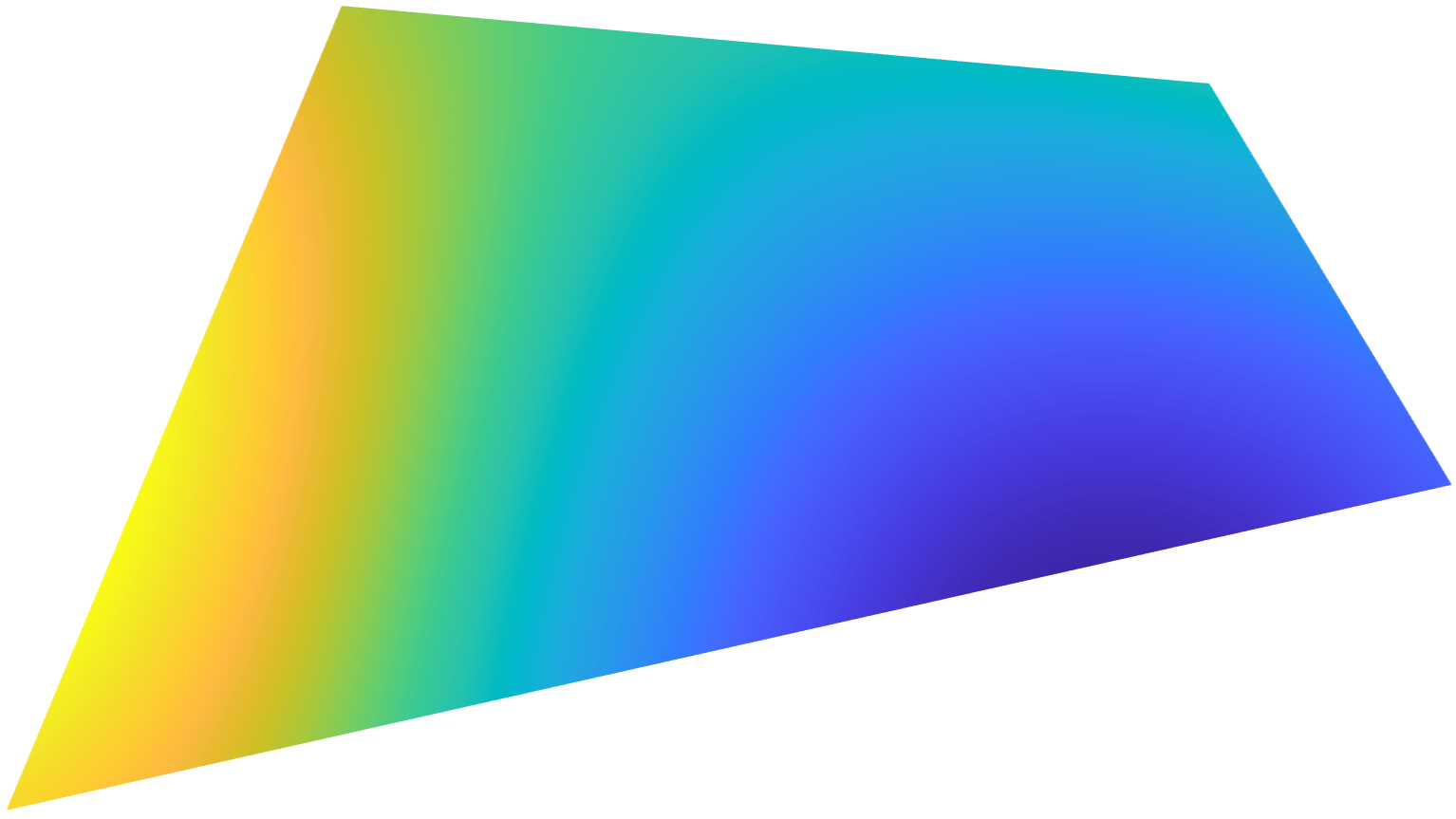}
    \hskip2.5em
    \includegraphics[scale=0.13]{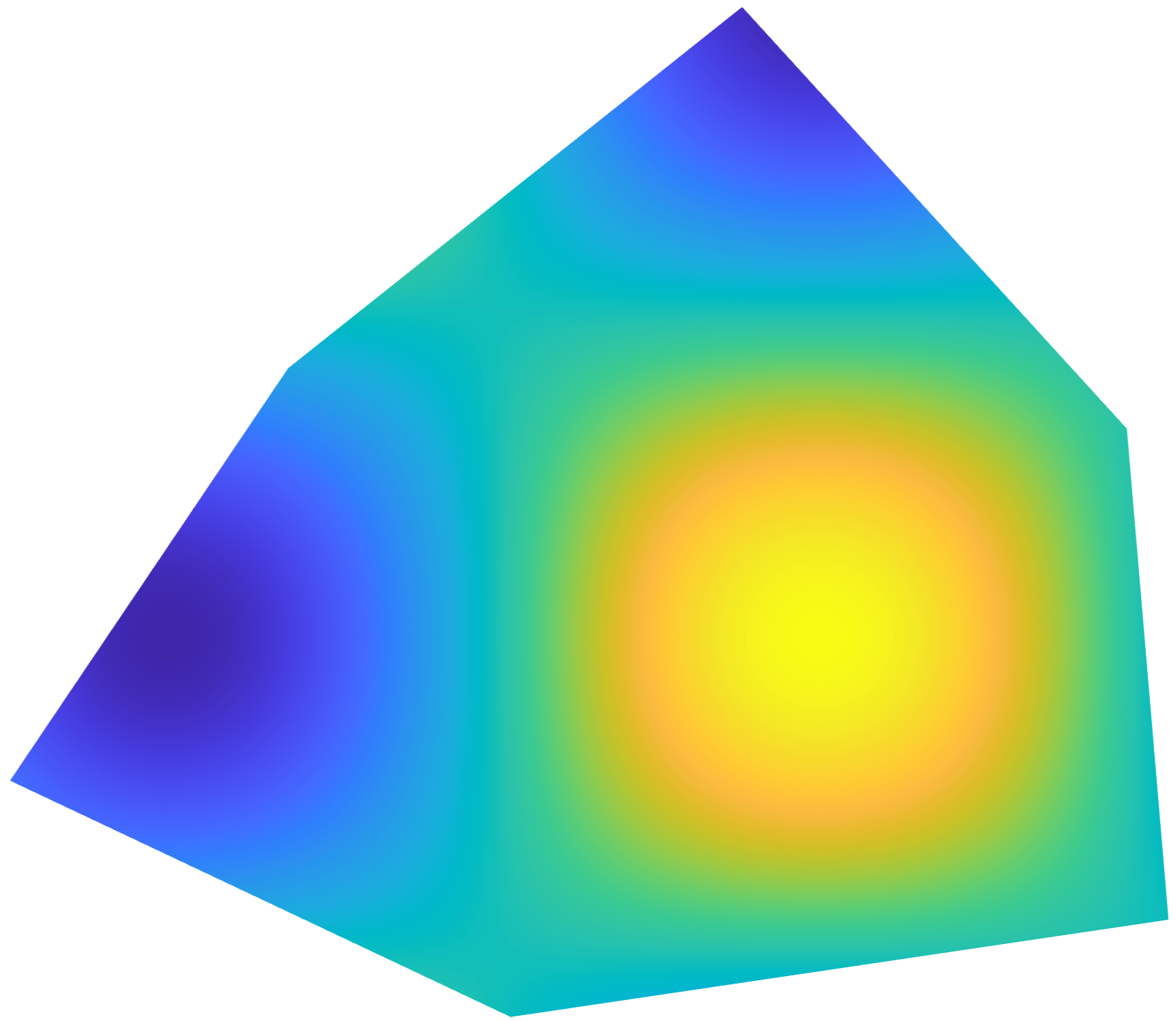}
     \hskip3.em
    \includegraphics[scale=0.19]{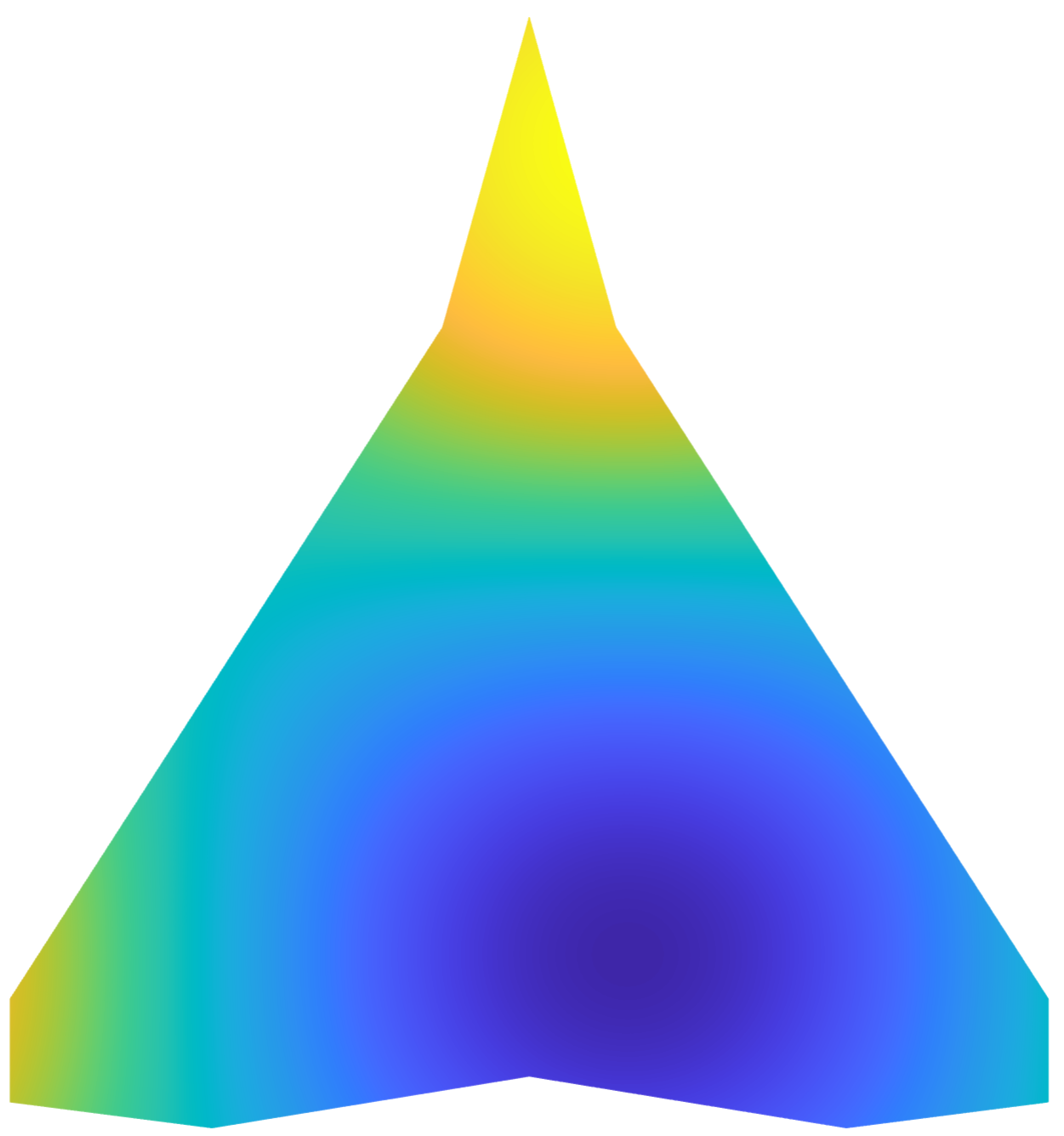} 
\caption{Top row: Plots of the bilinear one-patch Domain A, given in \eqref{eq:DomainA}, left, bilinear three-patch
Domain B, presented in \eqref{eq:three_patch_domain_param}, middle, and bilinear five-patch Domain C, given in  \eqref{eq:five_patch_domain_param}, right.  Bottom row: The three multi-patch domains together with the exact solution \eqref{eq:exactSolution}.}
    \label{fig:bilinearDomains}
\end{figure}

In addition, we will solve in Examples~\ref{ex:collocationPoison} and \ref{ex:collocationBiharmonic} the Poisson's and the biharmonic equation, respectively, over so-called bilinear-like~$G^s$ multi-patch geometries \cite{KaVi17c, KaVi20b}. This class of multi-patch parameterizations possesses for each inner edge~$\Gamma^{(i)}$, $i \in \mathcal{I}_{\Gamma}^I$, with $\overline{\Gamma^{(i)}}=\overline{\Omega^{(i_0)}} \cap \overline{\Omega^{(i_1)}}$, $i_0,i_1 \in \mathcal{I}_{\Omega}$, like the bilinear multi-patch domains, linear gluing functions $\alpha^{(i,\Side)},\beta^{(i,\Side)}, \Side \in \{\LL,\RR \}$, but allows in contrast to bilinear multi-patch domains the modeling of multi-patch domains with curved domain boundaries, too, see e.g.~\cite{KaSaTa19a, KaVi20b}. More precisely, a multi-patch parameterization 
is called \emph{bilinear-like $G^{s}$} if for any two neighboring 
patches~$\Omega^{(i_0)}$ and $\Omega^{(i_1)}$, $i_0,i_1 \in \mathcal{I}_{\Omega}$, with the inner edge~$\overline{\Gamma^{(i)}} = \overline{\Omega^{(i_0)}} \cap 
\overline{\Omega^{(i_1)}}$, $i \in \mathcal{I}_{\Gamma}^I$, 
and corresponding geometry mappings $\ab{F}^{(i_0)}$ and $\ab{F}^{(i_1)}$, 
parameterized as in Fig.~\ref{fig:multipatchCase} (left),
there exist linear functions 
$\alpha^{(i,\Side)},\beta^{(i,\Side)}, \Side \in \{\LL,\RR \}$, such that
 \begin{equation*}   \label{eq:FC}
 \ab{F}_\ell^{(i,\LL)}(\xi) = \ab{F}_\ell^{(i,\RR)}(\xi) =: \ab{F}_\ell^{(i)}(\xi), \quad \ell =0,1,\ldots,\sm,
 \end{equation*}
 with
 \begin{equation*}   \label{eq:FC2}
 \ab{F}_\ell^{(i,\Side)}(\xi) = \left(\alpha^{(i,\Side)}(\xi)\right)^{-\ell}\, \partial_1^\ell \ab{F}^{(\Side)}(0,\xi) - \sum_{j=0}^{\ell-1} {\ell \choose j} 
 \left(\frac{\beta^{(i,\Side)}(\xi)}{\alpha^{(i,\Side)}(\xi)}\right)^{\ell-j}  \dd^{\ell-j} \ab{F}_j^{(i,\Side)}(\xi),
 \;\; \Side\in \{\LL,\RR\}.
 \end{equation*} 
In detail, the first considered domain beyond bilinear (multi-patch) domains will be the biquadratic polynomial one-patch domain (Domain D), cf. Fig.~\ref{fig:bilinearLikeDomains}~(left), with the parameterization
\begin{equation} \label{eq:DomainD}
\ab{F}(\ab{\xi}) = \sum_{j_1,j_2=0}^{2} \ab{c}_{j_1,j_2} N^{\ab{2},\ab{2}}_{j_1,j_2} (\ab{\xi}),
\end{equation}
where
\begin{equation*}
\begin{split}
\ab{c}_{0,0}=\begin{pmatrix} 18/5 \\ 36/5 \end{pmatrix}&, \; \ab{c}_{0,1}= \begin{pmatrix} 6 \\ 39/5 \end{pmatrix}, \; \ab{c}_{0,2}=\begin{pmatrix} 42/5 \\ 36/5 \end{pmatrix}, \\ 
\ab{c}_{1,0}=\begin{pmatrix} 21/5 \\ 6  \end{pmatrix}&, \; \ab{c}_{1,1}=\begin{pmatrix} 6 \\ 6 \end{pmatrix}, \; \ab{c}_{1,2}=\begin{pmatrix} 39/5 \\ 6 \end{pmatrix}, \\ 
\ab{c}_{2,0}=\begin{pmatrix} 18/5 \\ 24/5 \end{pmatrix}&, \; \ab{c}_{2,1}=\begin{pmatrix} 6 \\ 21/5 \end{pmatrix}, \; \ab{c}_{2,2}=\begin{pmatrix} 42/5 \\ 24/5 \end{pmatrix}.
\end{split}
\end{equation*}
This domain is trivially $G^s$-smooth since it does not possess an inner edge.
\begin{figure}[htb!]
    \centering
    \includegraphics[scale=0.25]{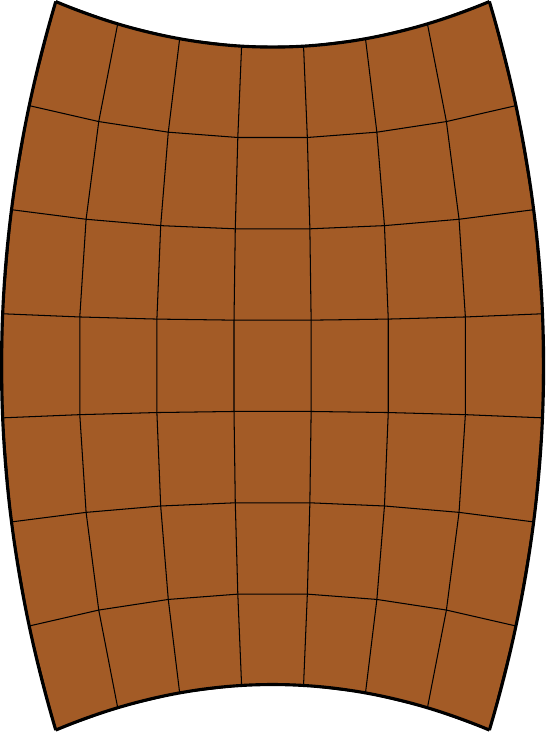}
     \hskip3.em
    \includegraphics[scale=0.32]{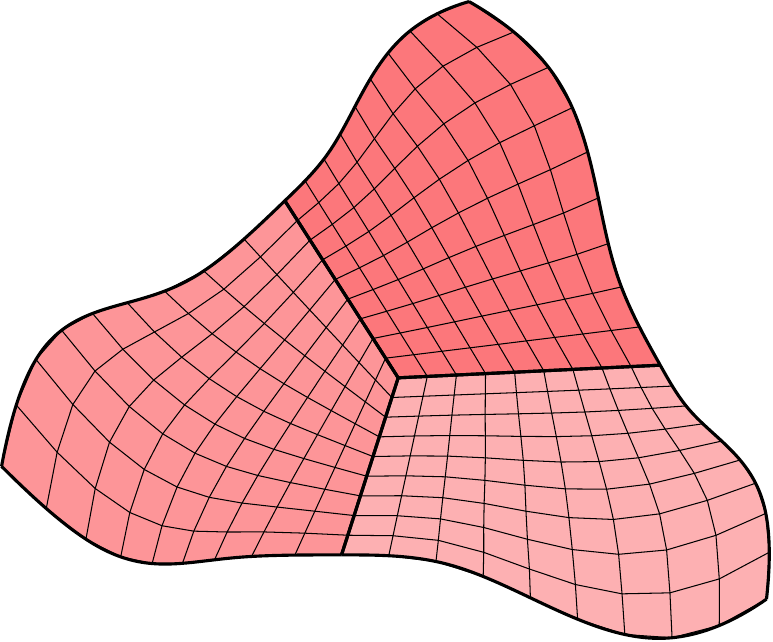}
     \hskip2.em
    \includegraphics[scale=0.32]{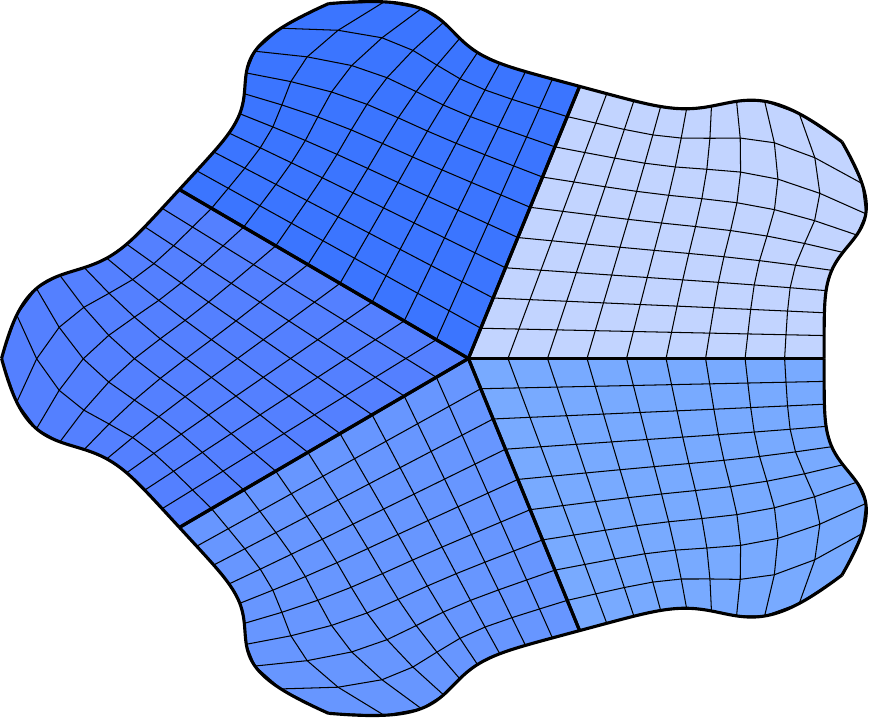}
    \\[0.6cm]
    \hskip-0.5em
    \includegraphics[scale=0.125]{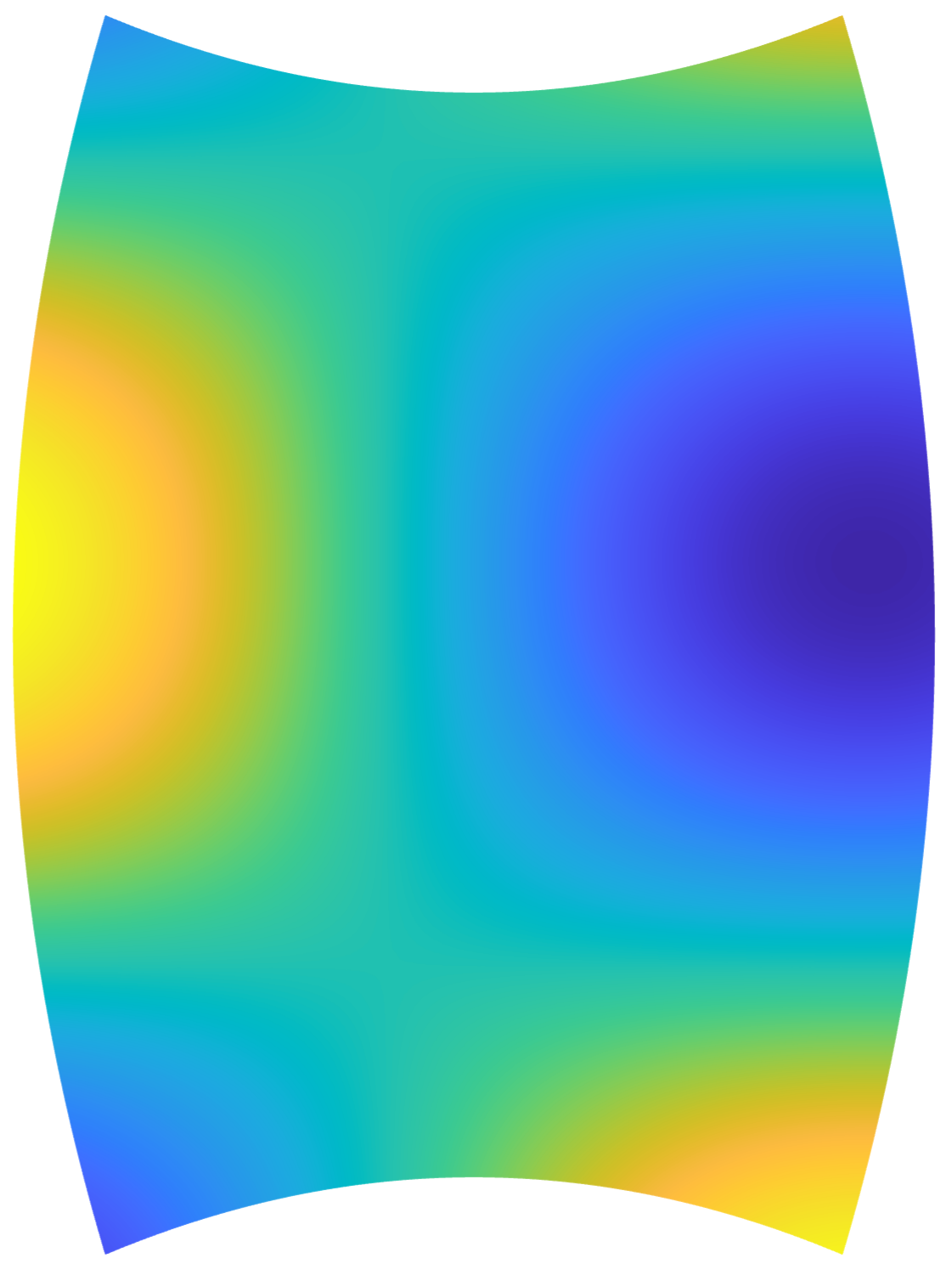}
    \hskip3.em
    \includegraphics[scale=0.16]{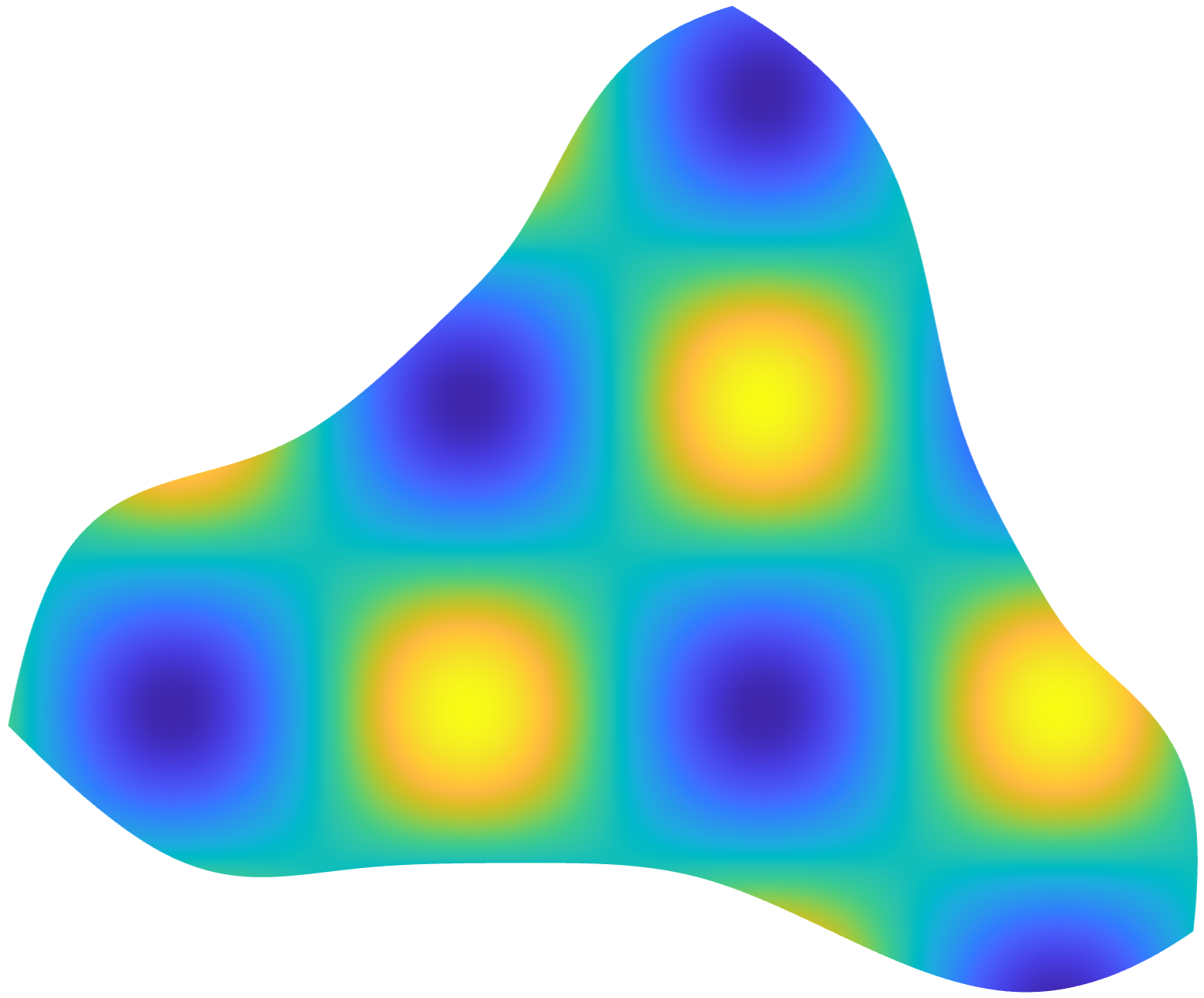}
     \hskip2.em
    \includegraphics[scale=0.15]{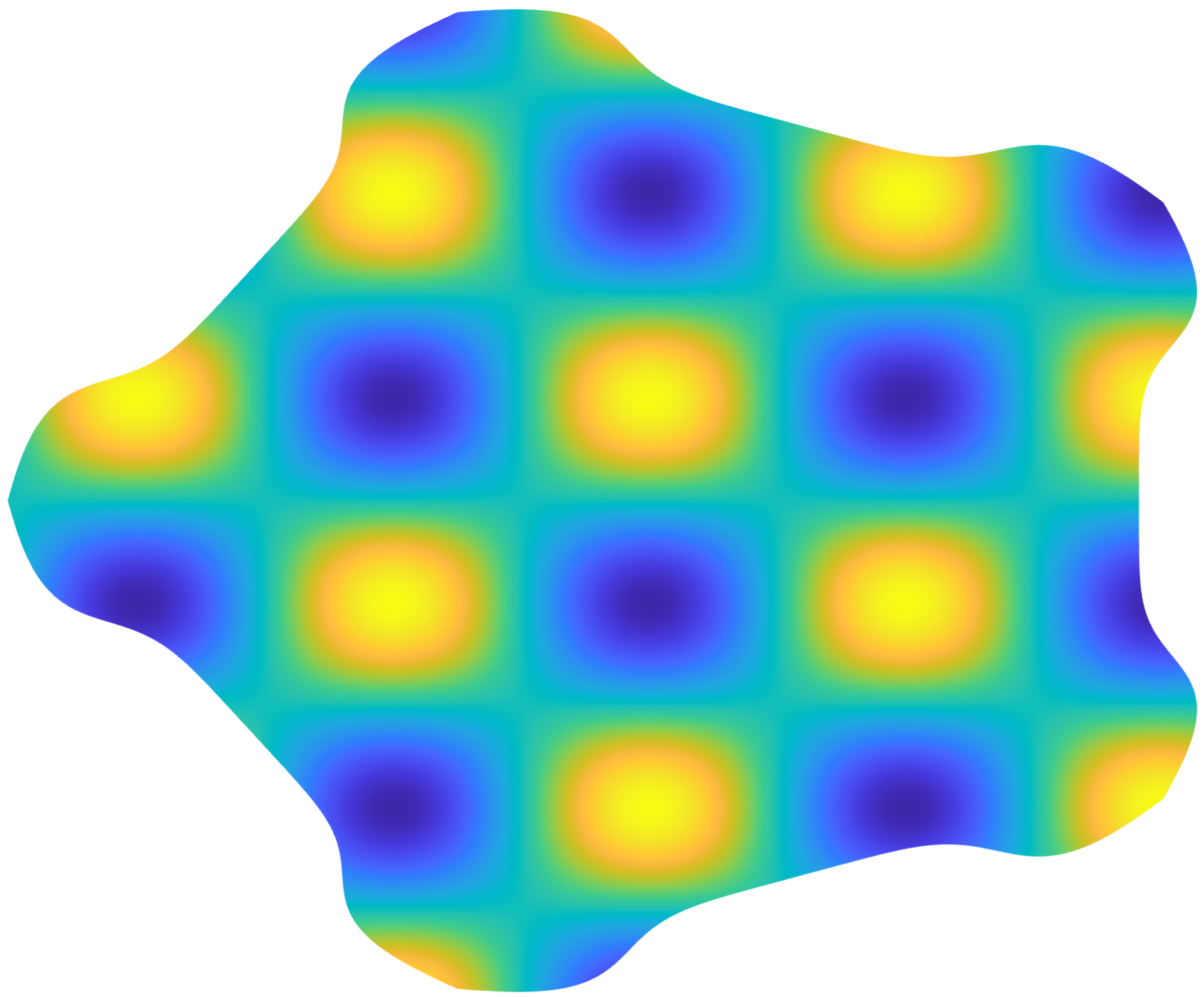}
\caption{Top row: Plots of the biquadratic one-patch Domain D, given in \eqref{eq:DomainD}, left, the bilinear-like $G^2$ three-patch Domain E, specified in \eqref{eq:three_patch_domain_bilinearLike}, middle, and the bilinear-like $G^4$ five-patch Domain F, given in \eqref{eq:five_patch_domain_bilinearLike}, right. Bottom row: The three multi-patch domains together with the exact solution \eqref{eq:exactSolution}.}
    \label{fig:bilinearLikeDomains}
\end{figure}
The second domain will be the bilinear-like $G^2$ three-patch domain (Domain E) presented in Fig.~\ref{fig:bilinearLikeDomains} (middle), where the individual patches are given by bicubic geometry mappings  
from the space $\mathcal{S}_{1/4}^{\ab{3}, \ab{2}}([0,1]^2) \times \mathcal{S}_{1/4}^{\ab{3}, \ab{2}}([0,1]^2)$, parameterized as 
\begin{equation}  \label{eq:three_patch_domain_bilinearLike}
 \ab{F}^{(i)} (\ab{\xi}) = \sum_{j_1,j_2=0}^{6} \ab{c}^{(i)}_{j_1,j_2} N^{\ab{3},\ab{2}}_{j_1,j_2}(\ab{\xi}), \quad i=0,1,2,
\end{equation}
and with the control points~$\ab{c}^{(i)}_{j_1,j_2}$ stated in \cite[Table 1]{KaKoVi24b}.
As third domain beyond bilinear multi-patch domains, we will consider the bilinear-like $G^4$ five-patch domain (Domain F) shown in Fig.~\ref{fig:bilinearLikeDomains} (right), which represents a possible \emph{Screw knob star domain}. Such domains 
do not only have an attractive design, but due to their shape a high torque can be achieved.
The individual patches are given by biquintic geometry mappings  
from the space $\mathcal{S}_{1/4}^{\ab{5}, \ab{4}}([0,1]^2) \times \mathcal{S}_{1/4}^{\ab{5}, \ab{4}}([0,1]^2)$, parameterized as 
\begin{equation}  \label{eq:five_patch_domain_bilinearLike}
 \ab{F}^{(i)} (\ab{\xi}) = \sum_{j_1,j_2=0}^{8} \ab{c}^{(i)}_{j_1,j_2} N^{\ab{5},\ab{4}}_{j_1,j_2}(\ab{\xi}), \quad i=0,\ldots,4,
\end{equation}
where the control points~$\ab{c}^{(i)}_{j_1,j_2}$ are listed in \ref{sec:AppendixPoints}.

A bilinear two-patch domain (Domain~G) will be handled in Example~\ref{ex:Lshape_squareSystem}. This domain will be given by the vertices \begin{equation} \label{eq:verticesTwoPatch}
\ab{\Xi}_0  = \begin{pmatrix} 3 \\ 3 \end{pmatrix}, \; \ab{\Xi}_1 = \begin{pmatrix}3 \\ 6 \end{pmatrix},  \; \ab{\Xi}_2 = \begin{pmatrix}1 \\ 6 \end{pmatrix},  \; \ab{\Xi}_3 = \begin{pmatrix}1 \\ 1  \end{pmatrix}, \;
\ab{\Xi}_4  = \begin{pmatrix} 7 \\ 1 \end{pmatrix},  \; \ab{\Xi}_5 = \begin{pmatrix} 7 \\ 3 \end{pmatrix},
\end{equation}
and will represent an L-shape domain, cf.~Fig.~\ref{fig:DomainG}.
\begin{figure}[htb!]
    \centering
    \includegraphics[scale=0.45]{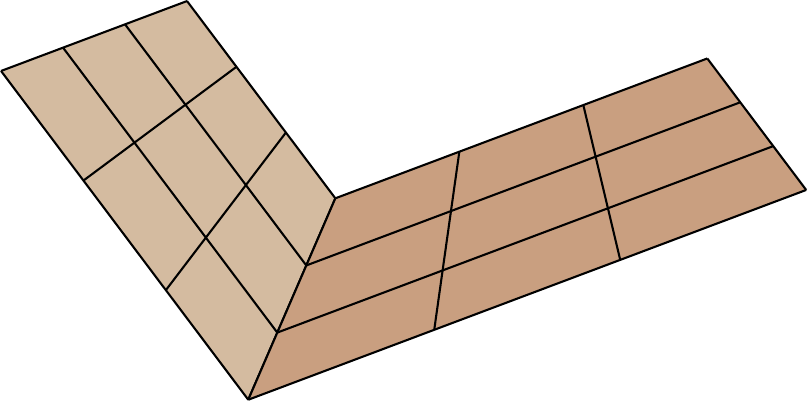}
    \hskip5em
    \includegraphics[scale=0.225]{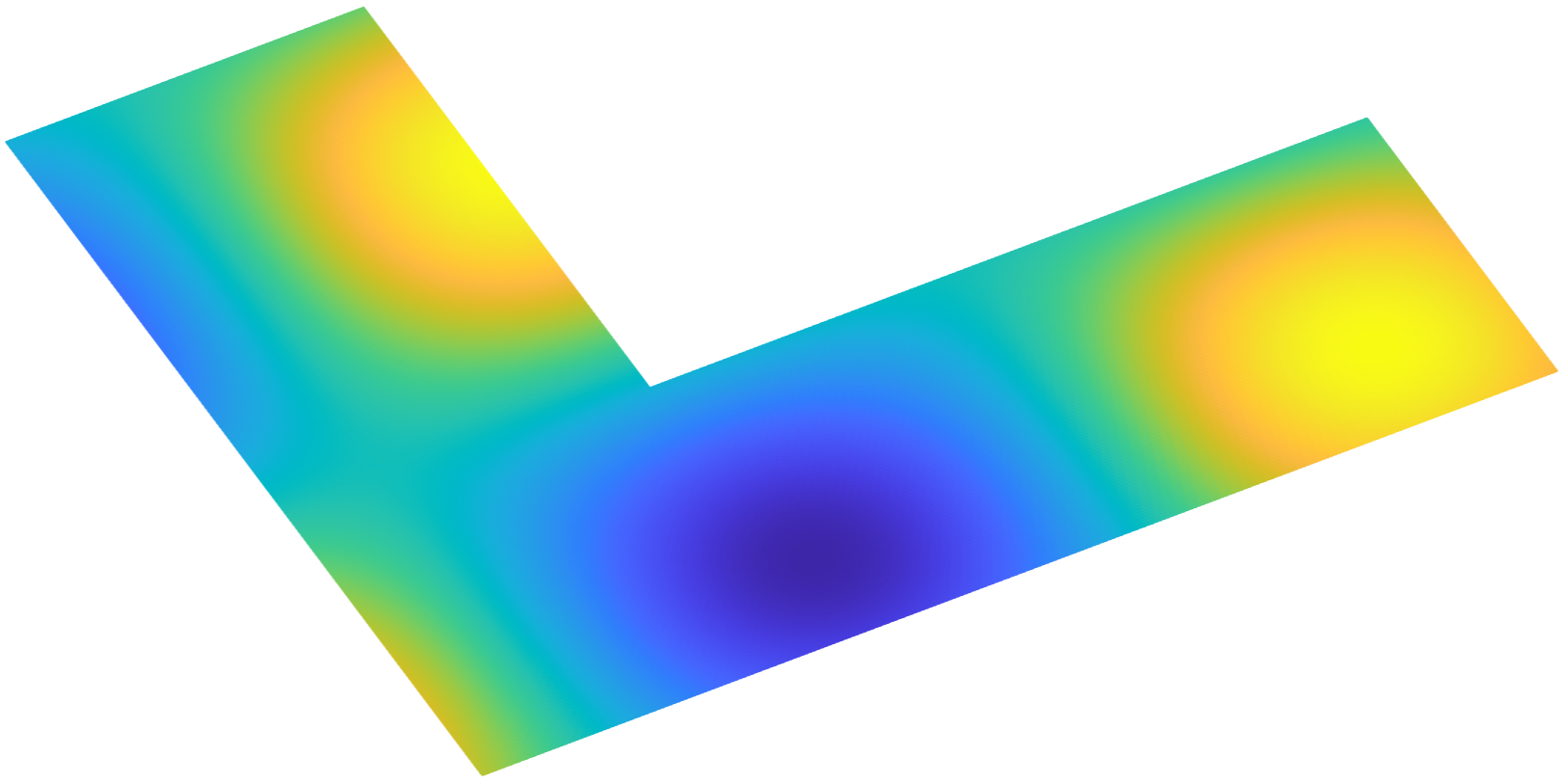}
\caption{The L-shape Domain~G, given by \eqref{eq:verticesTwoPatch}, left, and the domain together with the exact solution \eqref{eq:exactSolution}, right.
}
    \label{fig:DomainG}
\end{figure}

In Examples~\ref{ex:collocationPoison} and \ref{ex:collocationBiharmonic}, we will select the mixed degree Greville and mixed degree superconvergent points as described in Section~\ref{subsec:collocationPoints}, and will further perform the separation of the collocation points into inner and boundary collocation points as explained in Section~\ref{sec:collocation}. This leads in case of a one-patch domain (Domain~A and D) to a square linear system and in case of a multi-patch domain (Domains B--C and E--F) to a slightly overdetermined linear system. In Example~\ref{ex:Lshape_squareSystem}, we will propose a first strategy to select a suitable subset of the mixed degree superconvergent collocation points from Subsection~\ref{subsec:superconvergent} to impose a square linear system even in the multi-patch case. More precisely, we will study the case of a two-patch domain demonstrated on the basis of the L-shape Domain G. The extension to the case of multi-patch domains with extraordinary vertices is beyond the scope of the paper, and will be considered in a possible future work. 

\begin{ex} \label{ex:collocationPoison}
In this example, we solve the Poisson's equation over the Domains~A--F by using the $C^2$-smooth discretization space~$\mathcal{W}^2_h$ 
based on the underlying mixed degree spline space $\mathcal{S}_h^{(\ab{3}, \ab{5}),\ab{2}}([0,1]^2)$. For the one-patch Domains~A and D, we treat all four boundary edges as inner edges to make the usage of the underlying mixed degree spline space also for these domains necessary. The resulting relative errors~\eqref{eq:eqiuv2seminorms} with respect to the $L^2$, $H^1$ and $H^2$-(semi)norms using the mixed degree Greville and mixed degree superconvergent points are visualized in Fig.~\ref{fig:collocation_poisson} for the bilinear multi-patch Domains~A--C and in Fig.~\ref{fig:collocation_poissonBL} for the Domains~D--F.
\begin{figure}[htb!]
    \centering
    \includegraphics[scale=0.48]{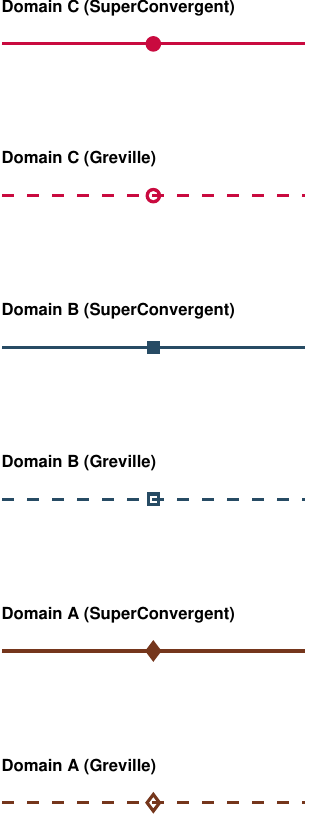}
    \hskip1.5em
    \includegraphics[scale=0.48]{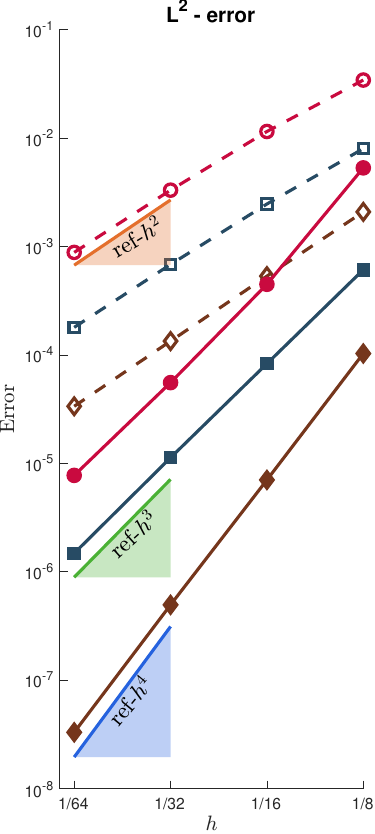}
    \hskip1.5em
    \includegraphics[scale=0.48]{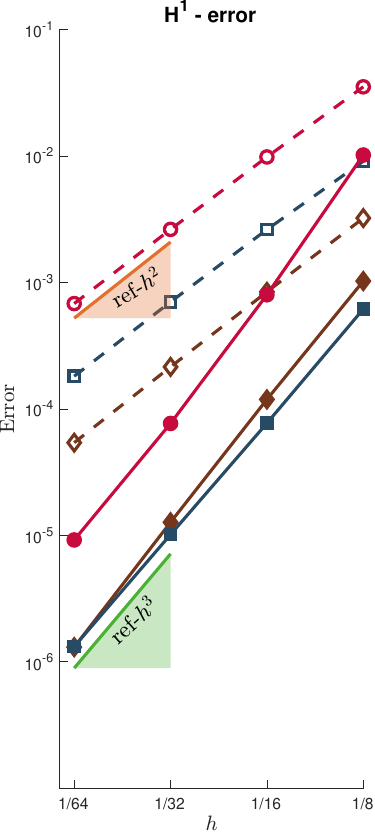}
    \hskip1.5em
    \includegraphics[scale=0.48]{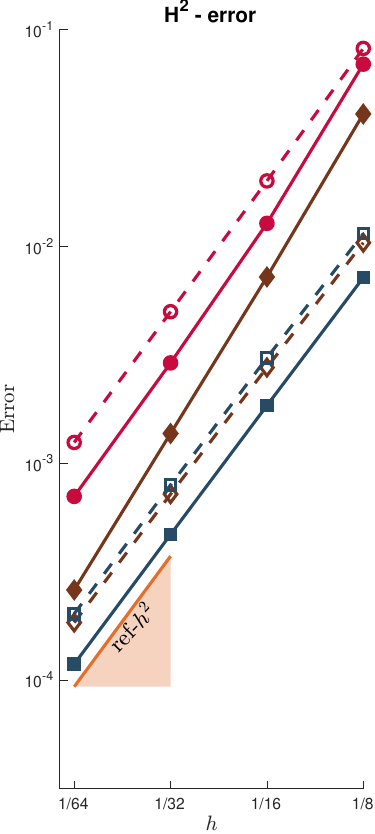}
    \caption{  Example~\ref{ex:collocationPoison}.
    Plots of the relative errors~\eqref{eq:eqiuv2seminorms} computed with respect to the $L^2$-norm and with respect to the $H^1$ and $H^2$-seminorms by solving the Poisson's equation~\eqref{eq:Poisson} over the bilinear multi-patch Domains~A, B and C using the $C^2$-smooth discretization space~$\mathcal{W}^2_h$ based on the underlying mixed degree spline space $\mathcal{S}_h^{(\ab{3}, \ab{5}),\ab{2}}([0,1]^2)$ for the corresponding mixed degree Greville and mixed degree superconvergent collocation points.}
    \label{fig:collocation_poisson}
\end{figure}
\begin{figure}[htb!]
    \centering
    \includegraphics[scale=0.48]{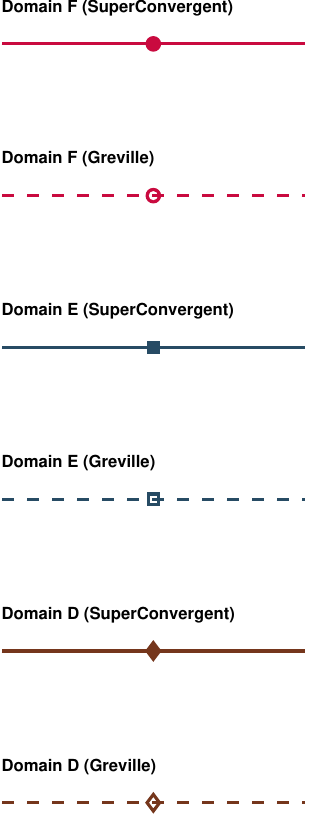}
    \hskip1.5em
    \includegraphics[scale=0.48]{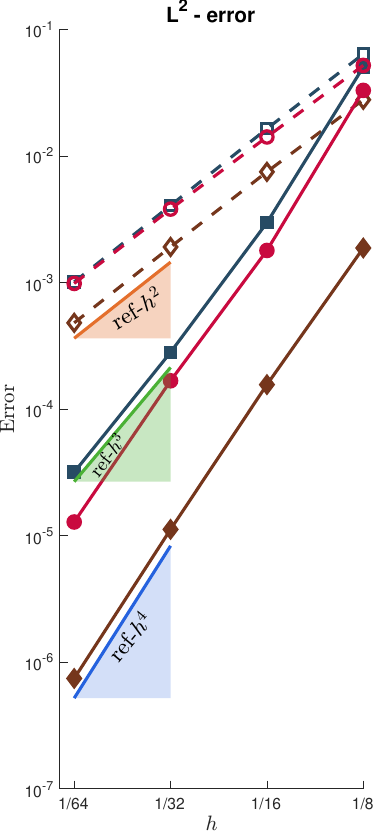}
    \hskip1.5em
    \includegraphics[scale=0.48]{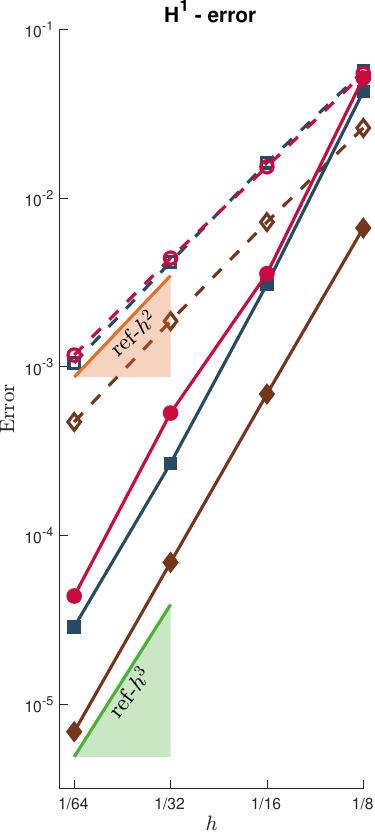}
    \hskip1.5em
    \includegraphics[scale=0.48]{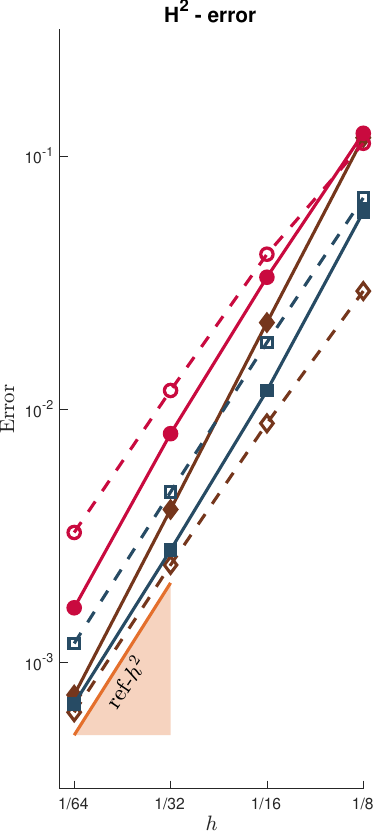}
      \caption{  Example~\ref{ex:collocationPoison}.
     Plots of the relative errors~\eqref{eq:eqiuv2seminorms} computed with respect to the $L^2$-norm and with respect to the $H^1$ and $H^2$-seminorms by solving the Poisson's equation~\eqref{eq:Poisson} over the Domains~D, E and F using the $C^2$-smooth discretization space~$\mathcal{W}^2_h$ based on the underlying mixed degree spline space $\mathcal{S}_h^{(\ab{3}, \ab{5}),\ab{2}}([0,1]^2)$ for the corresponding mixed degree Greville and mixed degree superconvergent collocation points.}
    \label{fig:collocation_poissonBL}
\end{figure}
While in case of the mixed degree Greville collocation points, the numerical results indicate for all six considered domains convergence orders of~$\mathcal{O}(h^{p_1-1}) = \mathcal{O}(h^{2}) $ for the $L^2$, $H^1$ and $H^2$-(semi)norms, in case of the mixed degree superconvergent points, the numerical results show for the one-patch Domains~A and D orders of~$\mathcal{O}(h^{p_1+1}) = \mathcal{O}(h^{4})$,  $\mathcal{O}(h^{p_1}) = \mathcal{O}(h^{3})$ and $\mathcal{O}(h^{p_1-1}) = \mathcal{O}(h^{2})$ for the $L^2$, $H^1$ and $H^2$-(semi)norms, respectively, and for the multi-patch Domains B, C, E and F orders of~$\mathcal{O}(h^{p_1}) = \mathcal{O}(h^{3}), \mathcal{O}(h^{p_1}) = \mathcal{O}(h^{3})$ and $\mathcal{O}(h^{p_1-1}) = \mathcal{O}(h^{2})$ for the $L^2$, $H^1$ and $H^2$-(semi)norms, respectively.
As seen, the orders for the mixed degree superconvergent points are higher in comparison with the orders for the mixed degree Greville collocation points in case of the $L^{2}$-norm and $H^1$-seminorm. In addition, note that the higher order in the $L^2$ norm for the superconvergent points in case of a one-patch domain compared to the case of a multi-patch domain has already been observed in \cite{KaVi20}.
%
\end{ex}

\begin{ex} \label{ex:collocationBiharmonic}
We solve now the biharmonic equation over the three bilinear multi-patch Domains~A--C, cf.~Fig.~\ref{fig:collocation_biharmonic}, and over the biquadratic one-patch Domain~D and the bilinear-like $G^4$ multi-patch Domain~F, cf.~Fig.~\ref{fig:collocation_biharmonicBL}. For this purpose, we use the $C^4$-smooth discretization space~$\mathcal{W}^4_h$ based on the underlying mixed degree spline space $\mathcal{S}_h^{(\ab{5}, \ab{9}),\ab{4}}([0,1]^2)$.  
The resulting relative errors~\eqref{eq:eqiuv2seminorms} with respect to the $L^2$ and $H^m$-(semi)norms, $m=1,\ldots,4$, by using the mixed degree Greville and the mixed degree superconvergent points are shown in Fig.~\ref{fig:collocation_biharmonic} and Fig.~\ref{fig:collocation_biharmonicBL}, respectively.
In case of the mixed Greville collocation points, the numerical results indicate for all five considered domains convergence orders of~$\mathcal{O}(h^{p_1-3}) = \mathcal{O}(h^{2}) $ for the $L^2$ and for all $H^m$-(semi)norms. In case of the mixed degree superconvergent points, the numerical results show for the two one-patch domains, i.e.~for the Domains A and D, orders of~$\mathcal{O}(h^{p_1-1}) = \mathcal{O}(h^{4})$ for the $L^2$, $H^1$ and $H^2$-(semi)norms, and orders of~$\mathcal{O}(h^{p_1-2}) = \mathcal{O}(h^{3})$ and $\mathcal{O}(h^{p_1-3}) = \mathcal{O}(h^{2})$ for the $H^3$ and $H^4$-(semi)norms, respectively. Furthermore, the error plots in case of the mixed degree superconvergent points indicate for the remaining multi-patch domains, i.e. for the Domains B, C and F, the orders~$\mathcal{O}(h^{p_1-2}) = \mathcal{O}(h^{3})$ for all the (semi)norms except for the $H^4$-seminorm, where the order $\mathcal{O}(h^{p_1-3}) = \mathcal{O}(h^{2})$ is already the optimal one. 
\begin{figure}[htb!]
    \centering
     \includegraphics[scale=0.45]{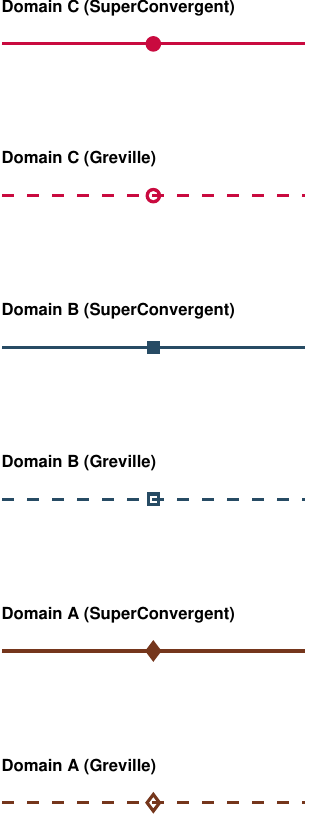}
     \hskip3em
    \includegraphics[scale=0.45]{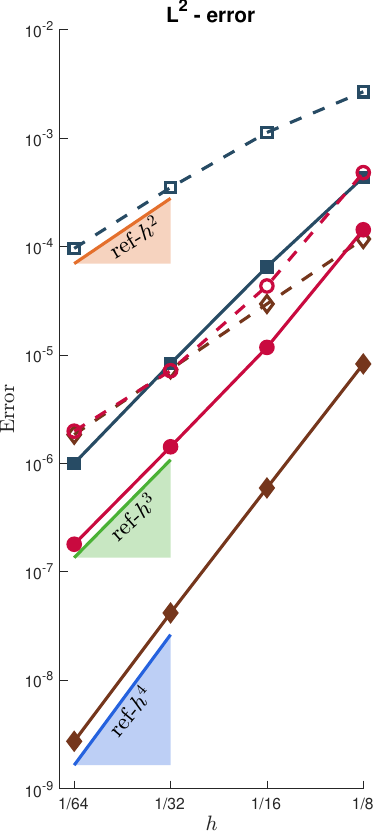}
    \hskip2.5em
    \includegraphics[scale=0.45]{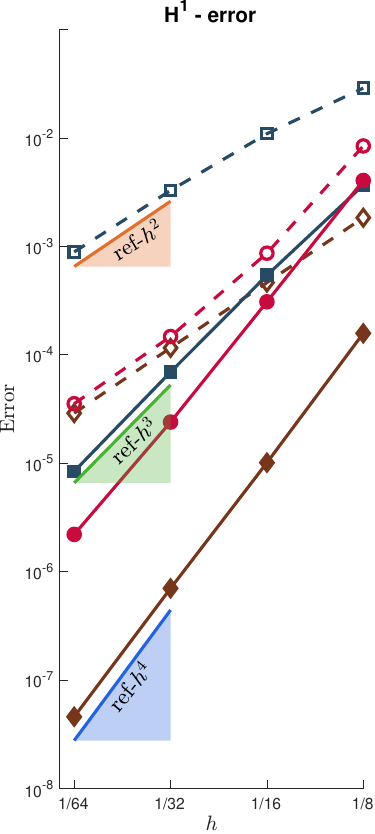}
     \\[0.4cm]
    \includegraphics[scale=0.45]{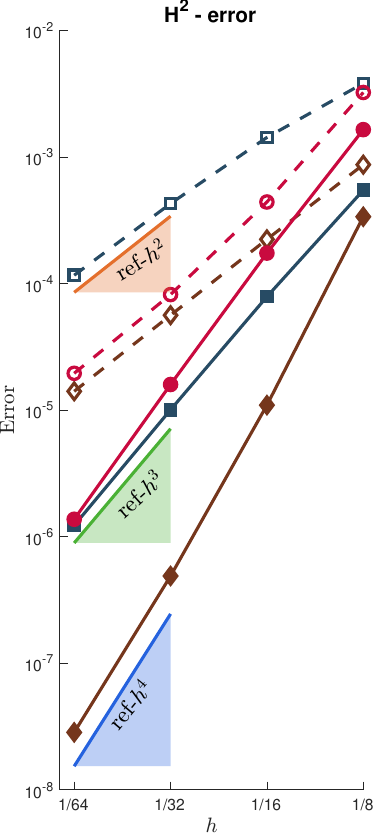}
    \hskip2.em
    \includegraphics[scale=0.45]{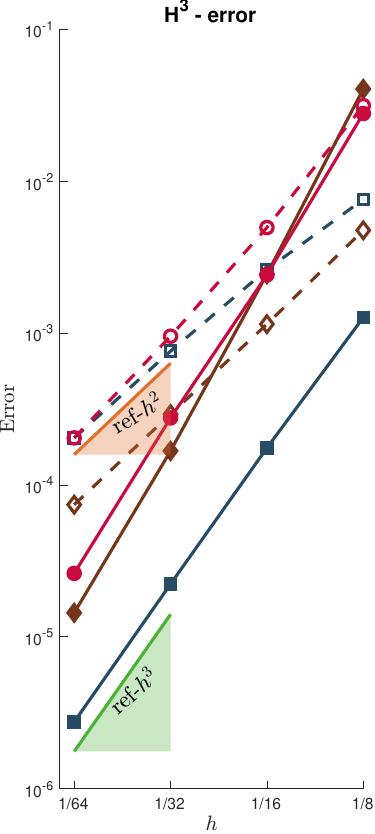}
    \hskip2.5em
    \includegraphics[scale=0.45]{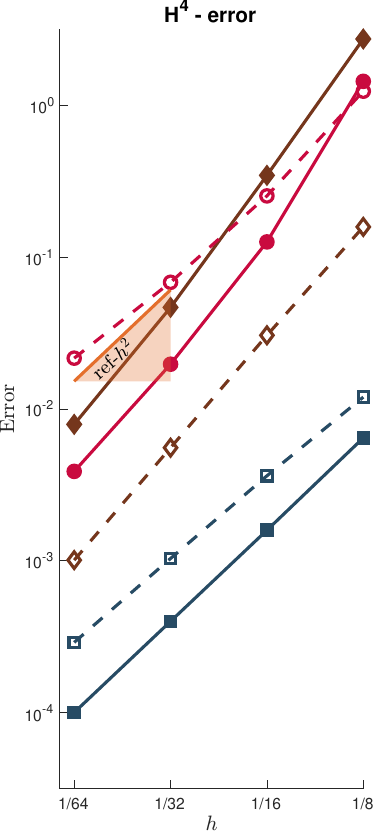}
      \caption{  Example~\ref{ex:collocationBiharmonic}.
    Plots of the relative errors~\eqref{eq:eqiuv2seminorms} computed with respect to the $L^2$-norm and with respect to the $H^1$, $H^2$, $H^3$ and $H^4$-seminorms for solving the biharmonic equation~\eqref{eq:biharmonic} over the bilinear multi-patch Domains A, B and C using the $C^4$-smooth discretization space~$\mathcal{W}^4_h$ based on the underlying mixed degree spline space $\mathcal{S}_h^{(\ab{5}, \ab{9}),\ab{4}}([0,1]^2)$ for the corresponding mixed degree Greville and mixed degree superconvergent collocation points.}
    \label{fig:collocation_biharmonic}
\end{figure}
\begin{figure}[htb!]
    \centering
     \includegraphics[scale=0.45]{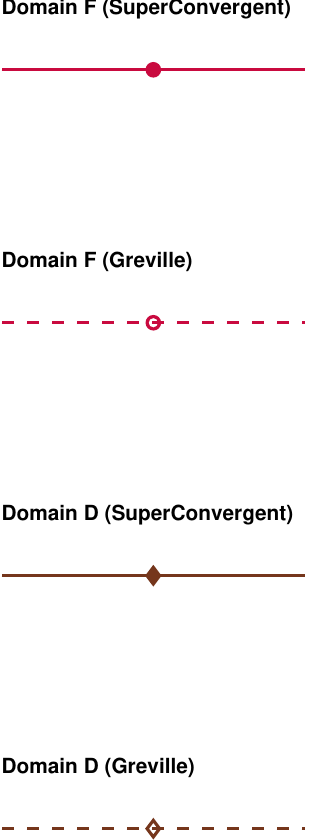}
     \hskip3em
    \includegraphics[scale=0.45]{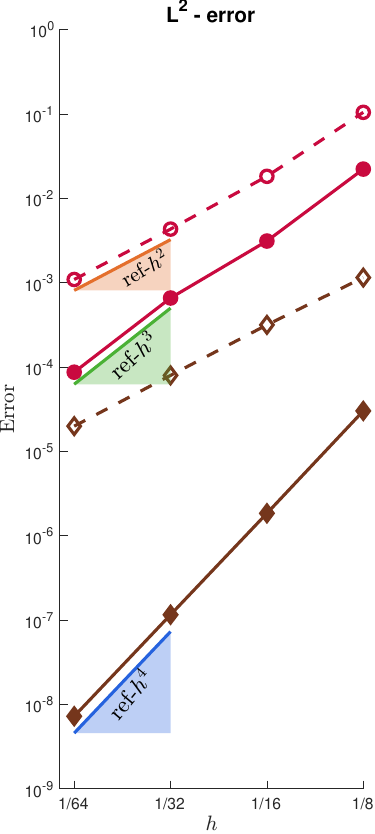}
    \hskip2.5em
    \includegraphics[scale=0.45]{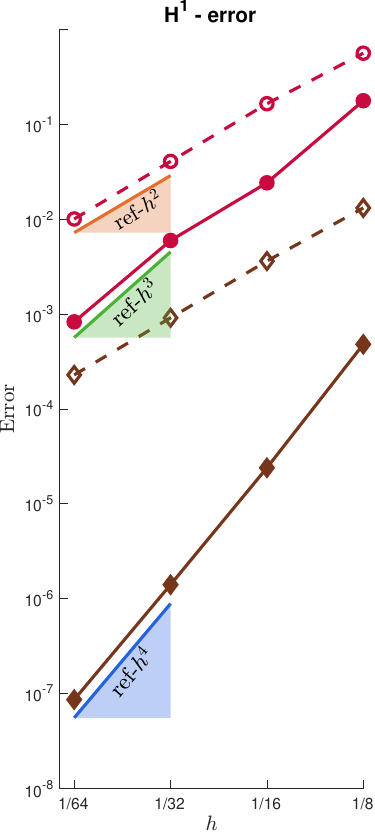}
    \\[0.4cm]
    \includegraphics[scale=0.45]{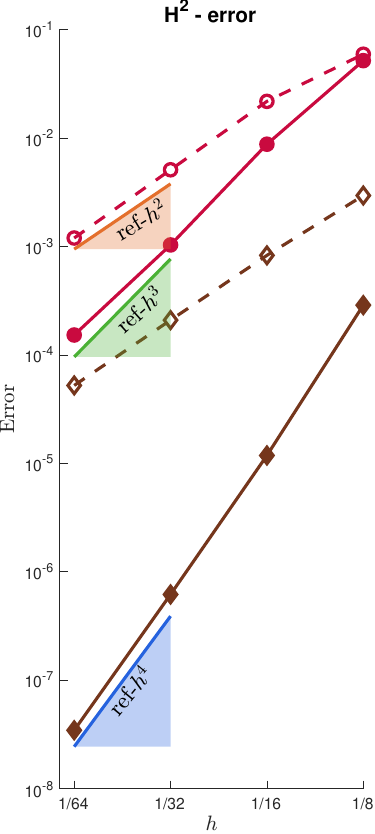}
    \hskip2.em
    \includegraphics[scale=0.45]{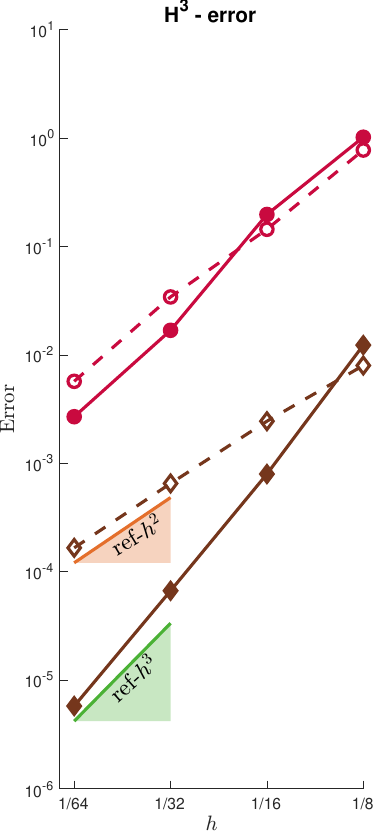}
    \hskip2.5em
    \includegraphics[scale=0.45]{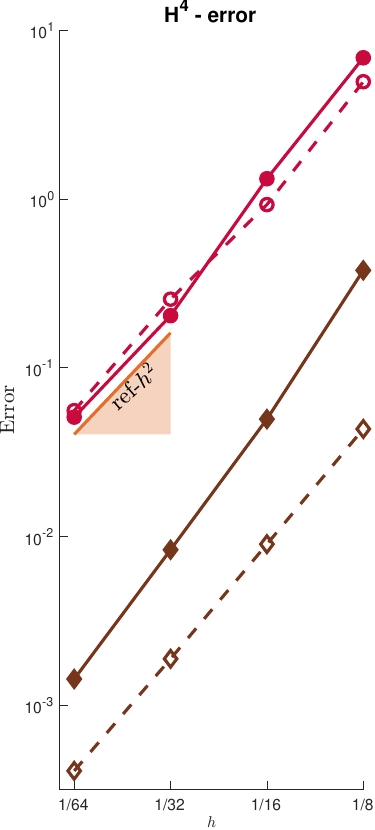}
    \caption{  Example~\ref{ex:collocationBiharmonic}.
    Plots of the relative errors~\eqref{eq:eqiuv2seminorms} computed with respect to the $L^2$-norm and with respect to the $H^1$, $H^2$, $H^3$ and $H^4$-seminorms for solving the biharmonic equation~\eqref{eq:biharmonic} over the Domains D and F using the $C^4$-smooth discretization space~$\mathcal{W}^4_h$ based on the underlying mixed degree spline space $\mathcal{S}_h^{(\ab{5}, \ab{9}),\ab{4}}([0,1]^2)$ for the corresponding mixed degree Greville and mixed degree superconvergent collocation points.}
 \label{fig:collocation_biharmonicBL}
\end{figure}
\end{ex}

\begin{ex} \label{ex:Lshape_squareSystem}
Given the L-shape Domain G, we study three different sets of mixed degree superconvergent collocation points to solve the Poisson's and the biharmonic equation over this bilinear two-patch domain by means of our isogeometric collocation method. 

The first set (Set 1) is just given by the described mixed superconvergent points from Section~\ref{subsec:superconvergent}. For the second set of collocation points (Set 2), we keep the same collocation points away from the inner edge (blue collocation points in Fig.~\ref{fig:collocation_poissonSquare} and Fig.~\ref{fig:collocation_biharmonicSquare}), while for the inner edge and for the $\sm$ neighboring columns on both sides of the inner edge, we select in the direction of the inner edge univariate clustered superconvergent collocation points which correspond to underlying spline spaces with maximal regularity, i.e. $\mathcal{S}_h^{2\sm+1-\ell,2\sm -\ell}([0,1])$, $\ell = 0,\ldots, \sm$, $\sm=2,4$. These superconvergent points on each knot span with respect to the reference interval~$[-1,1]$ are given for the case of the Poisson's equation ($\sm=2$) and of the biharmonic equation ($\sm=4$) in Table~\ref{tab:superconvergentpointsLocal}, cf.~\cite{Maurin2018}. 
\begin{table}[htb!]
    \centering
    \small
    \begin{tabular}{|c||c|c|c|}
         \hline\\[-0.35cm]
         \multirow{2}{*}{Poisson} & $\mathcal{S}_h^{3,2}([0,1])$ & $\mathcal{S}_h^{4,3}([0,1])$ & $\mathcal{S}_h^{5,4}([0,1])$  \\
         \cline{2-4}\\[-0.35cm]
         ($\sm=2$) & $ \mp 1/\sqrt{3}$ & $-1,0,1$ &  
         $\mp \sqrt{225-30\sqrt{30}}/15 $ \\
         \hline
    \end{tabular}
    \vskip1em
    \begin{tabular}{|c||c|c|c|c|c|}
         \hline\\[-0.35cm]
         \multirow{2}{*}{Biharmonic} & $\mathcal{S}_h^{5,4}([0,1])$ & $\mathcal{S}_h^{6,5}([0,1])$ & $\mathcal{S}_h^{7,6}([0,1])$ & $\mathcal{S}_h^{8,7}([0,1])$ & $\mathcal{S}_h^{9,8}([0,1])$ \\
         \cline{2-6}\\[-0.35cm]
          ($\sm=4$) & $\mp 1/\sqrt{3}$ & $-1,0,1$ & $\mp \sqrt{225-30\sqrt{30}}/15$ & $-1,0,1$ & $\mp 0.504918567512653$\\
         \hline
    \end{tabular}
  \caption{Example~\ref{ex:Lshape_squareSystem}. Superconvergent points for the spline spaces $\mathcal{S}_h^{2\sm+1-\ell,2\sm -\ell}([0,1])$, $\ell = 0,\ldots, \sm$, on each knot span with respect to the reference interval~$[-1,1]$, both for the Poisson's equation ($\sm=2$) and for the biharmonic equation ($\sm=4$).} \label{tab:superconvergentpointsLocal}
\end{table}
The clustering is done in an analogous way as in Section~\ref{subsec:superconvergent}, see~Fig.~\ref{fig:collocation_poissonSquare} (below left) and Fig.~\ref{fig:collocation_biharmonicSquare} (below left) for the Poisson's and the biharmonic equation, respectively (in both cases for the red and white points together). This set of points already possesses a much smaller cardinality than the Set 1, but the corresponding linear system is still slightly overdetermined, see Tables~\ref{tab:superconvergentpointsComaprisonPoisson} and \ref{tab:superconvergentpointsComaprisonBiharmonic}. 
\begin{table}[htb!]
    \centering
    \begin{tabular}{|c|c|c||c|c|c|c|}
         \hline     
          \multicolumn{3}{|c||}{} & Dimension & $L^2$-error & $H^1$-error & $H^2$-error  \\
         \hline
         \hline
         \multirow{9}{*}{\rotatebox{90}{{Poisson's equation}}} & \multirow{3}{*}{\rotatebox{90}{$h=\frac{1}{16}$}}& \rotatebox{0}{~{Set 1}~} & $939 \times 744$ & $4.7 \cdot 10^{-5}$  & $1.8 \cdot 10^{-4}$ & $3.5 \cdot 10^{-3}$ \\
         \cline{3-7}
         & & \rotatebox{0}{~{Set 2}~} & $804 \times 744$ & $5.2 \cdot 10^{-5}$ & $2.0 \cdot 10^{-4}$ & $3.6 \cdot 10^{-3}$  \\
         \cline{3-7}
         & & \rotatebox{0}{~{Set 3}~} & $744 \times 744 $ & $1.2 \cdot 10^{-4}$ & $3.2 \cdot 10^{-4}$ & $4.6 \cdot 10^{-3}$  \\
         \cline{2-7}\\[-0.35cm]
         \cline{2-7}
         & \multirow{3}{*}{\rotatebox{90}{$h=\frac{1}{32}$}} & \rotatebox{0}{{Set 1}} &  $2875 \times 2488$ & $3.0 \cdot 10^{-6}$ & $1.9 \cdot 10^{-5}$ & $8.8 \cdot 10^{-4}$  \\
         \cline{3-7}
         & & \rotatebox{0}{{Set 2}} &  $2596 \times 2488$ & $6.3 \cdot 10^{-6}$ & $2.2 \cdot 10^{-5}$ & $8.9 \cdot 10^{-4}$  \\
         \cline{3-7}
         & & \rotatebox{0}{{Set 3}} &  $2488 \times 2488$ & $2.0 \cdot 10^{-5}$ & $4.5 \cdot 10^{-5}$ & $1.1 \cdot 10^{-3}$ \\
         \cline{2-7}\\[-0.35cm]
        \cline{2-7}
        &  \multirow{3}{*}{\rotatebox{90}{$h=\frac{1}{64}$}} & \rotatebox{0}{~{Set 1}~} &  $9819 \times 9048$ & $3.7 \cdot 10^{-7}$ & $2.3 \cdot 10^{-6}$ & $2.2 \cdot 10^{-4}$  \\
         \cline{3-7}
         & & \rotatebox{0}{~{Set 2}~} &  $9252 \times 9048$ & $9.6 \cdot 10^{-7}$ & $2.8 \cdot 10^{-6}$ & $2.2 \cdot 10^{-4}$  \\
         \cline{3-7}
         & & \rotatebox{0}{~{Set 3}~} &  $9048 \times 9048$ & $3.2 \cdot 10^{-6}$ & $6.2 \cdot 10^{-6}$ & $2.5 \cdot 10^{-4}$  \\
         \hline
    \end{tabular}
     \caption{
    Example~\ref{ex:Lshape_squareSystem}.    
     Dimensions of the obtained linear system \eqref{eq:collocationSystemLocalPoisson} for all three different sets of superconvergent collocation points (Set 1--3) as well as the numerical errors with respect to the $L^2$-norm, and $H^1$ and $H^2$-seminorms for the Poisson's equation over the L-shape Domain~G in Fig.~\ref{fig:DomainG}.}    \label{tab:superconvergentpointsComaprisonPoisson}
\end{table}
The third set of collocation points (Set 3) is obtained by selecting in an alternating way a suitable subset of Set 2 for the $s$ neighboring columns on both sides of the inner edge, see Fig.~\ref{fig:collocation_poissonSquare} (below) and Fig.~\ref{fig:collocation_biharmonicSquare} (below) for the Poisson's and the biharmonic equation, respectively (in both cases the red points only). 
\begin{figure}[htb!]
    \centering
     \includegraphics[scale=0.36]{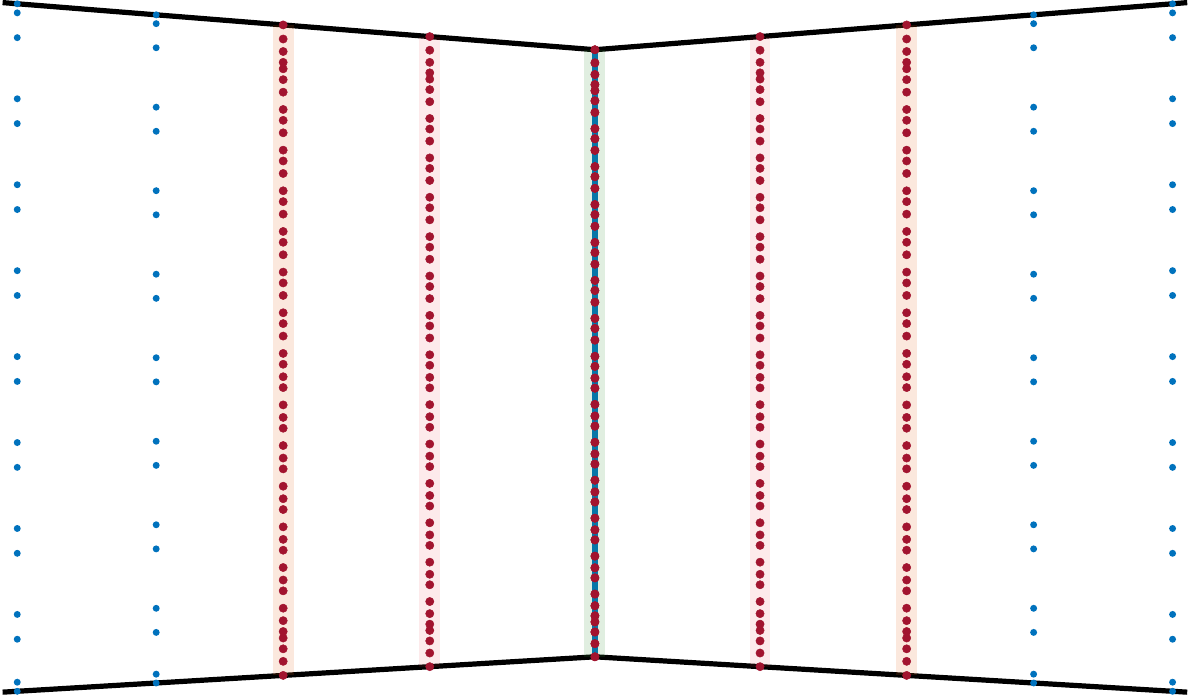}
     \hskip0.em
    \includegraphics[scale=0.36]{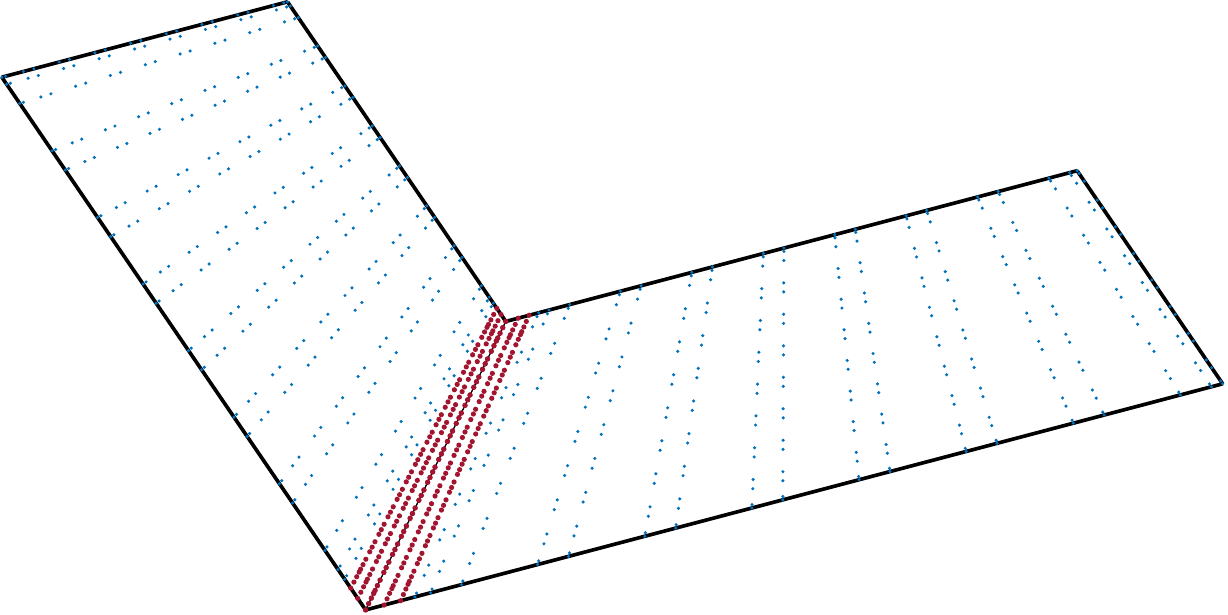}
    \\[0.3cm]
    \includegraphics[scale=0.36]{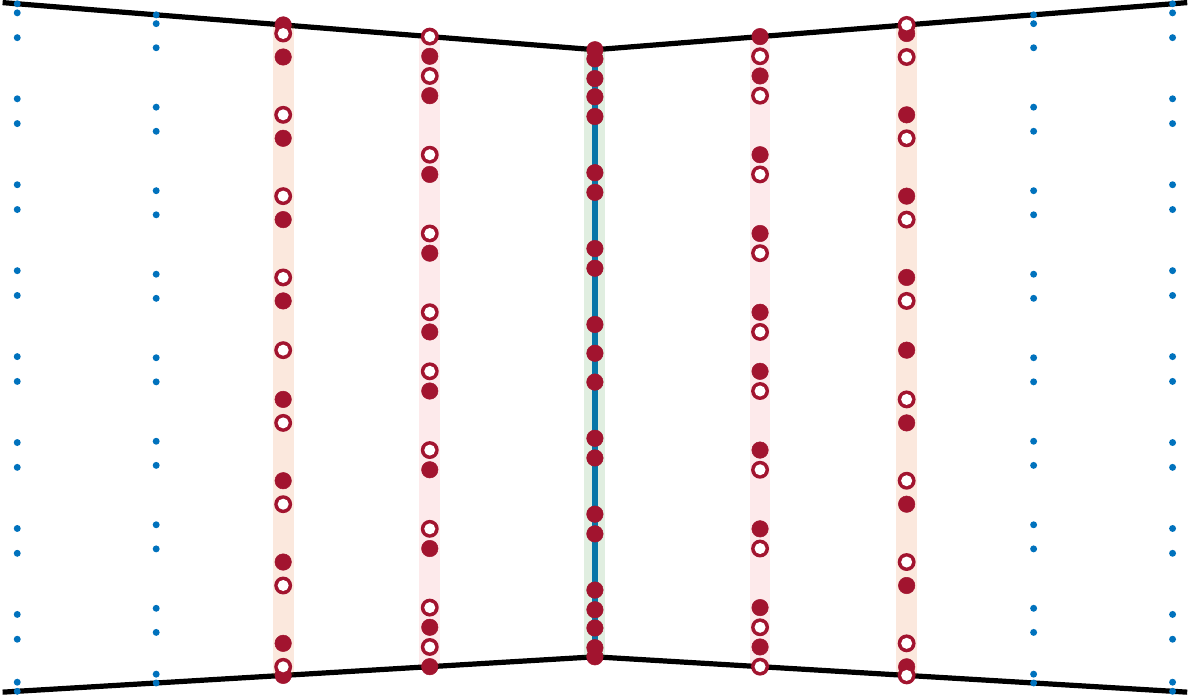}
    \hskip0.em
    \includegraphics[scale=0.36]{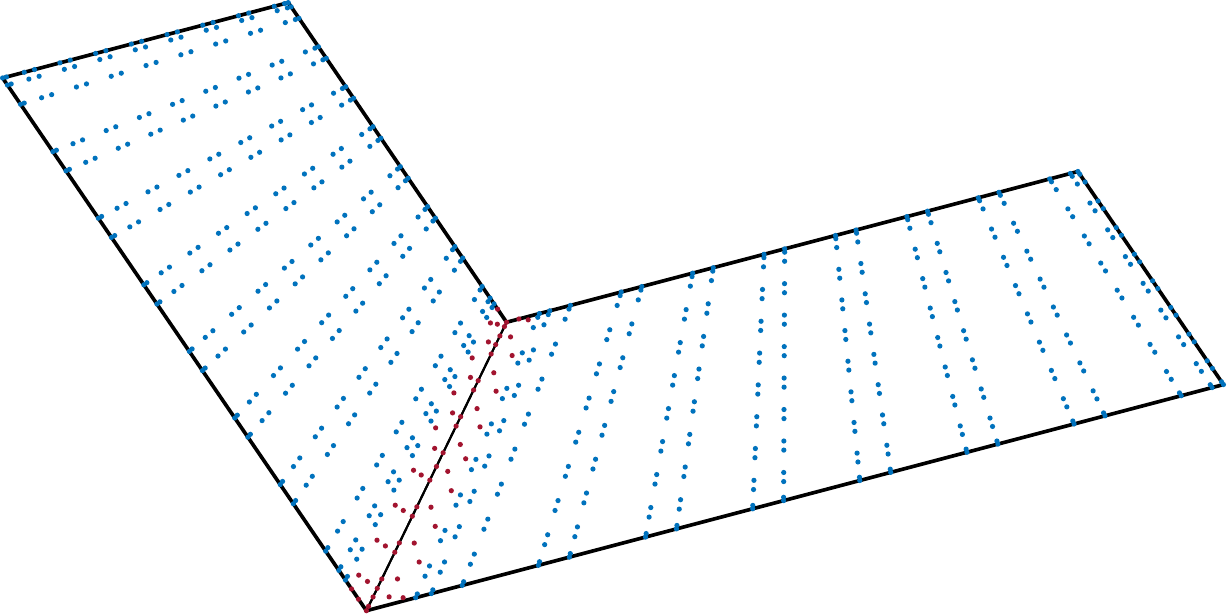}    \caption{Example~\ref{ex:Lshape_squareSystem}. Different sets of mixed degree superconvergent points (for $h=\frac{1}{16}$) in the vicinity of the inner edge for the Poisson's equation. Set 1 is shown above left, while Set 2 (red and white points) and Set 3 (red points only) are presented below left. On the right column, we visualize the application of these points (Set 1 above and Set 3 below) to the L-shape Domain~G.}
    \label{fig:collocation_poissonSquare}
\end{figure}

We compare in Table~\ref{tab:superconvergentpointsComaprisonPoisson} and Table~\ref{tab:superconvergentpointsComaprisonBiharmonic} the dimensions of the obtained linear systems \eqref{eq:collocationSystemLocalPoisson} as well as the numerical errors with respect to the $L^2$-norm, and $H^1$ and $H^2$-seminorms for the Poisson's equation and with respect to the $L^2$-norm, and $H^1, \ldots,H^4$-seminorms for the biharmonic equation, respectively, in both cases for the exact solution \eqref{eq:exactSolution} over the L-shape Domain~G shown in Fig.~\ref{fig:DomainG}. 
 The numerical results indicate the expected convergence orders for all (semi)norms, cf. Examples~\ref{ex:collocationPoison} and \ref{ex:collocationBiharmonic}.
\begin{table}[htb]
    \centering
    \begin{tabular}{|c|c|c|c|c|c|c|c|c|}
         \hline     
          \multicolumn{3}{|c|}{} & Dimension & $L^2$-error & $H^1$-error & $H^2$-error & $H^3$-error & $H^4$-error \\
         \hline
         \hline
         \multirow{9}{*}{\rotatebox{90}{{Biharmonic equation}}}& \multirow{3}{*}{\rotatebox{90}{$h=\frac{1}{16}$}}& \rotatebox{0}{{Set 1}} & $1597 \times 955$ & $1.8 \cdot 10^{-5}$  & $2.9 \cdot 10^{-5}$ & $5.4 \cdot 10^{-5}$ & $3.0 \cdot 10^{-4}$ & $3.5 \cdot 10^{-3}$\\
         \cline{3-9}
         & & \rotatebox{0}{{Set 2}} & $1064 \times 955$ & $3.0 \cdot 10^{-5}$ & $4.8 \cdot 10^{-5}$ & $8.5 \cdot 10^{-5}$ & $6.4 \cdot 10^{-4}$ & $1.6 \cdot 10^{-2}$ \\
         \cline{3-9}
         & & \rotatebox{0}{{Set 3}} & $955 \times 955$ & $1.0 \cdot 10^{-4}$ & $2.0 \cdot 10^{-4}$ & $7.7 \cdot 10^{-4}$ & $3.0 \cdot 10^{-2}$ & $4.3 \cdot 10^{-1}$ \\
         \cline{2-9}\\[-0.35cm]
         \cline{2-9}
         & \multirow{3}{*}{\rotatebox{90}{$h=\frac{1}{32}$}} & \rotatebox{0}{{Set 1}} &  $4145 \times 2859$ & $1.4 \cdot 10^{-6}$ & $2.2 \cdot 10^{-6}$ & $4.1 \cdot 10^{-6}$ & $2.6 \cdot 10^{-5}$ & $7.7 \cdot 10^{-4}$ \\
         \cline{3-9}
         & & \rotatebox{0}{{Set 2}} &  $3048 \times 2859$ & $2.4 \cdot 10^{-6}$ & $3.8 \cdot 10^{-6}$ & $6.7 \cdot 10^{-6}$ & $5.6 \cdot 10^{-5}$ & $2.9 \cdot 10^{-3}$ \\
         \cline{3-9}
         & & \rotatebox{0}{{Set 3}} &  $2859 \times 2859$ & $1.1 \cdot 10^{-5}$ & $2.0 \cdot 10^{-5}$ & $5.0 \cdot 10^{-5}$ & $2.5 \cdot 10^{-3}$ & $6.0 \cdot 10^{-2}$ \\
         \cline{2-9}\\[-0.35cm]
        \cline{2-9}
        &  \multirow{3}{*}{\rotatebox{90}{$h=\frac{1}{64}$}} & \rotatebox{0}{{Set 1}} &  $12305 \times 9739$ & $1.0 \cdot 10^{-7}$ & $1.6 \cdot 10^{-7}$ & $2.9 \cdot 10^{-7}$ & $2.5 \cdot 10^{-6}$ & $1.8 \cdot 10^{-4}$ \\
         \cline{3-9}
         & & \rotatebox{0}{{Set 2}} &  $10088 \times 9739$ & $1.7 \cdot 10^{-7}$ & $2.7 \cdot 10^{-7}$ & $4.7 \cdot 10^{-7}$ & $3.9 \cdot 10^{-6}$ & $3.6 \cdot 10^{-4}$ \\
         \cline{3-9}
         & & \rotatebox{0}{{Set 3}} &  $9739 \times 9739$ & $1.3 \cdot 10^{-6}$ & $2.6 \cdot 10^{-6}$ & $5.2 \cdot 10^{-6}$ & $2.0 \cdot 10^{-4}$ & $1.0 \cdot 10^{-2}$ \\
         \hline
    \end{tabular}
     \caption{
    Example~\ref{ex:Lshape_squareSystem}.    
     Dimensions of the obtained linear system \eqref{eq:collocationSystemLocal} for all three different sets of mixed degree superconvergent collocation points (Set 1--3) as well as the numerical errors with respect to the $L^2$-norm, and $H^1,\ldots,H^4$-seminorms for the biharmonic equation over the L-shape Domain~G in Fig.~\ref{fig:DomainG}.} \label{tab:superconvergentpointsComaprisonBiharmonic}
\end{table}
\begin{figure}[htb!]
    \centering
     \includegraphics[scale=0.355]{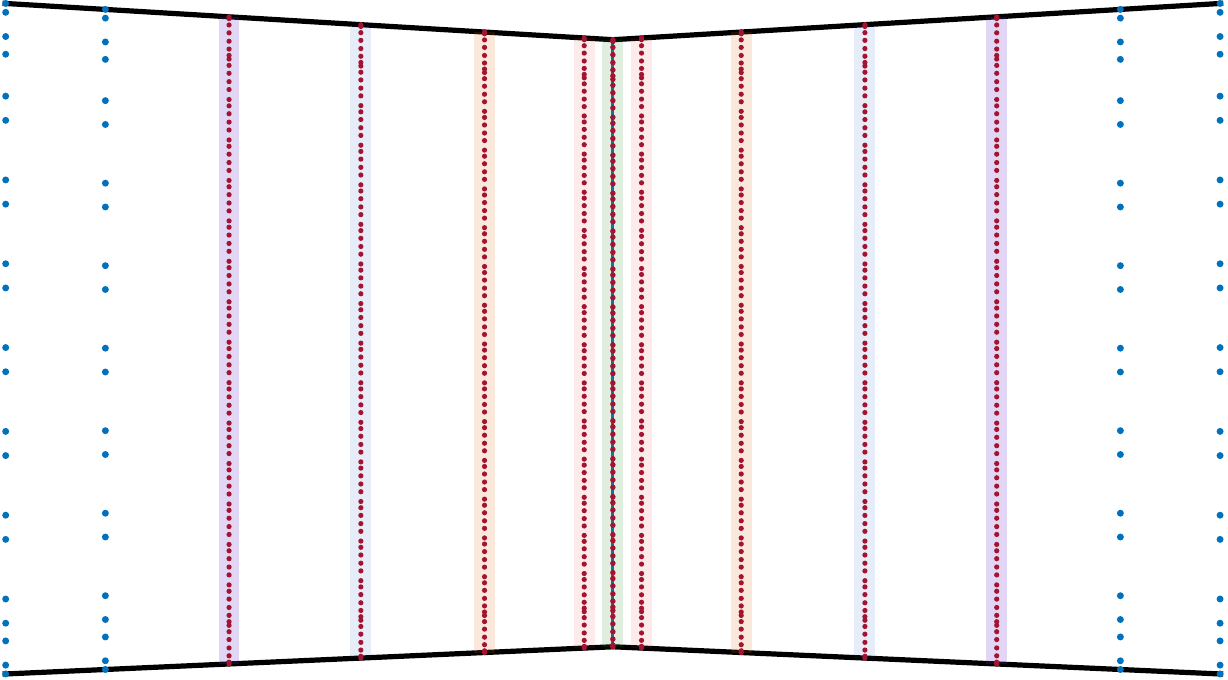}
     \hskip0em
    \includegraphics[scale=0.355]{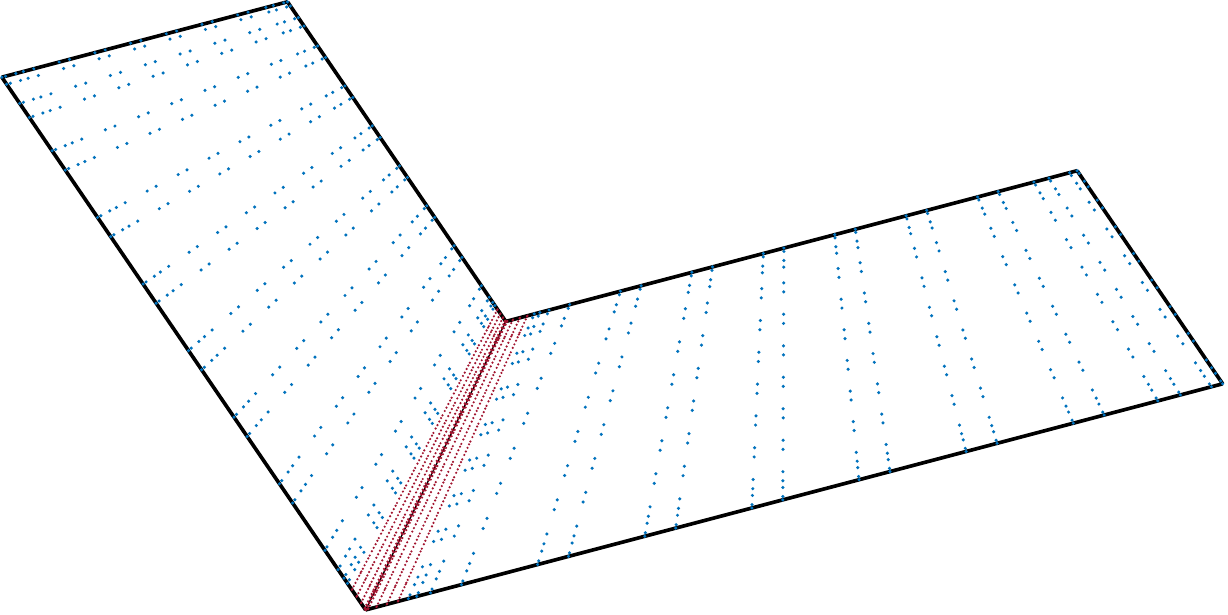}
    \\[0.3cm]
    \includegraphics[scale=0.355]{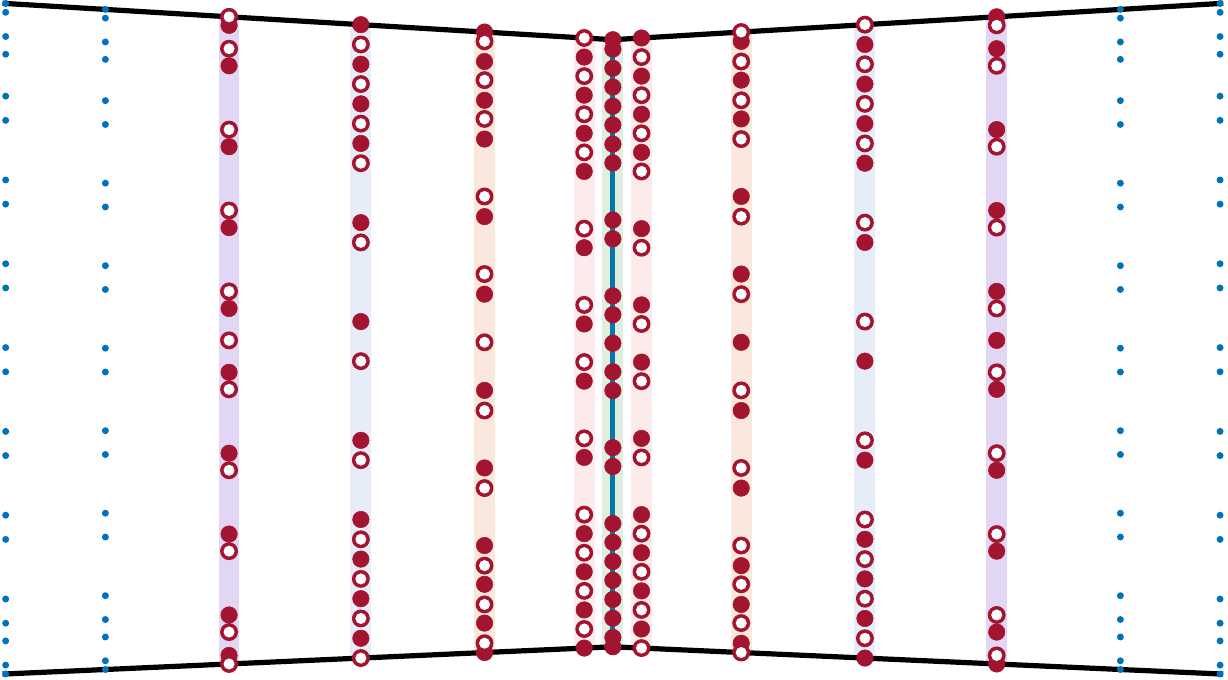}
    \hskip0em
    \includegraphics[scale=0.355]{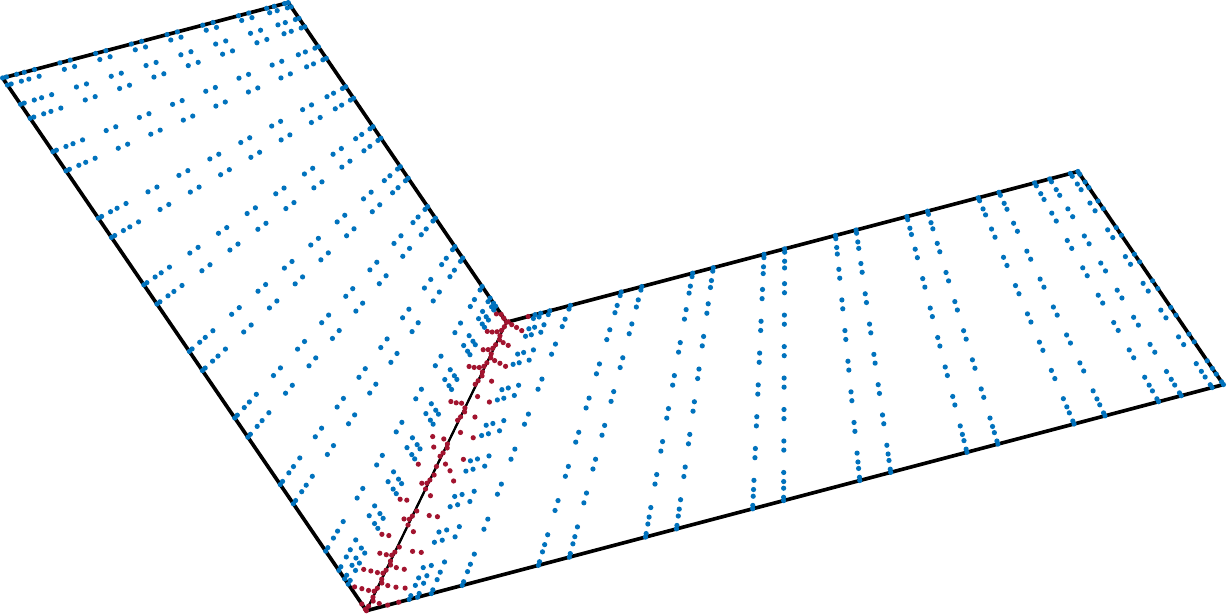}
    \caption{Example~\ref{ex:Lshape_squareSystem}. Different sets of superconvergent points (for $h=\frac{1}{16}$) in the vicinity of the inner edge for the biharmonic equation. Set 1 is shown above left, while Set 2 (red and white points) and Set 3 (red points only) are presented below left. On the right column we visualize the application of these points (Set 1 above and Set 3 below) to the L-shape Domain~G.}
\label{fig:collocation_biharmonicSquare}
\end{figure}
\end{ex}

\section{Conclusion} \label{sec:Conclusion}

We presented a new method for solving the Poisson's and the biharmonic equation in strong form via isogeometric collocation over planar bilinearly parameterized multi-patch domains. We further extended the approach on the basis of examples to the case of bilinear-like $G^s$ multi-patch geometries, which allow the modeling of curved boundaries, too. The proposed isogeometric collocation method is based on the use of a modified version of the $C^s$-smooth mixed degree isogeometric spline space~\cite{KaKoVi24b}, where the minimal possible degree~$p=s+1$ is employed everywhere on the multi-patch domain except in a small neighborhood of the inner edges and of the vertices of patch valency greater than one where a degree $p=2s+1$ is needed. In this way, it is possible to significantly reduce the number of degrees of freedom involved compared to the technique used in \cite{KaVi20} and \cite{KaKoVi24} for solving the Poisson's and the biharmonic equation, respectively, where the usage of the high degree~$p=2s+1$ is required on the entire multi-patch domain. 

For the choice of the collocation points, we introduced and studied two different sets of collocation points, namely the mixed degree Greville and the mixed degree superconvergent points. We tested both choices on one-patch domains as well as on multi-patch domains. For both sets of collocation points, we numerically studied the convergence behavior with respect to the $L^2$-norm and with respect to (equivalents of) the $H^m$-seminorms, $1\leq m\leq s$, $s=2,4$. 

In general, the number of collocation points in the multi-patch case is larger than the dimension of the discretization space, and hence a least squares approach is needed to solve the resulting linear system. Therefore, we made on the basis of a two-patch domain example a first strategy to select a suitable subset of the mixed degree superconvergent collocation points to impose also in the multi-patch case, in particular for the two-patch case in this work, a square linear system and to avoid the necessity of a least-squares method for solving it. A first open issue is to generalize the proposed strategy for the two-patch case to multi-patch domains with also extraordinary vertices. 
Moreover, the study of other applications such as the Kirchhoff plate or Kirchhoff-Love shell problem, 
and the extension of our approach to multi-patch surfaces or multi-patch volumes could be interesting research topics.

\paragraph*{\bf Acknowledgment}

M. Kapl has been partially supported by the Austrian Science Fund (FWF) through the project P~33023-N. V.~Vitrih has been partially supported by the 
Slovenian Research and Innovation Agency (research program P1-0404 and research projects N1-0296 and J1-4414). A.~Kosma\v c has been partially supported by the Slovenian Research and Innovation Agency (research program P1-0404, research project N1-0296 and Young Researchers Grant). This support is gratefully acknowledged.

\appendix

\section{The five possible variants of the underlying mixed degree spline space 
} 
\label{sec:AppendixSec2}

Let us consider in detail the five possible variants of the underlying mixed degree spline space $\mathcal{S}_h^{(\ab{p}_1, \ab{p}_2),\ab{\sm}}([0,1]^2)$ introduced in \eqref{eq:MixedSpace}, where
the five different 
cases depend on the number and position of edges of $[0,1]^2$ which correspond to the inner edges of the multi-patch domain $\overline{\Omega}$. 

\paragraph{Four inner edges}

The subspaces $\mathcal{S}_{1} ([0,1]^2)$, $\mathcal{\overline{S}}_1 ([0,1]^2)$ and $\mathcal{S}_2 ([0,1]^2)$ which form the space $\mathcal{S}_h^{(\ab{p}_1, \ab{p}_2),\ab{\sm}}([0,1]^2)$ in \eqref{eq:MixedSpace} are defined as
\begin{align*}
& \mathcal{S}_{1} ([0,1]^2) =  \Span \left\{ 
N_{j_1,j_2}^{\ab{p}_1,\ab{\sm}}, \; j_1,j_2= \sm+1,\ldots, n_{p_1}-\sm-2\right\},\\
& \mathcal{\overline{S}}_1 ([0,1]^2) = \Span \left\{ 
\overline{N}_{j_1}^{\,{p}_1,{\sm}} \, \overline{N}_{j_2}^{\,{p}_1,{\sm}}, \; j_1=1,\ldots,\sm, n_{p_1}-\sm-1, \ldots,n_{p_1}-2; \; j_2=1,\ldots,n_{p_1}-2 \right\} \oplus  \\
& \quad \qquad \Span \left\{ 
\overline{N}_{j_1}^{\,{p}_1,{\sm}} \, \overline{N}_{j_2}^{\,{p}_1,{\sm}}, \; 
j_1=\sm+1,\ldots,n_{p_1}-\sm-2;\; 
j_2=1,\ldots,\sm, n_{p_1}-\sm-1, \ldots,n_{p_1}-2 \right\},\\
&\mathcal{S}_2 ([0,1]^2) = \Span \left\{ 
N_{j_1,j_2}^{\ab{p}_2,\ab{\sm}}, \; j_1=0,\ldots,\sm, n_{p_2}-\sm-1, \ldots,n_{p_2}-1; \; j_2=0,\ldots,n_{p_2}-1 \right\} \oplus  \nonumber \\
 & \qquad \quad  \Span  \left\{ 
N_{j_1,j_2}^{\ab{p}_2,\ab{\sm}},\;  
j_1=\sm+1,\ldots,n_{p_2}-\sm-2;\;
j_2 =  0,\ldots,\sm, n_{p_2}-\sm-1, \ldots,n_{p_2}-1 \right\}, 
\end{align*}
cf.~Fig.~\ref{fig:SpacesMixed}~(above left).


\paragraph{Three inner edges}

Assume that the left, bottom and top edge correspond to inner edges of $\overline{\Omega}$, cf.~Fig.~\ref{fig:SpacesMixed}~(above right). Then
\begin{align*}
& \mathcal{S}_{1} ([0,1]^2) =  \Span \left\{ 
N_{j_1,j_2}^{\ab{p}_1,\ab{\sm}}, \; j_1 = \sm+1,\ldots,n_{p_1}-1 ;\, j_2= \sm+1,\ldots, n_{p_1}-\sm-2\right\},\\
& \mathcal{\overline{S}}_1 ([0,1]^2) = \Span \left\{ 
\overline{N}_{j_1}^{\,{p}_1,{\sm}} \, \overline{N}_{j_2}^{\,{p}_1,{\sm}},\;
j_1=1,\ldots,\sm ; \; j_2=1,\ldots,n_{p_1}-2 \right\} \oplus\\
& \qquad  \quad \;\;\,\Span \left\{ 
{N}_{j_1}^{\,{p}_1,{\sm}} \, \overline{N}_{j_2}^{\,{p}_1,{\sm}}, \; j_1=\sm+1,\ldots,n_{p_1}-1; \; j_2=1,\ldots,\sm, n_{p_1}-\sm-1, \ldots,n_{p_1}-2 \right\},\\
&\mathcal{S}_2 ([0,1]^2) = \Span \left\{ 
N_{j_1,j_2}^{\ab{p}_2,\ab{\sm}}, \; j_1=0,\ldots,\sm; \; j_2=0,\ldots,n_{p_2}-1 \right\} \oplus \nonumber \\
 & \qquad  \quad \;\;\,\Span  \left\{ 
N_{j_1,j_2}^{\ab{p}_2,\ab{\sm}}, \;  j_1=\sm+1,\ldots,n_{p_2}-1 ; \; j_2 =  0,\ldots,\sm, n_{p_2}-\sm-1, \ldots,n_{p_2}-1 \right\}.
\end{align*} 
\paragraph{Two adjacent inner edges}

Assume that the left and the bottom edge correspond to inner edges of $\overline{\Omega}$, cf.~Fig.~\ref{fig:SpacesMixed}~(below left). Then
\begin{align*}
& \mathcal{S}_{1} ([0,1]^2) =  \Span \left\{ 
N_{j_1,j_2}^{\ab{p}_1,\ab{\sm}}, \; j_1 = \sm+1,\ldots,n_{p_1}-1 ;\, j_2= \sm+1,\ldots, n_{p_1}-1\right\},\\
& \mathcal{\overline{S}}_1 ([0,1]^2) = \Span \left\{ 
\overline{N}_{j_1}^{\,{p}_1,{\sm}} \, \overline{N}_{j_2}^{\,{p}_1,{\sm}}
, \;  j_1=1,\ldots,\sm; \; j_2=1,\ldots,\sm \right\}
\oplus 
\Span \left\{ 
\overline{N}_{j_1}^{\,{p}_1,{\sm}} \, {N}_{j_2}^{\,{p}_1,{\sm}}
,  j_1=1,\ldots,s; \right.
\\
& 
\left.
j_2=\sm+1,\ldots,n_{p_1}-1 \right\} \oplus 
\Span \left\{ 
{N}_{j_1}^{\,{p}_1,{\sm}} \, \overline{N}_{j_2}^{\,{p}_1,{\sm}}, \;  j_1=\sm+1,\ldots,n_{p_1}-1; \; j_2=1,\ldots,\sm \right\},
\\
&\mathcal{S}_2 ([0,1]^2) = \Span \left\{ 
N_{j_1,j_2}^{\ab{p}_2,\ab{\sm}}, \; j_1=0,\ldots,\sm; \; j_2=0,\ldots,n_{p_2}-1 \right\} \oplus \nonumber \\
 & \qquad \qquad \quad \;\, \Span  \left\{ 
N_{j_1,j_2}^{\ab{p}_2,\ab{\sm}}, \;  j_1=\sm+1,\ldots,n_{p_2}-1; \; j_2 =  0,\ldots,\sm \right\}. 
\end{align*}

\paragraph{Two opposite inner edges}

Assume that the bottom and top edge correspond to inner edges of $\overline{\Omega}$, cf.~Fig.~\ref{fig:SpacesMixed}~(below middle). Then
\begin{align*}
& \mathcal{S}_{1} ([0,1]^2) =  \Span \left\{ 
N_{j_1,j_2}^{\ab{p}_1,\ab{\sm}}, \; j_1 = 0,\ldots,n_{p_1}-1 ;\, j_2= \sm+1,\ldots, n_{p_1}-\sm-2\right\},\\
& \mathcal{\overline{S}}_1 ([0,1]^2) = \Span \left\{ 
{N}_{j_1}^{\,{p}_1,{\sm}} \, \overline{N}_{j_2}^{\,{p}_1,{\sm}}, \; j_1=0,\ldots,n_{p_1}-1; \; j_2=1,\ldots,\sm, n_{p_1}-\sm-1, \ldots,n_{p_1}-2 \right\}, \\
&\mathcal{S}_2 ([0,1]^2) = \Span  \left\{ 
N_{j_1,j_2}^{\ab{p}_2,\ab{\sm}}, \; j_1=0,\ldots,n_{p_2}-1; \; j_2 =  0,\ldots,\sm, n_{p_2}-\sm-1, \ldots,n_{p_2}-1 \right\}.
\end{align*}

\paragraph{One inner edge}

Assume that only the bottom edge corresponds to an inner edge of $\overline{\Omega}$, see ~Fig.~\ref{fig:SpacesMixed}~(below right). Then
\begin{align*}
& \hskip-3.3cm \mathcal{S}_{1} ([0,1]^2) =  \Span \left\{ 
N_{j_1,j_2}^{\ab{p}_1,\ab{\sm}}, \; j_1 = 0,\ldots,n_{p_1}-1 ;\, j_2= \sm+1,\ldots, n_{p_1}-1\right\},\\
& \hskip-3.3cm  \mathcal{\overline{S}}_1 ([0,1]^2) = \Span \left\{ 
{N}_{j_1}^{\,{p}_1,{\sm}} \, \overline{N}_{j_2}^{\,{p}_1,{\sm}}, \; j_1=0,\ldots,n_{p_1}-1; \; j_2=1,\ldots,\sm \right\}, \\
& \hskip-3.3cm  \mathcal{S}_2 ([0,1]^2) = \Span  \left\{ 
N_{j_1,j_2}^{\ab{p}_2,\ab{\sm}}, \;  j_1=0,\ldots,n_{p_2}-1 ;\; j_2 =  0,\ldots,\sm \right\}. 
\end{align*}

\section{Control points of the bilinear-like $G^4$ five-patch geometry 
} 
\label{sec:AppendixPoints}

The spline geometry \eqref{eq:five_patch_domain_bilinearLike} consists of five biquintic geometry mappings $\ab{F}^{(i)}$ from the space $\mathcal{S}_{1/4}^{\ab{5}, \ab{4}}([0,1]^2) \times \mathcal{S}_{1/4}^{\ab{5}, \ab{4}}([0,1]^2)$ parameterized as in \eqref{eq:five_patch_domain_bilinearLike}
with the control points~$\ab{c}_{j_1,j_2}^{(i)}$ given in Table~\ref{five_patch_domain_bilinearLikeTable}. 
\begin{table}[htb!]
\tiny
\centering
\setlength{\tabcolsep}{-0.38em}
\begin{tabular}{|ccccccccc|}
\hline
\multicolumn{9}{|c|}{$\f{c}_{j_1,j_2}^{(0)}$} \\[0.1cm]
\hline
& & & & & & & & \\[-0.15cm]
  $(0,0)$ & $\left(\frac{3}{32},\frac{459}{1600}\right) $&$ \left(\frac{9}{32},\frac{1377}{1600}\right)$ &$
   \left(\frac{9}{16},\frac{1377}{800}\right) $&$ \left(\frac{15}{16},\frac{459}{160}\right) $&$ \left(\frac{21}{16},\frac{3213}{800}\right) $&$
   \left(\frac{51}{32},\frac{7803}{1600}\right) $&$ \left(\frac{57}{32},\frac{8721}{1600}\right) $&$ \left(\frac{15}{8},\frac{459}{80}\right) $\\[0.1cm] 
 $\left(\frac{3}{10},0\right) $&$ \left(\frac{249}{640},\frac{9069}{32000}\right) $&$ \left(\frac{363}{640},\frac{27207}{32000}\right) $&$
   \left(\frac{267}{320},\frac{27207}{16000}\right) $&$ \left(\frac{381}{320},\frac{9069}{3200}\right) $&$
   \left(\frac{99}{64},\frac{63483}{16000}\right) $&$ \left(\frac{1161}{640},\frac{154173}{32000}\right) $&$
   \left(\frac{255}{128},\frac{172311}{32000}\right) $&$ \left(\frac{333}{160},\frac{9069}{1600}\right)$ \\[0.1cm] 
 $\left(\frac{9}{10},0\right) $&$ \left(\frac{627}{640},\frac{8847}{32000}\right) $&$ \left(\frac{729}{640},\frac{26541}{32000}\right) $&$
   \left(\frac{441}{320},\frac{26541}{16000}\right) $&$ \left(\frac{543}{320},\frac{8847}{3200}\right) $&$
   \left(\frac{129}{64},\frac{61929}{16000}\right) $&$ \left(\frac{1443}{640},\frac{150399}{32000}\right) $&$
   \left(\frac{309}{128},\frac{168093}{32000}\right) $&$ \left(\frac{399}{160},\frac{8847}{1600}\right)$ \\[0.1cm] 
 $\left(\frac{9}{5},0\right)$ &$ \left(\frac{597}{320},\frac{4257}{16000}\right) $&$ \left(\frac{639}{320},\frac{12771}{16000}\right) $&$
   \left(\frac{351}{160},\frac{12771}{8000}\right) $&$ \left(\frac{393}{160},\frac{4257}{1600}\right) $&$
   \left(\frac{87}{32},\frac{29799}{8000}\right) $&$ \left(\frac{933}{320},\frac{72369}{16000}\right) $&$
   \left(\frac{195}{64},\frac{80883}{16000}\right) $&$ \left(\frac{249}{80},\frac{4257}{800}\right)$ \\[0.1cm] 
$ (3,0) $&$ \left(\frac{195}{64},\frac{807}{3200}\right) $&$ \left(\frac{201}{64},\frac{2421}{3200}\right) $&$
   \left(\frac{105}{32},\frac{2421}{1600}\right) $&$ \left(\frac{111}{32},\frac{807}{320}\right) $&$ \left(\frac{117}{32},\frac{5649}{1600}\right)
   $&$ \left(\frac{243}{64},\frac{13719}{3200}\right) $&$ \left(\frac{249}{64},\frac{15333}{3200}\right) $&$
   \left(\frac{63}{16},\frac{807}{160}\right)$ \\[0.1cm] 
$ \left(\frac{21}{5},0\right) $&$ \left(\frac{1353}{320},\frac{3813}{16000}\right) $&$ \left(\frac{1371}{320},\frac{11439}{16000}\right) $&$
   \left(\frac{699}{160},\frac{11439}{8000}\right) $&$ \left(\frac{717}{160},\frac{3813}{1600}\right) $&$
   \left(\frac{1911}{400},\frac{1383}{400}\right) $&$ \left(\frac{1881}{400},\frac{903}{200}\right) $&$
   \left(\frac{1857}{400},\frac{2097}{400}\right) $&$ \left(\frac{459}{100},\frac{2247}{400}\right) $\\[0.1cm] 
 $\left(\frac{51}{10},0\right) $&$ \left(\frac{3273}{640},\frac{7293}{32000}\right) $&$ \left(\frac{3291}{640},\frac{21879}{32000}\right) $&$
   \left(\frac{1659}{320},\frac{21879}{16000}\right) $&$ \left(\frac{1677}{320},\frac{7293}{3200}\right) $&$
   \left(\frac{144}{25},\frac{153}{50}\right) $&$ \left(\frac{567}{100},\frac{1641}{400}\right) $&$ \left(\frac{2157}{400},\frac{2007}{400}\right)
   $&$ \left(\frac{423}{80},\frac{273}{50}\right) $\\[0.1cm] 
 $\left(\frac{57}{10},0\right) $&$ \left(\frac{3651}{640},\frac{7071}{32000}\right) $&$ \left(\frac{3657}{640},\frac{21213}{32000}\right) $&$
   \left(\frac{1833}{320},\frac{21213}{16000}\right) $&$ \left(\frac{1839}{320},\frac{7071}{3200}\right) $&$
   \left(\frac{321}{50},\frac{1107}{400}\right) $&$ \left(\frac{2577}{400},\frac{711}{200}\right) $&$
   \left(\frac{1227}{200},\frac{1779}{400}\right) $&$ \left(\frac{381}{64},\frac{137}{28}\right) $\\[0.1cm]
  $(6,0) $&$ \left(6,\frac{87}{400}\right) $&$ \left(6,\frac{261}{400}\right) $&$ \left(6,\frac{261}{200}\right) $&$ \left(6,\frac{87}{40}\right) $&$
   \left(\frac{2703}{400},\frac{1041}{400}\right) $&$ \left(\frac{2733}{400},\frac{663}{200}\right) $&$ \left(\frac{13}{2},\frac{1653}{400}\right)
   $&$ \left(\frac{63}{10},\frac{201}{44}\right)$\\[0.1cm]
\hline
\hline
\multicolumn{9}{|c|}{$\f{c}_{j_1,j_2}^{(1)}$} \\[0.1cm]
\hline
& & & & & & & & \\[-0.15cm]
$(0,0) $&$ \left(-\frac{39}{160},\frac{57}{320}\right) $&$ \left(-\frac{117}{160},\frac{171}{320}\right) $&$
   \left(-\frac{117}{80},\frac{171}{160}\right) $&$ \left(-\frac{39}{16},\frac{57}{32}\right) $&$ \left(-\frac{273}{80},\frac{399}{160}\right) $&$
   \left(-\frac{663}{160},\frac{969}{320}\right) $&$ \left(-\frac{741}{160},\frac{1083}{320}\right) $&$ \left(-\frac{39}{8},\frac{57}{16}\right)$ \\[0.1cm] 
 $\left(\frac{3}{32},\frac{459}{1600}\right) $&$ \left(-\frac{237}{1600},\frac{7353}{16000}\right) $&$
   \left(-\frac{1011}{1600},\frac{12879}{16000}\right) $&$ \left(-\frac{543}{400},\frac{1323}{1000}\right) $&$
   \left(-\frac{93}{40},\frac{1611}{800}\right) $&$  \hspace{-0.15cm} \left(-\frac{1317}{400},\frac{5409}{2000}\right) $&$ \hspace{-0.15cm}
   \left(-\frac{6429}{1600},\frac{51561}{16000}\right) $&$ \hspace{-0.2cm} \hspace{0.03cm}\left(-\frac{7203}{1600},\frac{57087}{16000}\right)   $&$ \hspace{-0.10cm}
   \left(-\frac{759}{160},\frac{1197}{320}\right) $ \\[0.1cm] 
 $\left(\frac{9}{32},\frac{1377}{1600}\right) $&$ \left(\frac{69}{1600},\frac{16359}{16000}\right) $&$
   \left(-\frac{693}{1600},\frac{21537}{16000}\right) $&$ \left(-\frac{459}{400},\frac{3663}{2000}\right) $&$ \hspace{-0.1cm}
   \left(-\frac{21}{10},\frac{1983}{800}\right) $&$ \hspace{-0.1cm} \left(-\frac{1221}{400},\frac{1563}{500}\right) $&$ \hspace{-0.1cm}
   \left(-\frac{6027}{1600},\frac{57783}{16000}\right) $&$ \hspace{-0.15cm} \left(-\frac{6789}{1600},\frac{62961}{16000}\right) $&$ \hspace{-0.1cm}
   \left(-\frac{717}{160},\frac{1311}{320}\right)$ \\[0.1cm] 
 $\left(\frac{9}{16},\frac{1377}{800}\right) $&$ \left(\frac{33}{100},\frac{7467}{4000}\right) $&$ \left(-\frac{27}{200},\frac{8631}{4000}\right) $&$
   \left(-\frac{333}{400},\frac{10377}{4000}\right) $&$ \left(-\frac{141}{80},\frac{2541}{800}\right) $&$
   \left(-\frac{1077}{400},\frac{15033}{4000}\right) $&$ \left(-\frac{339}{100},\frac{16779}{4000}\right) $&$
   \left(-\frac{771}{200},\frac{17943}{4000}\right) $&$ \left(-\frac{327}{80},\frac{741}{160}\right)$ \\[0.1cm] 
 $\left(\frac{15}{16},\frac{459}{160}\right) $&$ \left(\frac{57}{80},\frac{1197}{400}\right) $&$ \left(\frac{21}{80},\frac{81}{25}\right) $&$
   \left(-\frac{33}{80},\frac{2889}{800}\right) $&$ \left(-\frac{21}{16},\frac{657}{160}\right) $&$ \left(-\frac{177}{80},\frac{3681}{800}\right) $&$
   \left(-\frac{231}{80},\frac{1989}{400}\right) $&$ \left(-\frac{267}{80},\frac{261}{50}\right) $&$ \left(-\frac{57}{16},\frac{171}{32}\right)$ \\[0.1cm] 
 $\left(\frac{21}{16},\frac{3213}{800}\right) $&$ \left(\frac{219}{200},\frac{16473}{4000}\right) $&$ \left(\frac{33}{50},\frac{17289}{4000}\right)
   $&$ \left(\frac{3}{400},\frac{18513}{4000}\right) $&$ \left(-\frac{69}{80},\frac{4029}{800}\right) $&$
   \left(-\frac{723}{400},\frac{2259}{400}\right) $&$ \left(-\frac{1137}{400},\frac{591}{100}\right) $&$
   \left(-\frac{1419}{400},\frac{243}{40}\right) $&$ \hspace{-0.1cm} \left(-\frac{783}{200},\frac{2457}{400}\right)$ \\[0.1cm] 
 $\left(\frac{51}{32},\frac{7803}{1600}\right) $&$ \left(\frac{2211}{1600},\frac{79401}{16000}\right) $&$
   \left(\frac{1533}{1600},\frac{82143}{16000}\right) $&$ \left(\frac{129}{400},\frac{5391}{1000}\right) $&$
   \left(-\frac{21}{40},\frac{4587}{800}\right) $&$ \left(-\frac{111}{100},  \frac{1293}{200}\right) $&$
   \left(-\frac{213}{100},\frac{1341}{200}\right) $&$ \hspace{-0.3cm} \left(-\frac{1239}{400},\frac{1347}{200}\right) $&$
   \hspace{-0.23cm} \left(-\frac{1419}{400},\frac{2709}{400}\right)$ \\[0.1cm] 
 $\left(\frac{57}{32},\frac{8721}{1600}\right) $&$ \left(\frac{2517}{1600},\frac{88407}{16000}\right) $&$
   \left(\frac{1851}{1600},\frac{90801}{16000}\right) $&$ \left(\frac{213}{400},\frac{11799}{2000}\right) $&$
   \left(-\frac{3}{10},\frac{4959}{800}\right) $&$ \left(-\frac{249}{400},\frac{1401}{200}\right) $&$
   \left(-\frac{273}{200},\frac{2913}{400}\right) $&$ \left(-\frac{921}{400},\frac{291}{40}\right) $&$ \left(-\frac{167}{60},\frac{507}{70}\right)$
   \\[0.1cm] 
 $\left(\frac{15}{8},\frac{459}{80}\right) $&$ \left(\frac{267}{160},\frac{9291}{1600}\right) $&$ \left(\frac{201}{160},\frac{9513}{1600}\right) $&$
   \left(\frac{51}{80},\frac{4923}{800}\right) $&$ \left(-\frac{3}{16},\frac{1029}{160}\right) $&$ \left(-\frac{141}{400},\frac{2913}{400}\right) $&$
   \left(-\frac{201}{200},\frac{3033}{400}\right) $&$ \left(-\frac{49}{26},\frac{753}{100}\right) $&$ \left(-\frac{189}{80},\frac{389}{52}\right)$
   \\[0.1cm]
\hline
\hline
\multicolumn{9}{|c|}{$\f{c}_{j_1,j_2}^{(2)}$} \\[0.1cm]
\hline
& & & & & & & & \\[-0.15cm]
 $(0,0) $&$ -\hspace{-0.05cm}\left(\frac{39}{160},\frac{57}{320}\right) $&$ -\hspace{-0.05cm}\left(\frac{117}{160},\frac{171}{320}\right) $&$
   -\hspace{-0.05cm}\left(\frac{117}{80},\frac{171}{160}\right) $&$ -\hspace{-0.05cm}\left(\frac{39}{16},\frac{57}{32}\right) $&$ -\hspace{-0.05cm}\left(\frac{273}{80},\frac{399}{160}\right) $&$
   -\hspace{-0.05cm}\left(\frac{663}{160},\frac{969}{320}\right) $&$ -\hspace{-0.05cm}\left(\frac{741}{160},\frac{1083}{320}\right) $&$ -\hspace{-0.05cm}\left(\frac{39}{8},\frac{57}{16}\right)$
   \\[0.1cm]
 $\left(-\frac{39}{160},\frac{57}{320}\right) $&$ \left(-\frac{771}{1600},0\right) $&$ 
 -\hspace{-0.05cm}\left(\frac{1533}{1600},\frac{57}{160}\right) $&$
   -\hspace{-0.05cm}\left(\frac{669}{400},\frac{57}{64}\right) $&$ -\hspace{-0.05cm}\left(\frac{21}{8},\frac{513}{320}\right) $&$ -\hspace{-0.05cm}\left(\frac{1431}{400},\frac{741}{320}\right)
   $&$ -\hspace{-0.05cm}\left(\frac{6867}{1600},\frac{57}{20}\right) $&$ \hspace{-0.2cm} -\hspace{-0.05cm}\left(\frac{7629}{1600},\frac{513}{160}\right) $&$ \hspace{-0.15cm}
   -\hspace{-0.05cm}\left(\frac{801}{160},\frac{1083}{320}\right) $\\[0.1cm]
 $\left(-\frac{117}{160},\frac{171}{320}\right) $&$ \left(-\frac{1533}{1600},\frac{57}{160}\right) $&$ \left(-\frac{2259}{1600},0\right) $&$
   -\hspace{-0.05cm}\left(\frac{837}{400},\frac{171}{320}\right) $&$ -\hspace{-0.05cm}\left(3,\frac{399}{320}\right) $&$ -\hspace{-0.05cm}\left(\frac{1563}{400},\frac{627}{320}\right) $&$
    -\hspace{-0.05cm} \left(\frac{7341}{1600},\frac{399}{160}\right) $&$ -\hspace{-0.05cm}\left(\frac{8067}{1600},\frac{57}{20}\right) $&$
   -\hspace{-0.05cm}\left(\frac{843}{160},\frac{969}{320}\right)$ \\[0.1cm]
 $\left(-\frac{117}{80},\frac{171}{160}\right) $&$ \left(-\frac{669}{400},\frac{57}{64}\right) $&$ \left(-\frac{837}{400},\frac{171}{320}\right) $&$
   \left(-\frac{1089}{400},0\right) $&$ 
   -\hspace{-0.05cm}\left(\frac{57}{16},\frac{57}{80}\right) $&$ -\hspace{-0.05cm}\left(\frac{1761}{400},\frac{57}{40}\right) $&$
   -\hspace{-0.05cm} \left(\frac{2013}{400},\frac{627}{320}\right) $&$ -\hspace{-0.05cm} \left(\frac{2181}{400},\frac{741}{320}\right) $&$
   -\hspace{-0.05cm} \left(\frac{453}{80},\frac{399}{160}\right) $\\[0.1cm]
 $\left(-\frac{39}{16},\frac{57}{32}\right) $&$ \left(-\frac{21}{8},\frac{513}{320}\right) $&$ \left(-3,\frac{399}{320}\right) $&$
   \left(-\frac{57}{16},\frac{57}{80}\right) $&$ \left(-\frac{69}{16},0\right) $&$ 
   -\hspace{-0.05cm}\left(\frac{81}{16},\frac{57}{80}\right) $&$
   -\hspace{-0.05cm}\left(\frac{45}{8},\frac{399}{320}\right) $&$ 
   -\hspace{-0.05cm}\left(6,\frac{513}{320}\right) $&$ -\hspace{-0.05cm}\left(\frac{99}{16},\frac{57}{32}\right)$ \\[0.1cm]
 $\left(-\frac{273}{80},\frac{399}{160}\right) $&$ \left(-\frac{1431}{400},\frac{741}{320}\right) $&$ \left(-\frac{1563}{400},\frac{627}{320}\right)
   $&$ \left(-\frac{1761}{400},\frac{57}{40}\right) $&$ \left(-\frac{81}{16},\frac{57}{80}\right) $&$ -\hspace{-0.05cm}\left(\frac{297}{50},\frac{3}{400}\right) $&$
   -\hspace{-0.05cm}\left(\frac{651}{100},\frac{183}{200}\right) $&$ 
    -\hspace{-0.05cm} \left(\frac{1377}{200},\frac{309}{200}\right) $&$ \hspace{-0.1cm}
   -\hspace{-0.05cm}\left(\frac{2823}{400},\frac{753}{400}\right) $\\[0.1cm]
 $\left(-\frac{663}{160},\frac{969}{320}\right) $&$ \left(-\frac{6867}{1600},\frac{57}{20}\right) $&$
   \left(-\frac{7341}{1600},\frac{399}{160}\right) $&$ \left(-\frac{2013}{400},\frac{627}{320}\right) $&$
   \left(-\frac{45}{8},\frac{399}{320}\right) $&$ \left(-\frac{651}{100},\frac{9}{10}\right) $&$ -\hspace{-0.05cm}\left(\frac{141}{20},\frac{3}{400}\right) $&$
   -\hspace{-0.05cm}\left(\frac{369}{50},\frac{369}{400}\right) $&$ \hspace{-0.1cm} 
   -\hspace{-0.05cm}\left(\frac{3021}{400},\frac{537}{400}\right)$ \\[0.1cm]
 \hspace{0.1cm} $\left(-\frac{741}{160},\frac{1083}{320}\right) $&$ \left(-\frac{7629}{1600},\frac{513}{160}\right) $&$
   \left(-\frac{8067}{1600},\frac{57}{20}\right) $&$ \left(-\frac{2181}{400},\frac{741}{320}\right) $&$ \left(-6,\frac{513}{320}\right) $&$
   \left(-\frac{1377}{200},\frac{153}{100}\right) $&$ \left(-\frac{369}{50},\frac{363}{400}\right) $&$
   -\hspace{-0.05cm}\left(\frac{3063}{400},\frac{3}{400}\right) $&$ -\hspace{-0.05cm}\left(\frac{373}{48},\frac{19}{40}\right)$ \\[0.1cm]
 $\left(-\frac{39}{8},\frac{57}{16}\right) $&$ \left(-\frac{801}{160},\frac{1083}{320}\right) $&$ \left(-\frac{843}{160},\frac{969}{320}\right) $&$
   \left(-\frac{453}{80},\frac{399}{160}\right) $&$ \left(-\frac{99}{16},\frac{57}{32}\right) $&$ \left(-\frac{2823}{400},\frac{15}{8}\right) $&$
   \left(-\frac{3021}{400},\frac{267}{200}\right) $&$ \left(-\frac{373}{48},\frac{19}{40}\right) $&$ \left(-\frac{63}{8},0\right) $\\[0.1cm]
\hline
\hline
\multicolumn{9}{|c|}{$\f{c}_{j_1,j_2}^{(3)}$} \\[0.1cm]
\hline
& & & & & & & & \\[-0.15cm]
 $(0,0) $&$ \left(\frac{3}{32},-\frac{459}{1600}\right) $&$ \left(\frac{9}{32},-\frac{1377}{1600}\right) $&$
   \left(\frac{9}{16},-\frac{1377}{800}\right) $&$ \left(\frac{15}{16},-\frac{459}{160}\right) $&$ \left(\frac{21}{16},-\frac{3213}{800}\right) $&$
   \left(\frac{51}{32},-\frac{7803}{1600}\right) $&$ \left(\frac{57}{32},-\frac{8721}{1600}\right) $&$ \left(\frac{15}{8},-\frac{459}{80}\right)$ \\[0.1cm]
 $-\hspace{-0.05cm}\left(\frac{39}{160},\frac{57}{320}\right) $&$ \hspace{0.09cm}-\hspace{-0.05cm} \left(\frac{237}{1600},\frac{7353}{16000}\right) $&$
   \left(\frac{69}{1600},-\frac{16359}{16000}\right) $&$ \left(\frac{33}{100},-\frac{7467}{4000}\right) $&$
   \left(\frac{57}{80},-\frac{1197}{400}\right) $&$ \hspace{-0.1cm} \left(\frac{219}{200},-\frac{16473}{4000}\right) $&$
   \hspace{-0.13cm} \left(\frac{2211}{1600},-\frac{79401}{16000}\right) $&$ \hspace{-0.2cm} \hspace{0.02cm} \left(\frac{2517}{1600},-\frac{88407}{16000}\right) $&$
   \hspace{-0.1cm} \left(\frac{267}{160},-\frac{9291}{1600}\right)$ \\[0.1cm]
 $-\hspace{-0.05cm}\left(\frac{117}{160},\frac{171}{320}\right) $&$ \hspace{0.05cm} -\hspace{-0.05cm}\left(\frac{1011}{1600},\frac{12879}{16000}\right) $&$
   -\hspace{-0.05cm}\left(\frac{693}{1600},\frac{21537}{16000}\right) $&$ -\hspace{-0.05cm}\left(\frac{27}{200},\frac{8631}{4000}\right) $&$
   \left(\frac{21}{80},-\frac{81}{25}\right) $&$ \hspace{-0.2cm} \left(\frac{33}{50},-\frac{17289}{4000}\right) $&$
   \hspace{-0.2cm} \left(\frac{1533}{1600},-\frac{82143}{16000}\right) $&$ \hspace{-0.2cm}  \left(\frac{1851}{1600},-\frac{90801}{16000}\right) $&$
   \hspace{-0.1cm} \left(\frac{201}{160},-\frac{9513}{1600}\right)$ \\[0.1cm]
 $-\hspace{-0.05cm}\left(\frac{117}{80},\frac{171}{160}\right) $&$ -\hspace{-0.05cm}\left(\frac{543}{400},\frac{1323}{1000}\right) $&$
   -\hspace{-0.05cm}\left(\frac{459}{400},\frac{3663}{2000}\right) $&$ -\hspace{-0.05cm}\left(\frac{333}{400},\frac{10377}{4000}\right) $&$
   -\hspace{-0.05cm}\left(\frac{33}{80},\frac{2889}{800}\right) $&$ \left(\frac{3}{400},-\frac{18513}{4000}\right) $&$
   \left(\frac{129}{400},-\frac{5391}{1000}\right) $&$ \left(\frac{213}{400},-\frac{11799}{2000}\right) $&$
   \left(\frac{51}{80},-\frac{4923}{800}\right)$ \\[0.1cm]
 $-\hspace{-0.05cm}\left(\frac{39}{16},\frac{57}{32}\right) $&$ -\hspace{-0.05cm}\left(\frac{93}{40},\frac{1611}{800}\right) $&$ -\hspace{-0.05cm}\left(\frac{21}{10},\frac{1983}{800}\right) $&$
   -\hspace{-0.05cm}\left(\frac{141}{80},\frac{2541}{800}\right) $&$ -\hspace{-0.05cm}\left(\frac{21}{16},\frac{657}{160}\right) $&$
   -\hspace{-0.05cm}\left(\frac{69}{80},\frac{4029}{800}\right) $&$ -\hspace{-0.05cm}\left(\frac{21}{40},\frac{4587}{800}\right) $&$ -\hspace{-0.05cm}\left(\frac{3}{10},\frac{4959}{800}\right)
   $&$ -\hspace{-0.05cm}\left(\frac{3}{16},\frac{1029}{160}\right)$ \\[0.1cm]
 $-\hspace{-0.05cm}\left(\frac{273}{80},\frac{399}{160}\right) $&$ -\hspace{-0.05cm} \left(\frac{1317}{400},\frac{5409}{2000}\right) $&$
   -\hspace{-0.05cm}\left(\frac{1221}{400},\frac{1563}{500}\right) $&$ -\hspace{-0.05cm}\left(\frac{1077}{400},\frac{15033}{4000}\right) $&$
   -\hspace{-0.05cm}\left(\frac{177}{80},\frac{3681}{800}\right) $&$ -\hspace{-0.05cm}\left(\frac{723}{400},\frac{453}{80}\right) $&$
   -\hspace{-0.05cm}\left(\frac{111}{100},\frac{162}{25}\right) $&$ -\hspace{-0.05cm}\left(\frac{249}{400},\frac{351}{50}\right) $&$
   -\hspace{-0.05cm}\left(\frac{141}{400},\frac{729}{100}\right)$ \\[0.1cm]
 $-\hspace{-0.05cm}\left(\frac{663}{160},\frac{969}{320}\right) $&$ -\hspace{-0.05cm} \left(\frac{6429}{1600},\frac{51561}{16000}\right) $&$
   -\hspace{-0.05cm}\left(\frac{6027}{1600},\frac{57783}{16000}\right) $&$ -\hspace{-0.05cm}\left(\frac{339}{100},\frac{16779}{4000}\right) $&$
   -\hspace{-0.05cm}\left(\frac{231}{80},\frac{1989}{400}\right) $&$ -\hspace{-0.05cm}\left(\frac{1137}{400},\frac{237}{40}\right) $&$
   -\hspace{-0.05cm}\left(\frac{213}{100},\frac{537}{80}\right) $&$ -\hspace{-0.05cm}\left\{\frac{273}{200},\frac{2919}{400}\right) $&$
   -\hspace{-0.05cm}\left(\frac{201}{200},\frac{759}{100}\right)$ \\[0.1cm]
 $-\hspace{-0.05cm}\left(\frac{741}{160},\frac{1083}{320}\right) $&$ -\hspace{-0.05cm}\left(\frac{7203}{1600},\frac{57087}{16000}\right) $&$
   -\hspace{-0.05cm}\left(\frac{6789}{1600},\frac{62961}{16000}\right) $&$ -\hspace{-0.05cm}\left(\frac{771}{200},\frac{17943}{4000}\right) $&$
   -\hspace{-0.05cm}\left(\frac{267}{80},\frac{261}{50}\right) $&$ -\hspace{-0.05cm}\left(\frac{1419}{400},\frac{609}{100}\right) $&$
   -\hspace{-0.05cm}\left(\frac{1239}{400},\frac{27}{4}\right) $&$ -\hspace{-0.05cm}\left(\frac{921}{400},\frac{2913}{400}\right) $&$
   -\hspace{-0.05cm}\left(\frac{49}{26},\frac{753}{100}\right) $\\[0.1cm]
 $-\hspace{-0.05cm}\left(\frac{39}{8},\frac{57}{16}\right) $&$ -\hspace{-0.05cm}\left(\frac{759}{160},\frac{1197}{320}\right) $&$ -\hspace{-0.05cm}\left(\frac{717}{160},\frac{1311}{320}\right)
   $&$ -\hspace{-0.05cm}\left(\frac{327}{80},\frac{741}{160}\right) $&$ -\hspace{-0.05cm}\left(\frac{57}{16},\frac{171}{32}\right) $&$
   -\hspace{-0.05cm}\left(\frac{783}{200},\frac{123}{20}\right) $&$ -\hspace{-0.05cm}\left(\frac{1419}{400},\frac{339}{50}\right) $&$
   -\hspace{-0.05cm}\left(\frac{167}{60},\frac{507}{70}\right) $&$ -\hspace{-0.05cm}\left(\frac{189}{80},\frac{389}{52}\right)$ \\[0.1cm]
\hline
\hline
\multicolumn{9}{|c|}{$\f{c}_{j_1,j_2}^{(4)}$} \\[0.1cm]
\hline
& & & & & & & & \\[-0.15cm]
$(0,0) $&$ \left(\frac{3}{10},0\right) $&$ \left(\frac{9}{10},0\right) $&$ \left(\frac{9}{5},0\right) $&$ (3,0) $&$ \left(\frac{21}{5},0\right) $&$
   \left(\frac{51}{10},0\right) $&$ \left(\frac{57}{10},0\right) $&$ (6,0) $\\[0.1cm]
 $\left(\frac{3}{32},-\frac{459}{1600}\right) $&$ \left(\frac{249}{640},-\frac{9069}{32000}\right) $&$
   \left(\frac{627}{640},-\frac{8847}{32000}\right) $&$ \left(\frac{597}{320},-\frac{4257}{16000}\right) $&$
   \left(\frac{195}{64},-\frac{807}{3200}\right) $&$ \left(\frac{1353}{320},-\frac{3813}{16000}\right) $&$
   \hspace{0.15cm}  \left(\frac{3273}{640},-\frac{7293}{32000}\right) $&$ \hspace{0.15cm} \left(\frac{3651}{640},-\frac{7071}{32000}\right) $&$ \left(6,-\frac{87}{400}\right)$ \\[0.1cm]
 $\left(\frac{9}{32},-\frac{1377}{1600}\right) $&$ \left(\frac{363}{640},-\frac{27207}{32000}\right) $&$
   \left(\frac{729}{640},-\frac{26541}{32000}\right) $&$ \left(\frac{639}{320},-\frac{12771}{16000}\right) $&$
   \left(\frac{201}{64},-\frac{2421}{3200}\right) $&$ \left(\frac{1371}{320},-\frac{11439}{16000}\right) $&$
   \hspace{0.08cm}  \left(\frac{3291}{640},-\frac{21879}{32000}\right) $&$ \hspace{0.08cm}  \left(\frac{3657}{640},-\frac{21213}{32000}\right) $&$ \left(6,-\frac{261}{400}\right) $\\[0.1cm]
 $\left(\frac{9}{16},-\frac{1377}{800}\right) $&$ \left(\frac{267}{320},-\frac{27207}{16000}\right) $&$
   \left(\frac{441}{320},-\frac{26541}{16000}\right) $&$ \left(\frac{351}{160},-\frac{12771}{8000}\right) $&$
   \left(\frac{105}{32},-\frac{2421}{1600}\right) $&$ \left(\frac{699}{160},-\frac{11439}{8000}\right) $&$
   \left(\frac{1659}{320},-\frac{21879}{16000}\right) $&$ \left(\frac{1833}{320},-\frac{21213}{16000}\right) $&$ \left(6,-\frac{261}{200}\right) $\\[0.1cm]
 $\left(\frac{15}{16},-\frac{459}{160}\right) $&$ \left(\frac{381}{320},-\frac{9069}{3200}\right) $&$
   \left(\frac{543}{320},-\frac{8847}{3200}\right) $&$ \left(\frac{393}{160},-\frac{4257}{1600}\right) $&$
   \left(\frac{111}{32},-\frac{807}{320}\right) $&$ \left(\frac{717}{160},-\frac{3813}{1600}\right) $&$
   \left(\frac{1677}{320},-\frac{7293}{3200}\right) $&$ \left(\frac{1839}{320},-\frac{7071}{3200}\right) $&$ \left(6,-\frac{87}{40}\right) $\\[0.1cm]
 $\left(\frac{21}{16},-\frac{3213}{800}\right) $&$ \left(\frac{99}{64},-\frac{63483}{16000}\right) $&$
   \left(\frac{129}{64},-\frac{61929}{16000}\right) $&$ \left(\frac{87}{32},-\frac{29799}{8000}\right) $&$
   \left(\frac{117}{32},-\frac{5649}{1600}\right) $&$ \left(\frac{1911}{400},-\frac{1389}{400}\right) $&$
   \left(\frac{144}{25},-\frac{123}{40}\right) $&$ \left(\frac{321}{50},-\frac{1113}{400}\right) $&$ \hspace{-0.15cm}\left(\frac{2703}{400},-\frac{261}{100}\right)$
   \\[0.1cm]
 $\left(\frac{51}{32},-\frac{7803}{1600}\right) $&$ \hspace{0.05cm} \left(\frac{1161}{640},-\frac{154173}{32000}\right) $&$
   \hspace{0.12cm}  \left(\frac{1443}{640},-\frac{150399}{32000}\right) $&$ \hspace{0.12cm} \left(\frac{933}{320},-\frac{72369}{16000}\right) $&$
   \hspace{0.1cm}  \left(\frac{243}{64},-\frac{13719}{3200}\right) $&$ \left(\frac{1881}{400},-\frac{363}{80}\right) $&$
   \left(\frac{567}{100},-\frac{411}{100}\right) $&$ 
\hspace{-0.2cm}\left(\frac{2577}{400},-\frac{357}{100}\right) $&$
   \hspace{-0.1cm} \left(\frac{2733}{400},-\frac{1329}{400}\right)\; $\\[0.1cm]
 $\left(\frac{57}{32},-\frac{8721}{1600}\right) $&$ \left(\frac{255}{128},-\frac{172311}{32000}\right) $&$
   \left(\frac{309}{128},-\frac{168093}{32000}\right) $&$ \left(\frac{195}{64},-\frac{80883}{16000}\right) $&$
   \left(\frac{249}{64},-\frac{15333}{3200}\right) $&$ \left(\frac{1857}{400},-\frac{21}{4}\right) $&$
   \left(\frac{2157}{400},-\frac{2013}{400}\right) $&$ \left(\frac{1227}{200},-\frac{891}{200}\right) $&$
   \left(\frac{13}{2},-\frac{1653}{400}\right) $\\[0.1cm]
 $\left(\frac{15}{8},-\frac{459}{80}\right) $&$ \left(\frac{333}{160},-\frac{9069}{1600}\right) $&$ \left(\frac{399}{160},-\frac{8847}{1600}\right)
   $&$ \left(\frac{249}{80},-\frac{4257}{800}\right) $&$ \left(\frac{63}{16},-\frac{807}{160}\right) $&$ \left(\frac{459}{100},-\frac{45}{8}\right) $&$
   \left(\frac{423}{80},-\frac{2187}{400}\right) $&$ \left(\frac{381}{64},-\frac{137}{28}\right) $&$ \left(\frac{63}{10},-\frac{201}{44}\right)$ \\[0.1cm]
\hline
\end{tabular}
\caption{Control points~$\ab{c}_{j_1,j_2}^{(i)}$
of the 
bilinear-like $G^4$ five-patch spline geometry \eqref{eq:five_patch_domain_bilinearLike} shown in Fig.~\ref{fig:bilinearLikeDomains} (right). 
}
\label{five_patch_domain_bilinearLikeTable} 
\end{table}

%

\end{document}